\documentclass{article}

\usepackage{arxiv}

\usepackage{multirow}
\usepackage{rotating}
\usepackage[utf8]{inputenc} 
\usepackage[T1]{fontenc}    
\usepackage{hyperref}       
\usepackage{url}            
\usepackage{booktabs}       
\usepackage{amsfonts}       
\usepackage{nicefrac}       
\usepackage{microtype}      
\usepackage{lipsum}		
\usepackage{graphicx}
\usepackage{natbib}
\usepackage{doi}
\usepackage{amsmath}
\usepackage{subcaption}
\usepackage{wrapfig}
\usepackage{ragged2e}
\usepackage[ruled,vlined]{algorithm2e}
\usepackage{tabularx}
\usepackage[section]{placeins}

\usepackage[table]{xcolor}
\usepackage[english]{babel}
\usepackage{amsthm}
\usepackage{mathtools}
\usepackage[ruled,vlined]{algorithm2e}
\usepackage{algorithmic}
\SetKwRepeat{Do}{do}{while}

\theoremstyle{plain}
\newtheorem{theorem}{Theorem}[section]
\newtheorem{definition}[theorem]{Definition}
\newtheorem{corollary}[theorem]{Corollary}
\newtheorem{assumption}[theorem]{Assumption}
\newtheorem{lemma}[theorem]{Lemma}
\newtheorem{proposition}[theorem]{Proposition}
\newtheorem{remark}[theorem]{Remark}
\newtheorem{notation}[theorem]{Notation}

\theoremstyle{definition}
\newtheorem{examplex}{Example}
\newenvironment{example}
  {\pushQED{\qed}\examplex}
  {\popQED\endexamplex}

\newcommand{\norm}[1]{\left\lVert#1\right\rVert}

\DeclarePairedDelimiterX{\inp}[2]{\langle}{\rangle}{#1, #2}

\newcommand\restr[2]{{
  \left.\kern-\nulldelimiterspace 
  #1 
  \vphantom{\big|} 
  \right|_{#2} 
  }}
  
\title{Time-dependent Zermelo navigation with tacking}

\author{Steen Markvorsen \\
	DTU Compute, Mathematics \\
	Technical University of Denmark \\
	DK-2800 Kgs. Lyngby, Denmark \\
	\texttt{stema@dtu.dk} \\
	\And
	Enrique Pend\'{a}s-Recondo \\
	Departamento de Matem\'{a}ticas \\
	Universidad de Oviedo \\
    33203 Gij\'{o}n, Asturias, Spain \\
	\texttt{pendasenrique@uniovi.es} \\
    \And
	Frederik M\"{o}bius Rygaard \\
	DTU Compute, Mathematics \\
	Technical University of Denmark \\
	DK-2800 Kgs. Lyngby, Denmark \\
	\texttt{fmry@dtu.dk} \\
}

\date{}

\hypersetup{
pdftitle={Time-dependent Zermelo navigation with tacking},
pdfsubject={Zermelo Navigation with Tacking},
pdfauthor={},
pdfkeywords={},
} 

\begin{document}

\maketitle

\begin{abstract}
    We address the time- and position-dependent Zermelo navigation problem within the framework of Lorentz-Finsler geometry. Since the initial work of E. Zermelo, the task is to find the time-minimizing trajectory between two regions for a moving object whose speed profile depends on time, position and direction. We give a step-by-step review of the classical formulation of the problem, where the geometric shape generated by the velocity vectors---the speed profile indicatrix---is strongly convex at each point. We derive new global results for the cases where the indicatrix field is only time-dependent. In such (meso-scale realistic) cases, Zermelo navigation exhibits particularly favorable properties that have not been previously explored, making them especially appealing for both theoretical and numerical investigations. Moreover, motivated by real-world phenomena and examples, we obtain novel results for non-convex (or multi-convex) navigation, i.e. when the indicatrices fail to be convex. In this new setting---which is not unlike the corresponding Finsler setting for Snell's law---optimal paths may involve so-called tacking, which stems from discontinuous shifts of direction of motion. The tacking behaviour thus results in zig-zag trajectories, as observed in time-optimal sailboat navigation and surprisingly also in the flight paths of far-ranging seabirds. Finally, we provide new efficient computational algorithms and illustrate the use of them to solve the Zermelo navigation (boundary value) problem, i.e. to find numerically the time-minimizing trajectory between two fixed points in the general non-convex setting.
\end{abstract}


\section{Introduction}
We consider a smooth environment (represented by a Riemannian manifold) in which an object (represented by a point) is moving freely in all directions---but everywhere and at all times with a given non-zero speed depending smoothly on direction, position, and time. This dependency, i.e. such a speed profile constraint, appears naturally in a multitude of physical settings, for example already in (mountain-) biking: the maximal possible speed upward (or upwind) is usually much smaller than the downward (or downwind) speed, even when you put the same maximal amount of energy into the pedals. In the literature there is an abundance of strongly motivating appearances of similar settings, e.g. concerning applications to such diverse phenomena as sailboat maneuvering, micro-swimming, seismic waves, wildfires, and long distance flying in the wind. However, in this paper we will mainly focus on the obvious general optimization problem: How to use the given speed profile to get from $A$ to $B$ on the manifold in the shortest amount of time.

The many possible future applications of the new results that we present in this work will so far only be indicated by way of references to recent pertinent papers, which are concerned with specific real-world phenomena like the ones alluded to above.

In the following few paragraphs, we now outline the tools and concepts that we use, as well as the main results of the paper:

\paragraph{Classical Zermelo navigation} Zermelo navigation is then---still in general terms---the problem of finding the fastest trajectory (using maximum speed) between two fixed points in an ambient space where the speed is anisotropic in the above sense (i.e. in particular also direction-dependent). This problem was first formulated by Ernst Zermelo in his much cited 1931 paper \cite{Z1931}, in which he solved the two-dimensional case using calculus of variations. Shortly thereafter, the problem attracted the attention of renown mathematicians such as Levi-Civita, Mani\`{a}, and von Mises (see \cite{levi-civita1931a,mania1937,mises1931a}), and it quickly became a classical topic in the calculus of variations (see the monographs by Carath\'{e}odory, \cite{Caratheodory1965,Caratheodory1967}). More recently, it has also been studied and solved using tools from optimal control theory (see e.g. \cite{biferale2019a,bonnard2023a,piro2022a,piro2021a,serres2009}). In this work, however, we will approach this problem from a more geometric perspective, using the framework of {\em Finsler geometry}.

\paragraph{Finsler geometry} From a geometric standpoint, the key element that defines Zermelo navigation is the {\em speed profile indicatrix} of the moving object---namely, the geometric shape generated at each point by the set of velocity vectors the object can achieve in all directions. If the speed is isotropic---the same in every direction---this indicatrix is a centered sphere at each point (a circle in dimension $2$), which can be interpreted as the unit vectors of a standard Riemannian metric. When the speed is affected by an external current or wind and becomes anisotropic, the indicatrix can instead be seen (under some assumptions) as the unit vectors of a more general type of metric known as a {\em Finsler metric}---essentially, a non-symmetric norm. In this setting, it turns out that the Finsler metric directly measures the travel time: the Finslerian length of a curve coincides with the time required for the object to traverse that trajectory at maximum speed. This means that the {\em time-minimizing} trajectories---the solutions to Zermelo's problem---are Finsler (pre-) geodesics. Thus, Finsler geometry provides a natural framework to model the anisotropy introduced in Zermelo navigation. This connection was first observed in \cite{BaoRob2,shen2003a} and since then, there has been a growing interest to model this problem from a Finslerian perspective (see \cite{caponio2021a,JPS2023} and references therein). In fact, Finsler geometry has proven to be a powerful tool for modeling general physical phenomena that exhibit an intrinsic anisotropy, such as the spread of wildfires (see e.g. \cite{Javaloyes2023a,SM2016}) or the propagation of seismic waves (see e.g. \cite{antonelli2003,yajima2009}).

\paragraph{Time-dependence} When the speed of the moving object depends not only on position and direction but also on time, the Finsler metric that models its travel through space becomes time-dependent, and its geodesics can no longer be defined in the usual way. This requires working at the level of a {\em spacetime}, where time is introduced as an explicit additional dimension. In this setting, the travel of the object through the spacetime is modeled by a {\em Lorentz-Finsler metric}, which generalizes Lorentzian metrics in a similar way that Finsler metrics generalize Riemannian ones. An analogous conclusion then follows: the solutions to the time-dependent Zermelo's problem are Lorentz–Finsler (pre-) geodesics. The use of Lorentz–Finsler geometry to extend Zermelo navigation to the time-dependent setting was first developed in \cite{JS2020,javaloyes2021a}.

Throughout this work, we will pay special attention to the time-only dependent case, where the Finsler metric depends on time but not on position. This is not only a reasonable assumption in certain real-world scenarios, but also, as already mentioned in the abstract, and as we will show explicitly, a case in which Zermelo navigation exhibits particularly favorable properties that have not been previously explored, making it especially appealing for both theoretical and numerical investigations.

\paragraph{Non-convex navigation with tacking} There is a key condition that the speed profile indicatrix must satisfy in order to properly define a (time-dependent) Finsler metric: it must be {\em strongly convex}, meaning that it has positive curvature everywhere. Without this condition, the indicatrix no longer defines a proper Finsler metric, and consequently, the solution to Zermelo's problem can no longer be interpreted as a (pre-) geodesic curve. For this reason, to the best of our knowledge, this assumption has always been made in the Finslerian formulation of the problem.

However, there are real-world scenarios in which this hypothesis fails, such as the sailboat navigation in the presence of wind (see e.g. \cite[Part~III, Chapter~2]{marchaj1982a} and \cite{buell1996a, pueschl2018a}) or the flight of certain far-ranging seabirds (see the recent survey by Thorne et al., \cite[Section~4.5 and the 12 open questions in Section~7]{thorne2023a}, and \cite{Clay2023,ventura2022a,ventura2020a}). In particular, albatrosses have received a great deal of attention, see \cite{Richardson2011,Richardson2018,sachs2005a,sachs2016a,sachs2013a}. In such cases, {\em tacking} becomes an effective strategy to reduce travel time, resulting in zig-zag trajectories as optimal paths. Amazingly, this tacking, this dynamic soaring by birds in wind shear, was already observed by Leonardo da Vinci, see \cite{Richardson2017,Richardson2019}, and in fact also commented upon by Peal \cite{peal1880a} and Rayleigh \cite{rayleigh1883a} in the beginning of the 1880s. Motivated by these situations, the main goal of this paper is to extend the classical framework to the {\em non-convex Zermelo navigation}, where the speed profile indicatrix may fail to be convex. This has been previously addressed within the framework of optimal control theory (see e.g. \cite{hays2014a,mclaren2014a,pinti2020a,Pokhrel2023,techy2011a}), but never from a purely geometric perspective.

\paragraph{Our approach and main results} We begin in {\bf{Section~\ref{sec:review}}} with a complete review of the classical version of Zermelo navigation, where the goal is to minimize the travel time between two regions or points $A,B$ in $N=\mathbb{R}^n$. Step by step, we first present the mathematical modeling of all relevant elements in Section~\ref{subsec:setting}, followed in Section~\ref{subsec:finsler_metrics} by how they give rise to a (time-dependent) Finsler metric when the speed profile indicatrix is strongly convex. In the time-independent case, the Finsler metric is enough to solve Zermelo's problem, as explained in Section~\ref{subsec:time-independent}. To address the time-dependence, we introduce Lorentz-Finsler metrics in Section~\ref{subsec:lorentz-finsler}, leading to the general solution (Theorem~\ref{thm:sol_zermelo}) in Section~\ref{subsec:solution}. Then, in Section~\ref{subsec:wavemap} we introduce the {\em wavemap}, a useful tool for finding explicit solutions to Zermelo's problem. Essentially, the wavemap encodes all the time-minimizing trajectories from the initial region. The endpoints of these trajectories form the {\em wavefront} at each time, representing the spread of the moving object through space if it were to depart in all directions simultaneously. We conclude the section by stating the key results concerning the wavemap (Propositions~\ref{prop:strictly_min}, \ref{prop:cut_points} and \ref{prop:cut_function}) that will be used in the following sections.

In {\bf{Section~\ref{sec:time_only}}}, we focus on the time-only dependent case. As shown in Section~\ref{subsec:time_only_general}, the wavemap exhibits exceptionally favorable behavior in this setting: as long as the initial region is strictly convex, there are no {\em cut points} (Theorem~\ref{thm:no_cut_points})---that is, no intersections between different time-minimizing trajectories---and the wavefront remains strictly convex at all times (Theorem~\ref{thm:convexity}). These advantages are then illustrated through explicit examples in the subsequent subsections. First, in Section~\ref{subsec:ellipses}, we analyze the (well-known) two-dimensional case where the indicatrix field consists of displaced ellipses, showing that, in the time-only dependent setting, the wavemap can even be computed analytically. Second, in Section~\ref{subsec:circles}, we further simplify the setting to the special case of shifted circles, which allows for reduced computational complexity.

In {\bf{Section~\ref{sec:non-convex}}}, we begin our analysis of the non-convex navigation, initially restricting our attention to the constant case---where the speed is both time- and position-independent---and assuming that $A$ and $B$ are single points. After introducing some preliminary definitions in Section~\ref{subsec:multi-convex} to generalize the notion of a \emph{norm}, we address the challenge posed by a non-convex indicatrix by introducing the concept of a \emph{multi-convex indicatrix}, defined as a union of patches from multiple strictly convex indicatrices. Then, Theorem~\ref{thm:time_min} characterizes all (possibly infinitely many) time-minimizing trajectories in this setting. In certain cases, the optimal paths can only be realized through tacking, resulting in zig-zag trajectories, as observed in the real-world scenarios mentioned above.

In {\bf{Section~\ref{sec:restricted}}}, we pursue two main goals: first, to impose certain restrictions on the non-convex navigation in order to guarantee the existence of a unique time-minimizing trajectory; and second, to slightly reformulate the problem and its solution method to facilitate its extension to the non-constant case. To this end, we introduce the \emph{restricted non-convex navigation}, in which only one tack point is allowed, and instead of explicitly using a multi-convex indicatrix, the travel is modeled using two proper Finsler metrics---one before the tacking and another afterwards. In Section~\ref{subsec:restricted_constant}, we prove that under these restrictions, there exists a unique optimal solution in the constant case (Theorem~\ref{thm:existence} and Corollary~\ref{cor:uniqueness}), which in fact coincides with the one Theorem~\ref{thm:time_min} would provide under the same conditions (Theorem~\ref{thm:uniqueness}). Moreover, the advantage of this new formulation lies in its reliance on the strict convexity of the wavefronts generated by the wavemaps in the constant case. Since, as previously shown in Section~\ref{sec:time_only}, this property also holds in the time-only dependent case, we are able to extend most of the results to that setting in Section~\ref{subsec:restricted_time}, provided that the initial and final regions are strictly convex. However, in the general non-constant case, the uniqueness of the optimal solution is generally lost, as demonstrated by the counterexamples in Section~\ref{subsec:counter_examples}. We conclude in Section~\ref{subsec:snell} by establishing the connection between the tacking maneuvering and Snell's law of refraction---that is, the change in direction experienced by a wave as it crosses the interface between two different media---which arises from the fact that wave rays also seek to minimize the travel time, as stated by Fermat's principle. Snell's law has been studied and generalized within the framework of Finsler geometry in \cite{MP2023}.

In {\bf{Section~\ref{sec:computations}}}, we present new computational algorithms to find the solution to the general non-convex Zermelo navigation, allowing for an arbitrary number of tack points. Since these algorithms represent an entirely novel method for solving the pregeodesic boundary value problem in a general multi-convex Lorentz-Finsler manifold with any number of allowed tack points, we display all the details of the algorithms, including their pseudo-codes and associated proofs of convergence. Specifically, in Section~\ref{subsec:pregeodesics_estimation} we introduce Algorithm~\ref{al:georceh} to compute Lorentz-Finsler pregeodesics between two fixed points, proving that it exhibits global convergence (Proposition~\ref{prop:global_convergence}) as well as local quadratic convergence (Proposition~\ref{prop:quad_conv}). Then, in Section~\ref{subsec:tack_points_estimation} we present Algorithm~\ref{al:tacking_optimization}, which minimizes the total travel time when tacking is allowed and different time- and position-dependent Finsler metrics can be applied to travel between the tack points.

Finally, in {\bf{Section~\ref{sec:examples}}} we present several examples in which the non-convex (bi-metric) navigation problem is solved using the algorithms introduced in Section~\ref{sec:computations}. We consider various scenarios based on elliptic indicatrices, as described in Section~\ref{subsec:ellipses}, ranging from constant metrics to general time- and position-dependent Finsler metrics. In the last example, we introduce a conformal position-only dependent Finsler bi-metric, which could serve as an initial approximation for modeling real-world scenarios such as e.g. the long distance flights (including tacking) of the Juan Fern\'{a}ndez petrels, that have been recently documented by Clay et al. in \cite{Clay2023}. Concerning this particular application, we must also refer to the thorough survey by Thorne et al. \cite{thorne2023a} and to the work by Sachs \cite{sachs2016a}, both of which give rise to a number of further interesting modeling questions, that relate directly to the non-convex Zermelo navigation problem for multi-metric flights. 

For the convenience of the reader, Table~\ref{tab:choices} summarizes the different assumptions that define the specific settings and versions of Zermelo navigation considered in each section:

\begin{table}[ht]
    \centering
    \scriptsize
    \begin{tabular*}{\textwidth}{@{\extracolsep\fill}lccccc}
        \midrule\midrule
        \textbf{Sections} & \textbf{Dimension} & $\boldsymbol{A},\boldsymbol{B}$ & \textbf{Indicatrix} & \textbf{Dependency} & \textbf{Tacking}\\
        \midrule\midrule
        Section~\ref{sec:review} & $n$ & Hypersurfaces & One Finsler indicatrix $\Sigma_{(t,x)}$ & General & No \\
        \midrule
        Section~\ref{subsec:time_only_general} & $n$ & Hypersurfaces & One Finsler indicatrix $\Sigma_{(t,x)}$ & Time-only & No \\
        \midrule
        Section~\ref{subsec:ellipses} & $2$ & Curves & Shifted ellipses & General (eventually time-only) & No \\
        \midrule
        Section~\ref{subsec:circles} & $2$ & Curves & Shifted circles & General (eventually time-only) & No \\
        \midrule
        Section~\ref{sec:non-convex} & $n$ & Points & Multi-convex indicatrix $\Sigma$ & Constant & $n_{\mathrm{tacks}}$ \\
        \midrule
        Section~\ref{subsec:restricted_constant} & $n$ & Points & Two Finsler indicatrices $\Sigma^{\alpha}, \Sigma^{\beta}$ & Constant & One tack point \\
        \midrule
        Section~\ref{subsec:restricted_time} & $n$ & Hypersurfaces & Two Finsler indicatrices $\Sigma^{\alpha}_t, \Sigma^{\beta}_t$ & Time-only & One tack point \\
        \midrule
        Section~\ref{subsec:counter_examples} & $2$ & Points & Two Riemannian indicatrices & Time-only and position-only & One tack point \\
        \midrule
        Section~\ref{subsec:snell} & $n$ & Hypersurfaces & Two Finsler indicatrices & General & One tack point \\
        \midrule
        Section~\ref{subsec:pregeodesics_estimation} & $n$ & Points & One Finsler indicatrix & General & No \\
        \midrule
        Section~\ref{subsec:tack_points_estimation} & $n$ & Points & $n_{\mathrm{tacks}}+1$ Finsler indicatrices & General & $n_{\mathrm{tacks}}$ \\
        \midrule
        Section~\ref{sec:examples} & $2$ & Points & Two Finsler indicatrices (shifted ellipses and circles) & General & $n_{\mathrm{tacks}}$ \\
        \midrule\midrule
    \end{tabular*}
    \caption{Specific assumptions used for Zermelo navigation across the different sections of the paper. See explanations for the column headings below.}
    \label{tab:choices}
\end{table}

\begin{itemize}
\item \textbf{Dimension}: All results are valid for a space $N=\mathbb{R}^n$ of general dimension $n \geq 2$, although we set $n=2$ when computing specific examples.
\item   $\boldsymbol{A},\boldsymbol{B}$: The initial region $A$ and the arrival region $B$ can be either hypersurfaces of $N$ (curves in dimension $2$) or single points. Results established for hypersurfaces can be trivially particularized to single points (but not vice versa).
\item    \textbf{Indicatrix}: The speed profile indicatrix generated by the velocity vectors of the moving object. We may use a single Finsler (i.e., strongly convex) indicatrix---such as ellipses or circles in specific examples---or multiple Finsler indicatrices (in the non-convex navigation).
\item     \textbf{Dependency}: ``General'' indicates that the Finsler metrics considered may be both time- and position-dependent. In some sections, we focus specifically on the time-only dependent case, while in others we simplify the study to the constant case (i.e., time- and position-independent).
\item     \textbf{Tacking}: Indicates whether tacking is an effective strategy for solving Zermelo's problem, and if so, the number of tack points allowed. $n_{\mathrm{tacks}}$ means there is no restriction on the number of tack points.
\end{itemize}

\section{Review on Zermelo navigation}
\label{sec:review}

Throughout this section, we present a step-by-step review of how to model and solve the classical Zermelo navigation using Finsler and Lorentz-Finsler geometry, highlighting the most important results in the general time- and position-dependent case. Our exposition mainly follows \cite{javaloyes2021a,EPR2024}. This will not only serve as an introduction to the problem but also provide the foundation upon which the following sections are built.

\subsection{The setting}
\label{subsec:setting}
In the classical version of Zermelo's problem, a moving object ``navigates'' between two {\em regions} in {\em space} with a certain anisotropic {\em speed profile}, and the goal is to find the trajectory that minimizes the {\em travel time}. So, in order to model this problem mathematically, first we need to specify all the elements in play:
\begin{itemize}
    \item The {\em space} will be assumed to be a smooth manifold $N$ of dimension $n \geq 2$ (for real-world scenarios, $n=2,3$). Since explicit computations will be needed, for simplicity we will assume $N=\mathbb{R}^n$ and work in natural coordinates throughout the text, i.e. points will be written as $x=(x^1,\ldots,x^n)\in N$ (we will use capital letters for important fixed points), tangent vectors as $v=(v^1,\ldots,v^n)\in T_xN$ and curves as $\gamma(s)=(x^1(s),\ldots,x^n(s))$.\footnote{Throughout this work, any {\em curve} will be considered regular and piecewise smooth, unless otherwise stated.} When needed, $N=\mathbb{R}^n$ will be seen as an affine space, in which case we will explicitly write $\mathbb{R}^n$ instead of $N$, and free vectors (translations) will be written as $v \in \mathbb{R}^n$.
    
    \item The {\em initial region} $A$ and the {\em arrival region} $B$ will be either single points in $N$ or (more generally) compact embedded hypersurfaces of $N$.\footnote{The case where the initial region is a submanifold of arbitrary codimension is studied in \cite{javaloyes2021a}.} In the latter case, obviously, we will assume that the regions do not intersect.
    
    \item {\em Distances} on $N$ are measured using a Riemannian metric $\langle\cdot,\cdot\rangle$ on $N$ (i.e. a scalar product on each tangent space $T_xN$). Namely, if $\gamma(s)$ is a curve in $N$ and $\gamma'(s)$ denotes the derivative with respect to the parameter $s \in [a,b]$, the distance traveled when following $\gamma$ is given by the length of the curve computed with $\langle\cdot,\cdot\rangle$:
    \begin{equation*}
        \int_a^b \sqrt{\langle\gamma'(s),\gamma'(s)\rangle} \ \mathrm{d}s=\int_a^b ||\gamma'(s)|| \ \mathrm{d}s,
    \end{equation*}
    where $||\cdot|| \coloneqq \sqrt{\langle\cdot,\cdot\rangle}$ is the norm induced by $\langle \cdot,\cdot \rangle$. This distance is independent of the parametrization of the curve. The simplest example is the standard Euclidean metric:
    \begin{equation}
    \label{eq:euclidean}
        \langle v,u \rangle=v^1u^1+\ldots+v^nu^n, \quad ||v||=\sqrt{(v^1)^2+\ldots(v^n)^2}, \quad \forall v,u \in T_xN,
    \end{equation}
    which is independent of the position $x\in N$, so this is the natural choice when the space $N$ is flat.
    
    \item In order to measure {\em time}, we can include an additional dimension to the space and define the {\em spacetime} $M=\mathbb{R} \times N$, where the natural projection (time coordinate) $t: M \rightarrow \mathbb{R}$ represents the global absolute (non-relativistic) time, i.e. the time that every observer measures. Points in the spacetime will be written as $q=(t,x)=(x^0,x^1,\ldots,x^n) \in M$, tangent vectors as $\hat{v}=(v^0,v)=(v^0,v^1,\ldots,v^n) \in T_qM$, curves as $\hat{\gamma}(s)=(t(s),\gamma(s))=(x^0(s),x^1(s),\ldots,x^n(s))$ and constant time slices as $\{t=t_0\}\coloneqq \{t_0\}\times N$. Without loss of generality, we will always assume that $t=0$ is the {\em initial time}, i.e. the time when the moving object departs from $A$.

    \item The (anisotropic) {\em speed profile} of the moving object is given, in the general time- and position-dependent case, by a positive smooth function $V(t,x,v)$ representing the maximum speed by which the object can move in direction $v \in T_{x}N$ at each time $t\in\mathbb{R}$ and each point $x\in N$.\footnote{We will assume that the object can travel in every direction. The case where some directions are restricted due to a ``strong wind'' is addressed in detail in \cite{caponio2021a,javaloyes2021a}.} In particular, $V(t,x,v)$ must depend only on the orientation of $v$, not its length, so $V(t,x,\lambda v)=V(t,x,v)$ for all $\lambda>0$.
    
    \item The vectors whose norm (computed with $\langle\cdot,\cdot\rangle$) coincides with the corresponding maximal speed, are assumed throughout to represent the actual allowed {\em non-zero velocity vectors} of the moving object. Namely, at each $(t,x)\in M$, the set of velocity vectors (a.k.a. the {\em speed profile indicatrix}) is
    \begin{equation}
    \label{eq:vel_vectors}
        \Sigma_{(t,x)} \coloneqq \{v\in T_xN: ||v||=V(t,x,v) > 0\}.
    \end{equation}
    Physically, $\Sigma_{(t,x)}$ can be interpreted as the ``infinitesimal'' travel of the object from $p$ in all directions. In other words, $x+\Sigma_{(t,x)}\in \mathbb{R}^n$ would represent all the points that the moving object can reach after one time unit, if it departed from $x\in \mathbb{R}^n$ at time $t\in \mathbb{R}$ under constant conditions.
\end{itemize}

The final ingredient we need to properly establish Zermelo navigation is the computation of the {\em travel time}. Let $\hat{\gamma}(s)=(t(s),\gamma(s))$ be a curve in the spacetime $M$. For $\hat{\gamma}$ to represent an actual trajectory of the moving object, the time $t(s)$ and position $\gamma(s)$ cannot be independent to each other; they must instead be related by the speed profile $V$. Specifically, the velocity of the curve---the derivative of the position $\gamma(s)$ with respect to the time coordinate $t$---must be an actual velocity vector of the moving object. Namely, if we consider a time reparametrization $s(t)$ and the corresponding $t$-parametrized curve $\gamma\circ s(t)=\gamma(s(t))$, then
\begin{equation*}
    V(t(s),\gamma(s),\gamma'(s))=\left\lvert\left\lvert \frac{d(\gamma \circ s)}{dt}(t(s)) \right\rvert\right\rvert = s'(t(s)) ||\gamma'(s)|| = \frac{||\gamma'(s)||}{t'(s)},
\end{equation*}
so we obtain
\begin{equation}
\label{eq:dot_t}
    t'(s) = \frac{||\gamma'(s)||}{V(t(s),\gamma(s),\gamma'(s))}.
\end{equation}
In conclusion, $\hat{\gamma}(s)=(t(s),\gamma(s))$ represents a $V$-controlled trajectory of the moving object if and only if \eqref{eq:dot_t} holds. In this case, the first component $t(s)$ directly measures the time spent during the travel, and the total {\em travel time} can be computed by integrating \eqref{eq:dot_t}:
\begin{equation*}
    \mathcal{T}[\hat{\gamma}]\coloneqq t(b)=\int_a^b \frac{||\gamma'(s)||}{V(t(s),\gamma(s),\gamma'(s))} \ \mathrm{d}s,
\end{equation*}
where $t(a)=0$ is chosen as the initial time. This is independent of the parametrization of $\hat{\gamma}$.

In this setting, Zermelo's problem can be formulated as follows:
\begin{quote}
Given $M=\mathbb{R}\times N$, a Riemannian metric $\langle\cdot,\cdot\rangle$ on $N$, a speed profile $V$ and initial and arrival regions $A, B \subset N$, find the curve $\hat{\gamma}(s)=(t(s),\gamma(s))$, $s \in [a,b]$, with the conditions $t(a)=0$, $\gamma(a)\in A$ and $\gamma(b)\in B$, such that \eqref{eq:dot_t} holds and the travel time $\mathcal{T}[\hat{\gamma}]=t(b)$ is minimum.
\end{quote}

\subsection{Finsler metrics}
\label{subsec:finsler_metrics}
Let us take a closer look at \eqref{eq:dot_t}. The right-hand side of the equation defines a time-dependent function on $TN$:
\begin{equation}
\label{eq:F_t}
    F_t(v) \equiv F(t,x,v)\coloneqq \frac{||v||}{V(t,x,v)}, \qquad \forall t\in \mathbb{R}, \quad \forall x\in N, \quad  \forall v \in T_xN,
\end{equation}
with the following properties:
\begin{itemize}
    \item[(i)] $F_t$ is continuous everywhere (with $F_t(0)\coloneqq 0$) and smooth away from the zero vector.
    \item[(ii)] $F_t$ is positive: $F_t(v)>0$ for all $v\in T_xN \setminus \{0\}$.
    \item[(iii)] $F_t$ is positive homogeneous of degree 1: $F_t(\lambda v)=\lambda F(v)$ for all $\lambda>0$ and all $v\in T_xN$.
\end{itemize}
This tells us that $F_t$ is essentially a time-dependent (non-symmetric) norm at each tangent space $T_xN$.\footnote{$F_t$ is non-symmetric in the sense that $F_t(v) \not=F_t(-v)$ in general.} Associated with this norm, we define at each $(t,x) \in M$ the {\em indicatrix} $\Sigma_{(t,x)}$ of $F_t$ as the $F_t$-unit vectors, which coincides with the set of velocity vectors \eqref{eq:vel_vectors} of the moving object:
\begin{equation*}
    \Sigma_{(t,x)} = \{v\in T_xN: F_t(v)=1\}.
\end{equation*}

\begin{remark}
\label{rem:indicatrix}
    Under the assumptions we have established, $\Sigma_{(t,x)}$ is always a hypersurface of $T_xN$ diffeomorphic to a sphere that encloses the zero vector (see \cite[Proposition~2.6]{JS2014}). Moreover, $\Sigma \coloneqq \cup_{q\in M}\Sigma_q$ contains all the information about $F_t$. Indeed, given $\Sigma_{(t,x)}$ at each $(t,x)\in M$, we can reconstruct $F_t$ as $F_t(v)=\lambda$, where $\lambda$ is the unique positive number such that $\frac{v}{\lambda}\in \Sigma_{(t,x)}$ (see \cite[Theorem~2.14]{JS2014}).
\end{remark}

When one has a norm, the way to obtain a corresponding scalar product is through the Hessian. For instance, using the particular expression in \eqref{eq:euclidean} for the Euclidean norm, the Hessian of $\frac{1}{2}||\cdot||^2$ provides back the Euclidean metric $\langle \cdot,\cdot \rangle$. With this in mind, we define the {\em fundamental tensor} of $F_t$ in the direction $v\in T_xN\setminus\{0\}$ as the symmetric bilinear form given by the Hessian of $\frac{1}{2}F_t^2$ at $v$:
\begin{equation}
\label{eq:fund_tensor_F}
g^{F_t}_v(u,w) \coloneqq \left. \textup{Hess}\left(\frac{1}{2} F_t^2\right) \right\rvert_v(u,w) = \frac{1}{2} \left. \frac{\partial^2}{\partial \delta \partial \eta} F_t(v + \delta u + \eta w)^2 \right\rvert_{\delta=\eta=0}, \quad \forall u,w \in T_xN.
\end{equation}
It is a straightforward computation to check that
\begin{equation*}
    g^{F_t}_v(v,u)=\frac{1}{2} \left. \frac{\partial}{\partial \delta} F_t(v + \delta u)^2 \right\rvert_{\delta=0}, \qquad g^{F_t}_v(v,v)=F_t(v)^2.
\end{equation*}
In matrix form, using natural Euclidean coordinates---so that $F_t(v) \equiv F(t,x^1,\ldots,x^n,v^1,\ldots,v^n)$:
\begin{equation*}
g^{F_t}_v(u,w) =
\begin{pmatrix}
u^1 & \cdots & u^n
\end{pmatrix}
\begin{pmatrix}
\frac{1}{2}\frac{\partial^2 F_t^2}{\partial(v^1)^2}(v) & \cdots & \frac{1}{2}\frac{\partial^2 F_t^2}{\partial v^1 \partial v^n}(v) \\
\vdots & \ddots & \vdots \\
\frac{1}{2}\frac{\partial^2 F_t^2}{\partial v^n \partial v^1}(v) & \cdots & \frac{1}{2}\frac{\partial^2 F_t^2}{\partial(v^n)^2}(v)
\end{pmatrix}
\begin{pmatrix}
w^1 \\
\vdots \\
w^n
\end{pmatrix},
\end{equation*}
where
\begin{equation}
\label{eq:matrix_gF}
    g^{F_t}_{ij}(v) \coloneqq \frac{1}{2}\frac{\partial^2 F_t^2}{\partial v^i \partial v^j}(v), \quad i,j=1,\ldots,n,
\end{equation}
are the coefficients of the coordinate matrix of $g^{F_t}_v$.

However, for $g^{F_t}_v(\cdot,\cdot)$ to be a proper scalar product at each tangent space $T_xN$ and for each direction $v$, we must require it to be positive definite, i.e. $g^{F_t}_v(u,u)>0$ for all $u \in T_xN \setminus \{0\}$. This is equivalent to the following condition on the indicatrix (see \cite[Proposition~2.3(v)]{JS2014}):
\begin{itemize}
    \item[(iv)] $\Sigma_{(t,x)}$ is {\em strongly convex} as a hypersurface of $T_xN$, for all $(t,x) \in M$. This means that $\Sigma_{(t,x)}$ has positive sectional curvature everywhere (with respect to the natural Euclidean metric on $T_xN$) or, equivalently, the second fundamental form of $\Sigma_{(t,x)}$ with respect to the inner normal vector is positive definite.
\end{itemize}
This condition also ensures that the strict triangle inequality is satisfied: $F_t(v,u) \leq F_t(v)+F_t(u)$ for all $v,u \in T_xN$, with equality if and only if $v=\lambda u$ or $u=\lambda v$, for some $\lambda \geq 0$. To be more precise, this inequality holds if and only if the closed $F_t$-unit ball $\{v\in T_xN: F_t(v)\leq 1\}$ is strictly convex as a subset of $\mathbb{R}^n$ (see \cite[Proposition~2.3(iv)]{JS2014}), which is a slightly weaker condition than (iv).

\begin{definition}
\label{def:finsler_metrics}
Any function $F_t: TN \rightarrow [0,\infty)$ that varies smoothly with $t \in \mathbb{R}$ and satisfies the conditions (i), (ii), (iii) and (iv) is called a {\em (time-dependent) Finsler metric} on $N$.\footnote{See \cite{JS2014} for background regarding Finsler metrics.} When $F(t,x,v)=F(t,v)$ depends on time $t\in \mathbb{R}$ but not on position $x \in N$, we say that it is a {\em time-only dependent Finsler metric}. When $F(t,x,v)=F(v)$ is both time- and position-independent, we say that it is a {\em constant Finsler metric} or a {\em Minkowski norm}. 
\end{definition}

\begin{notation}
\label{not:F}
We will write indistinctly $F_t(v)$ or $F(t,x,v)$, depending on the context. $F_t(v)$ highlights the time-dependence---we will write $F(v)$ when time-independent---and we can omit the reference to the position $x$, as it is always the base point of the vector $v$. In contrast, $F(t,x,v)$ explicitly shows all the actual dependencies of the function, making it more suitable for working in coordinates.
\end{notation}

To our knowledge, Zermelo navigation has always been studied with the strong convexity assumption (iv), i.e. the velocity vectors of the moving object form a smooth distribution of strongly convex {\em ovals}---hypersurfaces diffeomorphic to a sphere. The main purpose of this work is to extend this study to the cases where (iv) no longer holds. For now, however, we will assume throughout this review section that $F_t$ in \eqref{eq:F_t} is a time-dependent Finsler metric. In terms of $F_t$, given a curve $\hat{\gamma}(s)=(t(s),\gamma(s))$ in $M$, \eqref{eq:dot_t} becomes
\begin{equation}
    \label{eq:dot_t_F}
    t'(s)=F_{t(s)}(\gamma'(s)) \equiv F(t(s),\gamma(s),\gamma'(s)),
\end{equation}
the travel time \eqref{eq:traveltime} can be computed as
\begin{equation}
\label{eq:traveltime_F}
    \mathcal{T}[\hat{\gamma}]=t(b)=\int_a^b F(t(s),\gamma(s),\gamma'(s)) \ \mathrm{d}s,
\end{equation}
and we can reformulate Zermelo's problem as follows:
\begin{quote}
Given $M=\mathbb{R}\times N$, a time-dependent Finsler metric $F_t$ on $N$ and initial and arrival regions $A, B\subset N$, find the curve $\hat{\gamma}(s)=(t(s),\gamma(s))$, $s \in [a,b]$, with the conditions $t(a)=0$, $\gamma(a)\in A$ and $\gamma(b)\in B$, such that \eqref{eq:dot_t_F} holds and the travel time $\mathcal{T}[\hat{\gamma}]=t(b)$ is minimum.
\end{quote}

\begin{remark}
Observe that the travel time \eqref{eq:traveltime_F} depends only on $F_t$ and, in fact, Zermelo's problem can be formulated (and solved) without making explicit reference to the background Riemannian metric $\langle\cdot,\cdot\rangle$ and the speed profile $V$. This means that the only explicit information we need about the moving object is its velocity indicatrix vectors $\Sigma$, which univocally define $F_t$ (recall Remark~\ref{rem:indicatrix}).
\end{remark}

\subsection{The time-independent case}
\label{subsec:time-independent}
When $F_t=F$ is time-independent, we can work just at the level of the space $N$. We define the {\em geodesics} of $F$ between two points $x_1,x_2 \in N$ as the critical points of the {\em $F$-energy} functional
\begin{equation}
\label{eq:F_energy}
    \mathcal{E}_F[\gamma]\coloneqq \int_a^b F(\gamma(s),\gamma'(s))^2 \ \mathrm{d}s,
\end{equation}
among all (piecewise smooth) curves $\gamma:[a,b]\rightarrow N$ with fixed endpoints $x(a)=x_1$, $x(b)=x_2$.\footnote{By {\em critical point} we mean the usual notion that the curve makes the functional stationary under any standard variation. We formalize this concept for the travel time functional in Definition~\ref{def:travel_functional} below.} Since $\mathcal{E}[x]$ depends on the parametrization of $\gamma(s)$, which is not relevant in our problem, we define the {\em pregeodesics} of $F$ as those curves that can be reparametrized as geodesics, keeping the orientation. It turns out that pregeodesics are critical points of the {\em $F$-length} functional (see e.g. \cite[Section~5]{BCS})
\begin{equation}
    \label{eq:F_lenght}
    \mathcal{L}_F[\gamma]\coloneqq \int_a^b F(\gamma(s),\gamma'(s)) \ \mathrm{d}s,
\end{equation}
which exactly coincides with the travel time \eqref{eq:traveltime_F}. This means that, in the time-independent case, the curves that solve Zermelo's problem are pregeodesics of $F$. In the time-dependent case, in contrast, we cannot define geodesics only at the level of the space $N$, because the energy and length functionals also depend on the time function $t(s)$, so we need to generalize these notions at the level of the spacetime $M$.

\subsection{Lorentz-Finsler metrics}
\label{subsec:lorentz-finsler}
The goal is to find a metric on $M$ that describes the travel of the moving object through the spacetime when the object is moving with maximal speed according to its given speed profile indicatrix. Since $t'(s)>0$ (time increases as the object travels) and $F_t$ is positive, observe that \eqref{eq:dot_t} and \eqref{eq:dot_t_F} are equivalent to
\begin{equation*}
    t'(s)^2-F_{t(s)}(\gamma'(s))^2=0.
\end{equation*}
The left-hand side of this equation now defines the following function on $TM$, which carries the same information as $F_t$:
\begin{equation}
\label{eq:lorentz_finsler_metric}
    H(\hat{v})\equiv H(q,\hat{v}) \coloneqq (v^0)^2-F_t(v)^2, \qquad \forall q=(t,x)\in M, \quad \forall \hat{v}=(v^0,v) \in T_qM,
\end{equation}
with the following properties:
\begin{itemize}
    \item $H$ is continuous everywhere and smooth away from $\textup{Span}(\partial_t)=\{\hat{v}=(\lambda,0)\in T_qM: q \in M, \ \lambda\in \mathbb{R}\}$.
    \item $H$ is positive homogeneous of degree 2: $H(\lambda \hat{v})=\lambda^2H(\hat{v})$ for all $\lambda>0$ and all $\hat{v} \in T_qM$.
\end{itemize}
As with $F_t$, we can define the {\em fundamental tensor} of $H$ in the direction $\hat{v}\in T_qM \setminus \textup{Span}(\partial_t)$ as the Hessian of $\frac{1}{2}H$ at $\hat{v}$:\footnote{Observe that $g^H$ is defined in terms of $H$ (in contrast with $g^{F_t}$, defined in terms of $F_t^2$) because $H$ is already positive homogeneous of degree 2.}
\begin{equation*}
g^{H}_{\hat{v}}(\hat{u},\hat{w}) \coloneqq \left. \textup{Hess}\left(\frac{1}{2} H\right) \right\rvert_{\hat{v}}(\hat{u},\hat{w}) = \frac{1}{2} \left. \frac{\partial^2}{\partial \delta \partial \eta} H(\hat{v} + \delta \hat{u} + \eta \hat{w}) \right\rvert_{\delta=\eta=0}, \quad \forall \hat{u},\hat{w} \in T_qM,
\end{equation*}
which satisfies
\begin{equation*}
    g^{H}_{\hat{v}}(\hat{v},\hat{u})=\frac{1}{2} \left. \frac{\partial}{\partial \delta} H(\hat{v} + \delta \hat{u}) \right\rvert_{\delta=0}, \qquad g^{H}_{\hat{v}}(\hat{v},\hat{v})=H(\hat{v}).
\end{equation*}
In matrix form, using natural Euclidean coordinates---so that $H(v) \equiv H(t,x^1,\ldots,x^n,v^0,\ldots,v^n)$:
\begin{equation*}
g^{H}_{\hat{u}}(\hat{u},\hat{w}) =
\begin{pmatrix}
u^0 & \cdots & u^n
\end{pmatrix}
\begin{pmatrix}
\frac{1}{2}\frac{\partial^2 H}{\partial(v^0)^2}(\hat{v}) & \cdots & \frac{1}{2}\frac{\partial^2 H}{\partial v^0 \partial v^n}(\hat{v}) \\
\vdots & \ddots & \vdots \\
\frac{1}{2}\frac{\partial^2 H}{\partial v^n \partial v^0}(\hat{v}) & \cdots & \frac{1}{2}\frac{\partial^2 H}{\partial(v^n)^2}(\hat{v})
\end{pmatrix}
\begin{pmatrix}
w^0 \\
\vdots \\
w^n
\end{pmatrix},
\end{equation*}
where
\begin{equation}
\label{eq:matrix_gH}
    g^H_{ij}(\hat{v}) \coloneqq \frac{1}{2}\frac{\partial^2 H}{\partial v^i \partial v^j}(\hat{v}), \quad i,j=0,\ldots,n,
\end{equation}
are the coefficients of the coordinate matrix of $g^H_{\hat{v}}$.

Note that the relationship between the fundamental tensors of $H$ and $F_t$ is
\begin{equation}
    \label{eq:relation_tensors}
    g^H_{\hat{v}}(\hat{u},\hat{w})=u^0w^0-g^{F_t}_{v}(u,w) \quad  \textrm{for} \quad \hat{v} = (v^{0}, v), \quad \hat{u} = (u^{0}, u), \quad \hat{w} = (w^{0}, w),
\end{equation}
which immediately tells us that, if $F_t$ is a time-dependent Finsler metric, so that $g^{F_t}_{v}(\cdot,\cdot)$ is positive definite, then $g_{\hat{v}}^H(\cdot,\cdot)$ has index $n$, i.e. it is a Lorentzian scalar product with signature $(+,-,\ldots,-)$ at each tangent space $T_qM$ and for each direction $\hat{v}$. The Lorentzian signature naturally distinguishes between spatial and temporal dimensions, and it is a standard way to model propagation through spacetime---e.g. this is the way general relativity models the motion of material (timelike) particles as well as the propagation of light along (lightlike) geodesics. 

\begin{definition}
\label{def:causality}
    Given a time-dependent Finsler metric $F_t$, we say that $H=\mathrm{d}t^2-F_t^2$ in \eqref{eq:lorentz_finsler_metric} is the {\em Lorentz-Finsler metric} associated with $F_t$, and we define the following notions:
    \begin{itemize}
        \item The {\em future lightcone} $\mathcal{C}_q$ at each $q\in M$ is defined as
        \begin{equation*}
        \mathcal{C}_q \coloneqq \{\hat{v}=(v^0,v)\in T_qM: H(\hat{v})=0, \ v^0>0\},
        \end{equation*}
        and we say that $-\mathcal{C}_q=\{\hat{v}\in T_qM: -\hat{v}\in\mathcal{C}_q\}$ is the {\em past lightcone}.
        
        \item The open domain $\mathcal{A}_q \subset T_qM\setminus \{0\}$ enclosed by $\mathcal{C}_q$ is called the {\em cone domain}. Note that
        \begin{equation*}
        \mathcal{A}_q \coloneqq \{\hat{v}=(v^0,v)\in T_qM: H(\hat{v})>0, \ v^0>0\}.
        \end{equation*}
        
        \item We say that a vector $\hat{v}\in T_qM$ is
            \begin{itemize}
                \item {\em timelike} if $\hat{v}$ or $-\hat{v}$ belongs to $\mathcal{A}_q$,
                \item {\em lightlike} if $\hat{v}$ or $-\hat{v}$ belongs to $\mathcal{C}_q$,
                \item {\em causal} if $\hat{v}$ or $-\hat{v}$ is timelike or lightlike,
                \item {\em spacelike} if it is not causal.
            \end{itemize}
            
        \item A causal vector $\hat{v}\in T_qM$ is {\em future-directed} if $\hat{v}\in \mathcal{C}_q \cup \mathcal{A}_q$, and {\em past-directed} if $-\hat{v}\in \mathcal{C}_q \cup \mathcal{A}_q$.
        
        \item Analogously, a curve $\hat{\gamma}:[a,b] \rightarrow M$ is {\em (future or past-directed)} {\em timelike}, {\em lightlike}, {\em causal} or {\em spacelike}, when $\hat{\gamma}'(s)$ is so for all $s\in[a,b]$, respectively.
    \end{itemize}
\end{definition}

\begin{notation}
    Following Notation~\ref{not:F}, we will write $H(\hat{v})$ or $H(p,\hat{v})$ indistinctly. In addition, from now on, any causal vector or curve that is not explicitly stated as past-directed will be implicitly assumed to be future-directed.
\end{notation}

\begin{remark}
\label{rem:ligthcones}
    The lightcones of $H$, which can be expressed as
    \begin{equation*}
        \mathcal{C} \coloneqq \cup_{q\in M} \mathcal{C}_q = H^{-1}(0)\cap \mathrm{d}t^{-1}((0,\infty)),
    \end{equation*}
    define what is called a {\em cone structure} (essentially, a smooth distribution of cones).\footnote{See \cite{JS2020} for background regarding cone structures, Lorentz-Finsler metrics and Finsler spacetimes.} In fact, note that $v\in \Sigma_q$ if and only if $(1,v)\in \mathcal{C}_q$, so the cone structure contains the same information as the speed profile indicatrix and is thus the only explicit information we need about the moving object to solve the time-dependent Zermelo's problem (see \cite{javaloyes2021a}).
\end{remark}

In terms of $H$, given a curve $\hat{\gamma}(s)=(t(s),\gamma(s))$ in $M$, \eqref{eq:dot_t} and \eqref{eq:dot_t_F} are equivalent to
\begin{equation}
    \label{eq:dot_t_H}
    H(\hat{\gamma}'(s))\equiv H(\hat{\gamma}(s),\hat{\gamma}'(s)) = 0,
\end{equation}
which implies, as $t'(s)>0$, that $\hat{\gamma}(s)$ represents an actual trajectory of the moving object if and only if it is lightlike. In fact, we can deduce that timelike curves represent trajectories traveled at speeds lower than the speed profile $V$---the maximum speed---while spacelike curves are impossible trajectories, as they would require speeds exceeding the maximum. Following this reasoning, causal curves represent "admissible" trajectories for the moving object, although in Zermelo navigation it is always optimal to follow lightlike rather than (non-lightlike) causal curves.

Summing up, $\mathcal{C}$ can be interpreted as the "infinitesimal" travel of the object through the spacetime $M$, just like $\Sigma$ represents it through the space $N$, and we have yet another equivalent formulation of Zermelo's problem:
\begin{quote}
Given $M=\mathbb{R}\times N$, a time-dependent Finsler metric $F_t$ on $N$ and initial and arrival regions $A, B \subset N$, find the lightlike curve $\hat{\gamma}(s)=(t(s),\gamma(s))$, $s \in [a,b]$, of the associated Lorentz-Finsler metric $H=\mathrm{d}t^2-F_t^2$, with the conditions $t(a)=0$, $\gamma(a)\in A$ and $\gamma(b)\in B$, such that the travel time $\mathcal{T}[\hat{\gamma}]=t(b)$ is minimum.
\end{quote}

The main advantage of this formulation, as we have already anticipated, is that we can define geodesics in this general time-dependent setting, which will play a key role in Zermelo navigation. Specifically, we define the {\em geodesics} of $H$ between two points $q_1, q_2 \in M$ as the critical points of the $H$-energy functional
\begin{equation}
\label{eq:H_energy}
    \mathcal{E}_H[\hat{\gamma}] \coloneqq \int_a^b H(\hat{\gamma}(s),\hat{\gamma}'(s)) \ \mathrm{d}s,
\end{equation}
among all (piecewise smooth) curves $\hat{\gamma}: [a,b] \rightarrow M$ with fixed endpoints $\hat{\gamma}(a)=q_1$, $\hat{\gamma}(b)=q_2$. As usual, in order to remove the dependence on the parametrization, we also define the {\em pregeodesics} of $H$ as those curves that can be parametrized as geodesics, keeping the orientation.\footnote{Now we cannot identify these pregeodesics as critical points of a length functional, as in \eqref{eq:F_lenght}, because $\sqrt{H}$ fails to be smooth on lightlike vectors.}

\subsection{The solution to Zermelo's problem}
\label{subsec:solution}
Let us first identify the curves we are looking for.
\begin{definition}
\label{def:travel_functional}
    Let $\mathcal{Q}_{A,B}$ be the set of all the candidate solutions to Zermelo's problem, i.e. lightlike curves $\hat{\gamma}:[a,b] \rightarrow M$ departing from $A \subset N$ at time $t=0$ and arriving at $B \subset N$. Then:
    \begin{itemize}
        \item $\mathcal{T}: \mathcal{Q}_{A,B} \rightarrow (0,\infty)$ given by \eqref{eq:traveltime_F} is the {\em travel time functional} on $\mathcal{Q}_{A,B}$.
        \item $\hat{\gamma} \in \mathcal{Q}_{A,B}$ is a {\em critical point} of $\mathcal{T}$ if
        \begin{equation*}
        \left. \frac{d}{d\omega}\mathcal{T}[\hat{\gamma}_{\omega}] \right\rvert_{\omega=0}=0,
        \end{equation*}
    for any variation $\hat{\gamma}_{\omega}(s)$ of $\hat{\gamma}$ through curves in $\mathcal{Q}_{A,B}$, i.e. $\hat{\gamma}_0=\hat{\gamma}$ and $\hat{\gamma}_{\omega} \in \mathcal{Q}_{A,B}$ for all $\omega \in (-\varepsilon,\varepsilon)$, with $\varepsilon > 0$. This is independent of the parametrization of $\hat{\gamma}$.
    \end{itemize}
\end{definition}

Therefore, the goal is to find a global minimum of $\mathcal{T}$. It turns out that this formulation of Zermelo's problem---in terms of a Lorentz-Finsler metric $H$---is equivalent to the Finslerian version of Fermat's principle, which characterizes the critical points of the travel time functional between two points as lightlike pregeodesics of $H$ (see \cite[Theorem~4.2]{perlick2006a}). In addition, when the initial and arrival regions $A, B$ are hypersurfaces of $N$ (not single points), then the curve must also depart and arrive ``orthogonally'' to $A$ and $B$, respectively, in order to be critical (see \cite[Section~4.1]{javaloyes2021a} and \cite{JMPS,perlick1998a}). This orthogonality is measured with respect to the time-dependent Finsler metric, in the following sense.

\begin{definition}
    Let $F_t$ be a time-dependent Finsler metric on $N$. We say that $v\in T_xN\setminus \{0\}$ is {\em $F_t$-orthogonal} to $u \in T_xN$ at time $t\in \mathbb{R}$, denoted $v\bot_{F_t}u$, when
    \begin{equation*}
        g^{F_t}_v(v,u)=0.
    \end{equation*}
    Analogously, if $S$ is a submanifold of $N$, we say that $v$ is {\em $F_t$-orthogonal} to $S$ at time $t \in \mathbb{R}$, denoted $v \bot_{F_t}S$, if $v\bot_{F_t}u$ for all $u \in T_xS$.
\end{definition}

All in all, we have the following result (see also \cite[Corollary~6.10]{JS2020}).
\begin{theorem}
\label{thm:sol_zermelo}
    In the Zermelo navigation setting presented so far:
    \begin{itemize}
        \item[(i)] For any $A,B \subset N$, a global minimum of $\mathcal{T}$ (i.e. a solution to Zermelo's problem) exists if $ F_t $ is upper bounded by any time-independent Finsler metric and lower bounded by any time-independent complete Finsler metric.\footnote{{\em Complete} means that any geodesic can be extended to be defined in $\mathbb{R}$.}

        \item[(ii)] $\hat{\gamma}(s)=(t(s),\gamma(s))$ in $\mathcal{Q}_{A,B}$ is a critical point of $\mathcal{T}$ if and only if it is a pregeodesic of $H$ such that $\gamma'(a) \bot_{F_0} A$ and $\gamma'(b) \bot_{F_{t(b)}} B$.
    \end{itemize}
\end{theorem}

\begin{remark}
\label{rem:global_hyp}
    We will assume throughout this work that the condition in (i) holds, in order to ensure that the solution to Zermelo's problem always exists. In particular, this implies that the spacetime $(M,H)$ is {\em globally hyperbolic}, with each constant time slice $\{t=t_0\}$ being a {\em Cauchy hypersurface}, i.e. it is crossed exactly once by any inextendible causal curve (see \cite{JP} for the importance of this condition when looking for trajectories that minimize the travel time).
\end{remark}

Note that any solution to Zermelo's problem is, in particular, a critical point of $\mathcal{T}$ on $\mathcal{Q}_{A,B}$. However, the converse is obviously not true, as not every critical point of $\mathcal{T}$ must be a global minimum. Therefore, Theorem~\ref{thm:sol_zermelo} narrows the set of curves where we should look for solutions, but we still need to distinguish the global minima among the critical points.

\subsection{Wavemap and cut points}
\label{subsec:wavemap}
Following \cite[Section~4]{javaloyes2021a}, the basic idea we will implement to identify global minima of $\mathcal{T}$ is to include all the candidate critical points in a single map, so that we can easily compare their travel times. For this, it will be more convenient to consider $t$-parametrized trajectories---so that the parameter of the curves directly indicates the travel time.

\begin{definition}
\label{def:wavemap}
    Given an initial region $A \subset N$ and a time-dependent Finsler metric $F_t$, we define the {\em wavemap} from $A$ as
    \begin{equation}
    \label{eq:wavemap}
    \begin{array}{rrll}
    f \colon & [0,\infty)\times A & \longrightarrow & N\\
    & (t,z) & \longmapsto & f(t,z)=(x^1(t,z),\ldots,x^n(t,z)),
    \end{array}
    \end{equation}
    where, for each fixed $z_0 \in A$, the curve $t \mapsto \hat{f}(t,z_0) \coloneqq (t,f(t,z_0))$ in $M$ is the unique $t$-parametrized lightlike pregeodesic of $H=\mathrm{d}t^2-F_t^2$ with $f(0,z_0)=z_0\in A$ and $\partial_tf(0,z_0)\bot_{F_0}A$ pointing outward from $A$.\footnote{The wavemap was originally introduced in \cite{javaloyes2021a} to model the propagation of anisotropic waves (thence its name), since individual wave trajectories are solutions to Zermelo's problem (see also \cite{EPR2024}).} Any curve $t \mapsto f(t,z)$ in $N$ will be called a {\em spatial trajectory} of the wavemap, whereas the corresponding lightlike curve $t \mapsto \hat{f}(t,z) \coloneqq (t,f(t,z))$ in $M$ will be called a {\em spacetime trajectory}.
\end{definition}

\begin{notation}
\label{not:vel_wavemap}
    We will write
    \begin{equation*}
        \partial_tf(t_0,z_0) \coloneqq \left. \frac{\partial f(t,z_0)}{\partial t}\right\rvert_{t=t_0} \in T_{f(t_0,z_0)}N, \qquad \partial_t\hat{f}(t_0,z_0)\coloneqq (1,\partial_tf(t_0,z_0))\in T_{\hat{f}(t_0,z_0)}M,
    \end{equation*}
    to denote the velocity vectors of the spatial and spacetime trajectories of the wavemap, respectively, at any time $t=t_0$.
\end{notation}

\begin{remark}
\label{rem:wavemap}
    Some properties of the wavemap are worth mentioning:
    \begin{itemize}
        \item Assuming that condition (i) in Theorem~\ref{thm:sol_zermelo} holds, every trajectory of the wavemap can be extended to be defined in $[0,\infty)$ and, in fact, the image of the wavemap is $N \setminus \textup{Int}(A)$.
        
        \item Every spacetime trajectory of the wavemap is a (smooth) lightlike curve by definition, so it effectively represents a trajectory of the moving object, i.e. the following equivalent conditions hold (recall \eqref{eq:dot_t_F} and \eqref{eq:dot_t_H}):
        \begin{equation*}
            H(\partial_t\hat{f}(t,z))=0 \Leftrightarrow \partial_t\hat{f}(t,z) \in \mathcal{C}_{\hat{f}(t,z)} \Leftrightarrow F_t(\partial_tf(t,z))=1 \Leftrightarrow \partial_tf(t,z) \in \Sigma_{\hat{f}(t,z)}.
        \end{equation*}

        \item Given a point $z_0\in A$, there are exactly two directions $F_0$-orthogonal to $A$ at $z_0$: one points to the exterior of $A$, and the other to the interior. Here we are selecting the outward trajectories---we are assuming that $B$ is not contained inside $A$---although one can also define the analogous ``inward'' wavemap.

        \item When $A$ is a single point, then every vector is assumed to be $F_0$-orthogonal to $A$, so the parameter $z=\theta \in (0,2\pi]$ parametrizes directions in $T_AN$, rather than points in $A$ (see the last part of Section~\ref{subsec:circles} below).
    \end{itemize}
\end{remark}

The comparison of the travel times gives rise to the following definition, which distinguishes the solutions to Zermelo's problem among all the trajectories in the wavemap.
\begin{definition}
\label{def:time_min_wavemap}
    A spatial trajectory $t \mapsto f(t,z_0)$ of the wavemap is {\em time-minimizing} (resp. {\em strictly time-minimizing}) from $A$ at time $\tau > 0$ if, for every $t_0 \in (0,\tau]$, any $t$-parametrized causal curve $\hat{\varphi}(t)=(t,\varphi(t))$ such that $\varphi(0)\in A$ and $\varphi(t_1)=f(t_0,z_0)$ satisfies $t_0 \leq t_1$ (resp. $t_0<t_1$). We define the {\em wavefront} at $t=\tau$ as
    \begin{equation*}
        \mathcal{W}_{\tau} \coloneqq \{f(\tau,z)\in N: z \in A, \ t\mapsto f(t,z) \text{ is time-minimizing at } t=\tau\},
    \end{equation*}
    with $\mathcal{W}_0 \coloneqq A$.
\end{definition}

Essentially, if $t \mapsto f(t,z_0)$ is time-minimizing at $t=\tau$, it means that the corresponding spacetime trajectory is a solution to Zermelo's problem for the travel from the initial region $A$ to the arrival point $f(t_0,z_0)$, for any $t_0 \in (0,\tau]$. However, a time-minimizing trajectory might lose this property after some time, hence the following notion.

\begin{definition}
    We define the {\em cut function} of the wavemap as
    \begin{equation*}
    \begin{array}{rrll}
    c \colon & A & \longrightarrow & (0,+\infty] \\
    & z & \longmapsto & c(z) := \textup{Sup}\{ t > 0: f(t,z) \in \mathcal{W}_t \}.
    \end{array}
    \end{equation*}
    If $c(z_0) < \infty$ for some $z_0 \in A$, we call $c(z_0) \in (0,\infty)$ and $f(c(z_0),z_0)) \in N$ the {\em cut instant} and {\em cut point}, respectively, of the corresponding spatial trajectory $t\mapsto f(t,z_0)$.
\end{definition}

With all these ingredients, we can now establish the following results, which are key for our purposes here (see \cite[Proposition~A.1]{Javaloyes2023a}, \cite[Section~4.1]{javaloyes2021a} and \cite[Theorem~3]{JP}).

\begin{proposition}
\label{prop:strictly_min}
    Every spatial trajectory of the wavemap is time-minimizing (resp. strictly time-minimizing) at its cut instant (resp. before its cut instant). Moreover, any $t$-parametrized curve $\varphi(t)$ in $N$ such that $(t,\varphi(t))$ is causal, $\varphi(0)\in A$ and $\varphi(t_0)\in \mathcal{W}_{t_0}$ for some $t_0>0$, must be a time-minimizing spatial trajectory of the wavemap.
\end{proposition}

\begin{proposition}
\label{prop:cut_points}
    $t_0 \in (0,\infty)$ is the cut instant of $t \mapsto f(t,z_0)$ if and only if at least one (but possibly both) of the following conditions holds:
    \begin{itemize}
        \item[(a)] $\hat{f}(t_0,z_0)\in M$ is the first intersection point of the spacetime trajectory $t \mapsto \hat{f}(t,z_0)$ with another spacetime trajectory of the wavemap.
        \item[(b)] $\hat{f}(t_0,z_0)\in M$ is the first {\em focal point} of $A$ along $t \mapsto \hat{f}(t,z_0)$, i.e. there exists a non-trivial variation $\hat{\gamma}_{\omega}(t)$ of $t \mapsto \hat{f}(t,z_0)$ through lightlike pregeodesics of $H$ departing $F_0$-orthogonally from $A$, such that $\frac{\partial}{\partial \omega}\hat{\gamma}_\omega(t_0)|_{\omega=0}=0$ (and this does not occur for $t<t_0$).
    \end{itemize}
\end{proposition}

\begin{proposition}
\label{prop:cut_function}
    There exists a time $\tau > 0$ such that $c(z) \geq \tau$ for all $z \in A$, i.e. every spatial trajectory of the wavemap is time-minimizing (resp. strictly time-minimizing) at $t=\tau$ (resp. at any $t<\tau$). Moreover, for any $t \in (0,\tau)$, the wavefront $\mathcal{W}_t=\{f(t,z):z\in A\}$ is a smooth hypersurface of $N$.
\end{proposition}

Then, if $\tau>0$ is the time provided by Proposition~\ref{prop:cut_function}, the wavemap $f(t,z)=(x^1(t,z),\ldots,x^n(t,z))$ satisfies (and is characterized by) the following {\em orthogonality conditions} in $[0,\tau) \times A$ (see \cite[Theorem~4.14]{javaloyes2021a}):
\begin{equation}
    \label{eq:ort_F}
    \left\lbrace{
    \begin{array}{l}
    F_t(\partial_tf(t,z)) = 1, \\
    \partial_tf(t,z) \bot_{F_t} \mathcal{W}_t, \quad \text{with } \partial_tf \text{ pointing outward},
    \end{array}
    }\right.
\end{equation}
and the following {\em geodesic equations} in $[0,\infty) \times A$ (see \cite[Theorem~4.11]{javaloyes2021a}):
\begin{equation}
    \label{eq:geodesics_H}
    \left\lbrace{
    \begin{array}{l}
    \partial_t^2 x^k = -\gamma_{\ ij}^k(\partial_t\hat{f})\partial_t x^i \partial_t x^j + \gamma_{\ ij}^0(\partial_t\hat{f}) \partial_t x^i \partial_t x^j \partial_t x^k, \quad k=1,\ldots,n.\\
    \text{Initial position: } f(0,z)=z \in A.\\
    \text{Initial velocity: } \partial_tf(0,z) \text{ satisfying \eqref{eq:ort_F} at } t=0.
    \end{array}
    }\right.
\end{equation}
Here, Einstein's summation convention is used, summing the indices that appear up and down from $0$ to $n$ (with $x^0=t$), and $\gamma_{\ ij}^k(\hat{v})$ are the {\em formal Christoffel symbols} of $H=\mathrm{d}t^2-F_t^2$, defined for every $\hat{v}\in \mathcal{C}_q$ as
\begin{equation*}
    \gamma_{\ ij}^k(\hat{v}) \coloneqq \frac{1}{2}(g^H)^{kr}(\hat{v})\left(\frac{\partial g^H_{rj}}{\partial x^i}(\hat{v})+\frac{\partial g^H_{ri}}{\partial x^j}(\hat{v})-\frac{\partial g^H_{ij}}{\partial x^r}(\hat{v})\right), \quad i,j,k = 0,\ldots,n,
\end{equation*}
where $g^H_{ij}(\hat{v})$ are given by \eqref{eq:matrix_gH}, and $(g^H)^{ij}(\hat{v})$ denote the coefficients of the inverse matrix.\footnote{Using \eqref{eq:relation_tensors}, the geodesic equations of $H$ can be expressed only in terms of the fundamental tensor of $F_t$ (see \cite[Equation~(3)]{JPS2025}).}

\section{The time-only dependent case}
\label{sec:time_only}
There is a specific situation of great theoretical interest, as first observed by Richards in \cite{richards1993a}. This is the case when the Finsler metric $F_t$ is time-only dependent, i.e. the velocity vectors of the moving object $\Sigma_{(t,x)}=\Sigma_t$ depend on time $t\in \mathbb{R}$ but not on position $x\in N$. Our goal in this section is to show that, in this case, the wavemap exhibits exceptionally favorable properties and can even be determined analytically when the expression of $F_t$ is sufficiently simple.

In practice, although time-only dependence might initially seem overly restrictive, it is often a reasonable assumption in most of the realistic scenarios we are interested in here. The paradigmatic example is sailboat navigation, where the speed profile is mainly determined by the wind, which can be assumed to be approximately uniform across a small region, varying only with time.\footnote{In contrast, this assumption is not appropriate in situations where the velocity depends on properties of the space. For instance, in wildfire propagation, the velocity of the fire depends on the slope (see \cite{Javaloyes2023a}), making it strongly position-dependent.}

\subsection{New properties of the wavemap}
\label{subsec:time_only_general}
First, we present new results that show the remarkably good behavior of the wavemap in the time-only dependent case. This not only enhances the theoretical interest of the scenario, but also provides key properties that will be essential for some of the theoretical developments in the subsequent sections. We begin with an observation taken from \cite[Remark~4.2]{Javaloyes2023a}.

\begin{remark}
\label{rem:killing}
    When $F_t$ is time-only dependent, every constant spatial vector field $v$ is Killing for $H$, which implies that the orthogonality conditions \eqref{eq:ort_F} reduce to
    \begin{equation}
        \label{eq:ort_F_killing}
        \left\lbrace{
        \begin{array}{l}
        F_t(\partial_tf(t,z)) = 1, \\
        \partial_tf(t,z) \bot_{F_t} v, \quad \forall v \in T_{z}A, \quad \text{with } \partial_tf \text{ pointing outward}.
    \end{array}
    }\right.
\end{equation}
\end{remark}

In the following results, by {\em convexity} we mean the usual notion in $\mathbb{R}^n$. In particular, $A$ or any wavefront $\mathcal{W}_t$ is {\em convex} (resp. {\em strictly convex}) when the region they enclose is convex (resp. strictly convex) as a subset of $N=\mathbb{R}^n$.

\begin{theorem}
\label{thm:no_cut_points}
In the time-only dependent case, if $A$ is convex (or a single point), then the wavemap has no cut points, i.e. every spatial trajectory of the wavemap is strictly time-minimizing for all time.
\end{theorem}
\begin{proof}
Let $\gamma: t \mapsto f(t,z_0)$ be a spatial trajectory of the wavemap, and $\hat{\gamma}: t \mapsto \hat{f}(t,z_0) = (t,f(t,z))$ its corresponding spacetime trajectory. Consider the affine hyperplane $\eta \coloneqq f(0,z_0)+T_{z_0}A = \{f(0,z_0)+\lambda v \in \mathbb{R}^n: v\in T_{z_0}A, \ \lambda\in \mathbb{R}\}$, tangent to $A$ at $f(0,z_0)$, and assume first, as a limit case, that $\eta$ is the initial region, so that every $v \in T_{z_0}A$ is also tangent to $\eta$ at every point. This means that every spacetime trajectory of the wavemap from $\eta$ (pointing to the same side of $\eta$ as $\gamma$) has exactly the same shape as $\hat{\gamma}$ in $M=\mathbb{R}^{n+1}$ but displaced from each other by a spatial translation tangent to $\eta$, as they all satisfy the same position-independent differential equations \eqref{eq:geodesics_H} with the same initial velocity---determined by \eqref{eq:ort_F_killing} at $t=0$---but different initial points in $\eta$. In fact, all the wavefronts from $\eta$ are affine hyperplanes parallel to $\eta$. Therefore, applying Proposition \ref{prop:cut_points}, $\gamma$ cannot have a cut point because, on the one hand, $\hat{\gamma}(t) \not= \hat{f}(t,z)$ for all $t \in [0,\infty)$ and $z \in \eta \setminus \{z_0\} $, and on the other hand, if we consider a non-trivial variation $\hat{\gamma}_{\omega}(t)=(t,\gamma_{\omega}(t))$ of $\hat{\gamma}$ by ($t$-parametrized) lightlike pregeodesics of $H$ departing $F_0$-orthogonally from $\eta$, it is clear by definition that $t \mapsto \hat{\gamma}_{\omega}(t)$ is a spacetime trajectory of the wavemap from $\eta$, which only differs from $\hat{\gamma}$ by a constant spatial translation, so
\begin{equation*}
    \hat{\gamma}_{\omega}(t)=\hat{f}(t,\gamma_{\omega}(0)) = \hat{\gamma}(t)+ (0,v_{\omega}),
\end{equation*}
where $v_{\omega} \coloneqq \gamma_{\omega}(0)-\gamma(0) \in \mathbb{R}^n$, and $\gamma_{\omega}(0)$ is a regular curve in $\eta$ (since the variation is non-trivial). Therefore, $\left. \frac{\partial}{\partial\omega}\hat{\gamma}_{\omega}(t)\right\rvert_{\omega=0} = \left. \frac{\partial}{\partial\omega}\hat{\gamma}_{\omega}(0)\right\rvert_{\omega=0} \not=0$.

Assume now that $A$ is the initial region. If $\gamma(t_0)$ is the cut point of $\gamma$, then the global hyperbolicity of the spacetime (recall Remark~\ref{rem:global_hyp}) ensures that, for any $\gamma(t_0+\delta)$, $\delta > 0$, there exists another spatial trajectory $\varphi: t \mapsto f(t,z_1)$ of the wavemap such that $\varphi(t_1) = \gamma(t_0+\delta)$, with $t_1 < t_0+\delta$. Now, notice that $\eta$ splits $N=\mathbb{R}^n$ into two half-spaces, one containing $A$ (due to its convexity) and the other $\gamma$. So, $\varphi$ must spend some time $\varepsilon \geq 0$ to reach $\eta$. However, this means that we can construct the (non-lightlike) causal curve
\begin{equation}
    \label{eq:causal_curve}
    \hat \rho(t) :=
    \left\lbrace 
    \begin{array}{l}
    (t,\varphi(\varepsilon)), \quad t\in[0,\varepsilon), \\
    (t,\varphi(t)), \quad t\in[\varepsilon,t_1],
    \end{array}
    \right.
\end{equation}
whose projection on $N$ departs from $\eta$ at time $t=0$ and arrives at $\gamma(t_0+\delta)$ earlier than $\gamma$, which is a contradiction because we have seen that $\gamma$ is strictly time-minimizing from $\eta$ for all time.

Finally, if $A$ is a single point, at least for a small time the wavefront must be convex (in fact, strongly convex; see \cite{JMPS}), as so is the indicatrix $\Sigma_A$---which represents the infinitesimal propagation. Taking this wavefront as the new initial region, we are back to the previous case and, therefore, we conclude that the wavemap has no cut points.
\end{proof}

\begin{theorem}\label{thm:convexity}
In the time-only dependent case, if $A$ is strictly convex or a single point (resp. $A$ is convex), then the wavefront $\mathcal{W}_t$ is strictly convex (resp. convex), for all $t > 0$.
\end{theorem}
\begin{proof}
Let $f(t,z)$ and $\mathcal{W}_t$ be the wavemap and wavefront (from $A$) associated with $F_t$, respectively. Assume first that $A$ is convex but $\mathcal{W}_{t_0}$ is non-convex at some time $t_0 > 0$. We can construct another time-only dependent Finsler metric $\tilde{F}_t$ such that $F_t = \tilde{F}_t$ for all $t \in [0,t_0]$, but $\tilde{F}_t$ is constant for all $t \geq t_0+\varepsilon$, with $\varepsilon > 0$ arbitrarily small. This means that the wavefront $\tilde{\mathcal{W}}_{t}$ (from $A$) associated with $\tilde{F}_t$ is non-convex at $t_0$ and thus, choosing a sufficiently small $\varepsilon$, it also remains non-convex at $t_0+\varepsilon$. But since $\tilde{F}_t$ is constant beyond the time $t_0+\varepsilon$, cut points are guaranteed to appear at later times (because the trajectories of the wavemap are straight lines). Therefore, we have constructed a time-only dependent Finsler metric $\tilde{F}_t$ with a convex initial region $A$ whose wavemap has cut points, in contradiction with Theorem \ref{thm:no_cut_points}. We conclude that $\mathcal{W}_t$ must be convex for all $t > 0$.

Now, assume that $A$ is strictly convex. Then, in particular, $A$ is convex, so the above applies and $\mathcal{W}_t$ must be also convex for all $t > 0$. Assume, however, that $\mathcal{W}_{t_0}$ is not strictly convex at some time $t_0 > 0$. This means that $\mathcal{W}_{t_0}$ contains a straight segment $\sigma$. Let $\varphi: [c,d] \rightarrow A$ be a curve in $A$ such that $\sigma=f(t_0,\varphi[c,d])$ and take $s_0 \in (c,d)$. Consider the affine hyperplane $\eta$, tangent to $A$ at $\varphi(s_0)$, and the corresponding wavemap $\tilde{f}$ and wavefronts $\tilde{\mathcal{W}}_t$ from $\eta$. Since $\eta$ only intersects $A$ at $\varphi(s_0)$, due to the strict convexity of $A$, obviously $\tilde{\mathcal{W}}_{t_0}$ must enclose $\mathcal{W}_{t_0}$ (i.e. $\mathcal{W}_{t_0}$ must be on the side of $\tilde{\mathcal{W}}_{t_0}$ that contains $A$) and
\begin{equation}
\label{eq:intersection}
    \tilde{\mathcal{W}}_{t_0} \cap \mathcal{W}_{t_0}=\{f(t_0,\varphi(s_0))\}
\end{equation}
because, otherwise, we could construct a (non-lightlike) causal curve (as in \eqref{eq:causal_curve}) going from $\eta$ at $t=0$ to $\tilde{\mathcal{W}}_{t_0}$ at $t=t_0$, in contradiction with Proposition~\ref{prop:strictly_min}. Now, since $\tilde{\mathcal{W}}_{t_0}$ is an affine hyperplane parallel to $\eta$ (see the first part of the proof of Theorem~\ref{thm:no_cut_points}) that passes through $f(t_0,\varphi(s_0))\in \sigma$, then $\sigma \subset \tilde{\mathcal{W}}_{t_0}$ (otherwise, $\sigma \subset \mathcal{W}_{t_0}$ would be transverse to $\tilde{\mathcal{W}}_{t_0}$, and $\mathcal{W}_{t_0}$ would not be entirely enclosed by $\tilde{\mathcal{W}}_{t_0}$). However, this means that $\sigma \subset \tilde{\mathcal{W}}_{t_0} \cap \mathcal{W}_{t_0}$ and then, by \eqref{eq:intersection}, $\sigma$ must be a single point, which contradicts the assumption that $\mathcal{W}_{t_0}$ is not strictly convex.

Finally, if $A$ is a single point, at least for a small time the wavefront generated by the wavemap must be strictly convex (in fact, strongly convex; see \cite{JMPS}), as so is the indicatrix $\Sigma_A$. Taking this wavefront as the new initial region, we are back to the previous case and, therefore, we conclude that $\mathcal{W}_t$ must be strictly convex for all $t>0$.
\end{proof}

\subsection{A natural example: Shifted ellipses as indicatrices}
\label{subsec:ellipses}
To illustrate the concepts discussed so far, we will consider the simplest and most natural example of Zermelo navigation (very well-known in the literature; see \cite{BaoRob2,CJS2011,caponio2021a,JPS2023,javaloyes2021a,Garcia2025,richards1993a,shen2003a}). As we shall see, the time-only dependence in this setting will allow us to obtain the analytic expression of the wavemap, making this example an excellent choice for both theoretical and numerical analysis. In fact, the expressions derived here and in the following subsection will be used explicitly in Sections~\ref{subsec:counter_examples} and \ref{sec:examples} below.

\paragraph{General description}
Let $N=\mathbb{R}^2$ with the standard Euclidean metric $\langle \cdot,\cdot \rangle$ and norm $||\cdot||$. Suppose that the moving object has a self-propelled velocity $\Sigma^0$ that takes the form of a smooth field of (time-dependent) centered ellipses on $TN$. Then, these self-propelled velocities are displaced by a ``wind'' $W=(w_1,w_2)$ represented by a (time-dependent) vector field on $N$, generating the actual speed profile indicatrix $\Sigma$. Specifically:
\begin{itemize}
    \item Each $\Sigma^0_{(t,x)} \subset T_xN$ is an ellipse of semi-axes $a(t,x)$ and $b(t,x)$, rotated by an angle $\theta(t,x)$ in the clockwise direction, i.e. $\Sigma^0_{(t,x)} = \{v=(v^1,v^2)\in T_xN: E_t(v)=1\}$, where
    \begin{equation*}
        E_t(v^1,v^2)=\left(\frac{v^1\cos\theta(t,x)-v^2\sin\theta(t,x)}{a(t,x)}\right)^2 + \left(\frac{v^1\sin\theta(t,x)+v^2\cos\theta(t,x)}{b(t,x)}\right)^2.
    \end{equation*}

    \item Note that $E_t$ can be seen as the square of a (time-dependent) norm whose unit vectors are $\Sigma^0$. Since the Hessian of $E_t$ is direction-independent (namely, non-Finslerian), the corresponding scalar product is a (time-dependent) Riemannian metric $h_t$ on $N$:
    \begin{equation*}
        h_t(u,v) \coloneqq \textup{Hess}\left(\frac{1}{2} E_t\right) (u,v), \quad \forall u,v \in T_xN.
    \end{equation*}
    In matrix form, using natural Euclidean coordinates:
    \begin{equation}
    \label{eq:ht_ellipses}
    \{(h_t)_{ij}\} = \frac{1}{a^2 b^2}
    \left(\begin{array}{cc}
    a^2\sin^2\theta+b^2\cos^2\theta & (a^2-b^2)\sin\theta\cos\theta \\
    (a^2-b^2)\sin\theta\cos\theta & a^2\cos^2\theta+b^2\sin^2\theta
    \end{array}\right),
    \end{equation}
    and the following holds for any $t \in \mathbb{R}$ and $v \in T_xN$:
    \begin{itemize}
        \item $||v||_{h_t} \coloneqq \sqrt{h_t(v,v)} = \sqrt{E_t(v)}$,
        \item $||v||_{h_t} = 1$ if and only if $v \in \Sigma^0_{(t,x)}$.
    \end{itemize}
    
    \item The total velocity vectors, i.e the speed profile indicatrix field $\Sigma$, is then given by the translation of $\Sigma^0$ by $W$:
    \begin{equation*}
        \Sigma_{(t,x)} = \Sigma^0_{(t,x)}+W(t,x) = \{v+W(t,x) \in T_xN: v\in \Sigma_{(t,x)}^0\},
    \end{equation*}
    where we assume that the wind is never stronger than the self-propelled speed, i.e. $||W||_{h_t} <1$, so that $\Sigma$ always encloses the zero vector. Since $v-W(t,x) \in \Sigma^0_{(t,x)}$ for any $v \in \Sigma_{(t,x)}$, the time-dependent Finsler metric $F_t$ with indicatrix $\Sigma$ can be obtained by imposing the condition
    \begin{equation*}
        h_t \left( \frac{v}{F_t(v)}-W(t,x),\frac{v}{F_t(v)}-W(t,x) \right) = 1, \qquad \forall t \in \mathbb{R}, \quad \forall x \in N, \quad \forall v\in T_xN\setminus\{0\}.
    \end{equation*}
    Solving for $F_t(v)$, we obtain
    \begin{equation}
    \label{eq:zermelo_metric}
        F_t(v) = \frac{\omega_t(v)+\sqrt{\Lambda(t,x)h_t(v,v)+\omega_t(v)^2}}{\Lambda(t,x)},
    \end{equation}
    where $\omega_t \coloneqq -h_t(\cdot,W)$ is a (time-dependent) one-form on $N$, and $\Lambda \coloneqq 1-h_t(W,W)$ is a real function on $M$. This is usually referred to as the {\em Zermelo metric} in the literature (see e.g. \cite[Equation~(6)]{JPS2023}).
\end{itemize}

Let us assume that the initial region $A$ is a compact embedded hypersurface of $N$ (not a single point). Since $n=2$, this means that we can explicitly parametrize $A$ by a simple closed curve $\alpha: I \subset \mathbb{R} \rightarrow A$, with $\alpha(I)=A$ and $\alpha'(s)\not=0$ for all $s\in I$. Then, the wavemap \eqref{eq:wavemap} can be expressed as 
\begin{equation}
\label{eq:wavemap_alpha}
    \begin{array}{rrll}
    f \colon & [0,\infty)\times I & \longrightarrow & N\\
    & (t,s) & \longmapsto & f(t,s)=(x^1(t,s),x^2(t,s)),
    \end{array}
\end{equation}
with $f(0,s)=\alpha(s)\in A$, and the curves $s \mapsto f(t,s)$ parametrize the wavefronts $\mathcal{W}_t$ for any $t < \tau$, where $\tau$ is the time given by Proposition~\ref{prop:cut_function}. Following Notation~\ref{not:vel_wavemap}, we will denote by $\partial_tf(t,s)=(\partial_tx^1,\partial_tx^2)$ and $\partial_sf(t,s)=(\partial_sx^1,\partial_sx^2)$ the velocity vectors of the curves $t \mapsto f(t,s)$ and $s \mapsto f(t,s)$, respectively. In this case, it is straightforward to check that the orthogonality conditions \eqref{eq:ort_F}, written in terms of the so-called {\em Zermelo data} $(h_t,W)$---which contain all the information about this particular Zermelo's problem---become the following conditions in $[0,\tau) \times I$ (see \cite[Propostion~5.4]{javaloyes2021a}):
\begin{equation}
\label{eq:orth_W}
    \left\lbrace{
    \begin{array}{l}
    ||\partial_tf-W||_{h_t} = 1, \\
    \partial_tf-W \bot_{h_t} \partial_sf, \quad \textup{ with } \partial_tf \textup{ pointing outward},
    \end{array}
    }\right.
\end{equation}
and they can be explicitly expressed as the following PDE system (see \cite[Theorem~5.5]{javaloyes2021a}):
\begin{equation}
    \label{eq:pde_ellipses}
    \begin{split}
    \partial_tx^1 & = \pm \frac{a^2\cos\theta(\partial_sx^1\sin\theta+\partial_sx^2\cos\theta)-b^2\sin\theta(\partial_sx^1\cos\theta-\partial_sx^2\sin\theta)}{\sqrt{a^2(\partial_sx^1\sin\theta+\partial_sx^2\cos\theta)^2+b^2(\partial_sx^1\cos\theta-\partial_sx^2\sin\theta)^2}} + w_1, \\
    \partial_tx^2 & = \pm \frac{-a^2\sin\theta(\partial_sx^1\sin\theta+\partial_sx^2\cos\theta)-b^2\cos\theta(\partial_sx^1\cos\theta-\partial_sx^2\sin\theta)}{\sqrt{a^2(\partial_sx^1\sin\theta+\partial_sx^2\cos\theta)^2+b^2(\partial_sx^1\cos\theta-\partial_sx^2\sin\theta)^2}} + w_2,
    \end{split}
\end{equation}
where $\pm$ has to be chosen so that $\partial_tf$ points outward ($+$ when $A$ is counter-clockwise parametrized, and $-$ otherwise).

\paragraph{Time-only dependent case} When $a(t,x)=a(t)$, $b(t,x)=b(t)$, $\theta(t,x)=\theta(t)$ and $W(t,x)=W(t)$, the orthogonality conditions \eqref{eq:orth_W} are valid in $[0,\infty) \times I$, thanks to Theorem~\ref{thm:no_cut_points}. Moreover, Remark~\ref{rem:killing} tells us that each $\partial_sf(0,s)=\alpha'(s)$ is Killing (when extended as a constant vector field on $N$), which implies that the second condition in \eqref{eq:orth_W} reduces to $\partial_tf(t,s)-W(t) \bot_{h_t} \partial_sf(0,s)$. Therefore, we can replace $\partial_sf(t,s)$ with $\partial_sf(0,s)=\alpha'(s)$ in \eqref{eq:pde_ellipses} and the resulting PDE system can be integrated directly (see \cite[Equations (21) and (22)]{richards1993a}):
\begin{equation}
\label{eq:solution_ellipses}
    \begin{split}
        x^1(t,s) & = x^1(0,s) \pm \int_0^t \Bigl( a(r)\cos{\theta(r)}\cos{\phi(r,s)} + b(r)\sin{\theta(r)}\sin{\phi(r,s)} + w_1(r) \Bigr) \ \mathrm{d}r, \\
        x^2(t,s) & = x^2(0,s) \pm \int_0^t \Bigl( -a(r)\sin{\theta(r)}\cos{\phi(r,s)} + b(r)\cos{\theta(r)}\sin{\phi(r,s)} + w_2(r) \Bigr) \ \mathrm{d}r,
    \end{split}
\end{equation}
where
\begin{equation*}
    \cos{\phi(t,s)} \coloneqq \frac{C(t,s)}{\sqrt{C(t,s)^2+D(t,s)^2}}, \quad \sin{\phi(t,s)} \coloneqq \frac{C(t,s)}{\sqrt{C(t,s)^2+D(t,s)^2}},
\end{equation*}
and
\begin{equation*}
\begin{split}
    C(t,s) & \coloneqq a(t)(\partial_s x^1(0,s)\sin{\theta(t)} + \partial_s x^2(0,s)\cos{\theta(t)}), \\
    D(t,s) & \coloneqq b(t)(-\partial_s x^1(0,s)\cos{\theta(t)} + \partial_s x^2(0,s)\sin{\theta(t)}).
\end{split}
\end{equation*}

In order to obtain a simplified version of these solutions, we consider next a particular case.

\subsection{A particular case: Shifted circles as indicatrices}
\label{subsec:circles}
When the self-propelled velocity is isotropic, the indicatrix $\Sigma$ is a smooth distribution of displaced circles, which greatly simplifies the expressions in Section~\ref{subsec:ellipses}.

\paragraph{General description} The self-propelled velocity $\Sigma_{(t,x)}^0 \subset T_xN$ is now a centered circle of radius $R(t,x)$---representing the self-propelled isotropic speed. Choosing $R(t,x)=a(t,x)=b(t,x)$ and $\theta(t,x)=0$ in $\eqref{eq:ht_ellipses}$, we obtain
\begin{equation*}
    \{(h_t)_{ij}\} = \frac{1}{R(t,x)^2}
    \left(\begin{array}{cc}
    1 & 0 \\
    0 & 1
    \end{array}\right),
\end{equation*}
which means that $h_t = \frac{1}{R^2}\langle\cdot,\cdot\rangle$ is conformal to the standard Euclidean metric. Then, the PDE system \eqref{eq:pde_ellipses} that characterizes the wavemap \eqref{eq:wavemap_alpha} (with $A$ parametrized by a closed curve $\alpha: I \rightarrow A$) takes the following form:
\begin{equation*}
    \begin{split}
    \partial_tx^1 & = \pm \frac{R \ \partial_sx^2}{\sqrt{(\partial_sx^1)^2+(\partial_sx^2)^2}} + w_1, \\
    \partial_tx^2 & = \pm \frac{-R \ \partial_sx^1}{\sqrt{(\partial_sx^1)^2+(\partial_sx^2)^2}} + w_2,
    \end{split}
\end{equation*}
where, as above, $+$ is chosen when $A$ is counter-clockwise parametrized, and $-$ otherwise.

\paragraph{Time-only dependent case} When $R(t,x)=R(t)$ and $W(t,x)=W(t)$, the analytic solution of the wavemap \eqref{eq:solution_ellipses} reduces to
\begin{equation*}
    \begin{split}
        x^1(t,s)=x^1(0,s) \pm \frac{\partial_sx^2(0,s)}{\sqrt{(\partial_sx^1(0,s))^2+(\partial_sx^2(0,s))^2}}\int_0^t R(r) \ \mathrm{d}r + \int_0^t w_1(r) \ \mathrm{d}r, \\
        x^2(t,s)=x^2(0,s) \pm \frac{-\partial_sx^1(0,s)}{\sqrt{(\partial_sx^1(0,s))^2+(\partial_sx^2(0,s))^2}}\int_0^t R(r) \ \mathrm{d}r + \int_0^t w_2(r) \ \mathrm{d}r.
    \end{split}
\end{equation*}
We can simplify this even further by selecting a counter-clockwise parametrization $\alpha$ of $A$ such that $(\partial_sx^1(0,s))^2+(\partial_sx^2(0,s))^2=1 $. In this case we obtain
\begin{equation}
\label{eq:time_only_W}
    \begin{split}
        x^1(t,s) & =x^1(0,s) + \partial_sx^2(0,s)\int_0^t R(r) \ \mathrm{d}r + \int_0^t w_1(r) \ \mathrm{d}r, \\
        x^2(t,s) &= x^2(0,s) - \partial_sx^1(0,s)\int_0^t R(r) \ \mathrm{d}r + \int_0^t w_2(r) \ \mathrm{d}r,
    \end{split}
\end{equation}
with $f(0,s) = (x^1(0,s),x^2(0,s)) = \alpha(s) \in A$ and $\partial_sf(0,s) = (\partial_sx^1(0,s),\partial_sx^2(0,s)) = \alpha'(s) \in T_{\alpha(s)}A$ such that $||\alpha'(s)||=1$.

\paragraph{Point-like initial region}
When the initial region $A$ is a single point in $N$, we can replace $\alpha'(s)=\partial_sf(0,s)$ in \eqref{eq:time_only_W} with any curve that parametrizes all the unitary directions in $T_AN$. For instance, we can choose the parametrization $(-\sin(\theta),\cos(\theta))$, with $\theta \in [0,2\pi)$, to obtain
\begin{equation}
\label{eq:time_only_point}
    \begin{split}
        x^1(t,\theta) & =x_0^1 + \cos(\theta)\int_0^t R(r) \ \mathrm{d}r + \int_0^t w_1(r) \ \mathrm{d}r, \\
        x^2(t,\theta) &= x_0^2 + \sin(\theta)\int_0^t R(r) \ \mathrm{d}r + \int_0^t w_2(r) \ \mathrm{d}r,
    \end{split}
\end{equation}
where $x_0 \coloneqq (x^1_0,x^2_0) = A \in N$. However, sometimes it is more convenient to write these equations in terms of the initial velocity $v_0(\theta) \coloneqq \partial_tf(0,\theta)$. Taking derivatives with respect to $t$ in \eqref{eq:time_only_point}, we get
\begin{equation*}
\begin{split}
    v_0^1(\theta) \coloneqq \partial_tx^1(0,\theta) & = \cos(\theta) R(0)+w_1(0), \\
    v_0^2(\theta) \coloneqq \partial_tx^2(0,\theta) & = \sin(\theta) R(0)+w_2(0),
\end{split}
\end{equation*}
which we can solve for $\cos(\theta)$ and $\sin(\theta)$, and then substitute in \eqref{eq:time_only_point}, obtaining
\begin{equation*}
    \begin{split}
        x^1(t,\theta) & = x^1_0 + \frac{v_0^1(\theta)-w_1(0)}{R(0)} \int_0^t R(r) \ \mathrm{d}r + \int_0^t w_1(r) \ \mathrm{d}r, \\
        x^2(t,\theta) & = x^2_0 + \frac{v_0^2(\theta)-w_2(0)}{R(0)} \int_0^t R(r) \ \mathrm{d}r + \int_0^t w_2(r) \ \mathrm{d}r.
    \end{split}
\end{equation*}
These equations are valid for any parametrization of the initial velocities $v_0(\theta)=(v_0^1(\theta),v_0^2(\theta))\in T_AN$ such that $F_0(v_0(\theta))=1$ or, equivalently, $||v_0(\theta)-W(0)||_{h_0}=1$.

\section{Non-convex Zermelo navigation}
\label{sec:non-convex}

When introducing Zermelo navigation theoretically in Section~\ref{sec:review}, the strong convexity of the speed profile indicatrix---condition (iv) from Section~\ref{subsec:finsler_metrics}---emerged as a key property that allowed us to define geodesic curves in the usual way and, through them, obtain solutions to Zermelo's problem. In fact, this condition has always been assumed in the Finslerian modeling of Zermelo navigation. However, as pointed out in the introduction, there are real-world scenarios---particularly in sailboat navigation and in the flight of seabirds such as albatrosses---where this condition fails. In such cases, the optimal paths are not smooth Finsler or Lorentz-Finsler pregeodesics but rather zig-zag trajectories with tacking points (see e.g. \cite[Figures~5 and 6]{buell1996a} and \cite[Figures~181 and 243]{marchaj1982a} for sailboat navigation, or \cite[Figure~1]{Richardson2018}, \cite[Figure~1]{sachs2005a} and \cite[Figures~2 and 3]{sachs2016a} for the albatross flight).

Motivated by this, our goal here is to generalize Zermelo navigation to situations where the geometric shape formed by the velocity vectors is non-convex. To the best of our knowledge, this has not been previously studied within the Finslerian framework, so we begin our analysis in this section by restricting ourselves to the simplest case:
\begin{itemize}
    \item Every element of Zermelo navigation will be assumed to be constant, i.e. time- and position-independent.
    \item Accordingly, we will work in $ N=\mathbb{R}^n $ as an affine space, with its standard Euclidean norm $ ||\cdot|| $. In particular, there will be no need to work at the level of the spacetime $M$.
    \item Both the initial and arrival regions $A,B$ will be assumed to be single points.
\end{itemize}

\subsection{Norms and travel time}
\label{subsec:norms}
In general, norms on real vector spaces are assumed to satisfy the triangle inequality, which allows them to be defined by selecting a convex indicatrix (see e.g. \cite{JS2014}). The following definitions relax this notion.

\begin{definition}
\label{def:topological}
Let $ \Sigma $ be a topological $ (n-1) $-sphere in $ \mathbb{R}^n $ such that every ray departing from the origin intersects $ \Sigma $ exactly once. Then, the continuous function $ ||\cdot||_{\Sigma}: \mathbb{R}^n \rightarrow \mathbb{R} $, where $ ||v||_{\Sigma} $ is the unique positive number such that $ v/||v||_{\Sigma} \in \Sigma $, will be called a {\em norm} with {\em indicatrix} $ \Sigma $ (see Figure~\ref{fig:norm1}).
\end{definition}

\begin{remark}
Note that $ \Sigma = \{ v \in \mathbb{R}^n: ||v||_{\Sigma} = 1 \} $, and $ ||\cdot||_{\Sigma} $ is
\begin{itemize}
\item Positive: $ ||v||_{\Sigma} \geq 0 $, with equality if and only if $ v = 0 $.
\item Positive homogeneous of degree 1: $ ||\lambda v||_{\Sigma} = \lambda ||v||_{\Sigma} $ for all $ \lambda > 0 $ and all $v \in \mathbb{R}^n$.
\end{itemize}
\end{remark}

As we have seen in Section~\ref{subsec:finsler_metrics}, the convexity of $\Sigma$ plays a fundamental role in Zermelo navigation. We say that $ \Sigma $ is {\em convex} or {\em strictly convex} if the closed unit ball $ \{ v \in \mathbb{R}^n: ||v||_{\Sigma} \leq 1 \} $ is convex or strictly convex, respectively, as a subset of $ \mathbb{R}^n $.

\begin{definition}
A norm $ ||\cdot||_{\Sigma} $ with indicatrix $ \Sigma $ will be called
\begin{itemize}
\item A {\em convex norm} if $ \Sigma $ is convex. This is equivalent to saying that $ ||\cdot||_{\Sigma} $ satisfies the triangle inequality: $ ||v+u||_{\Sigma} \leq ||v||_{\Sigma}+||u||_{\Sigma} $ for all $v,u \in \mathbb{R}^n$.
\item A {\em strictly convex norm} if $ \Sigma $ is strictly convex. This is equivalent to saying that $ ||\cdot||_{\Sigma} $ satisfies the strict triangle inequality: $ ||v+u||_{\Sigma} \leq ||v||_{\Sigma}+||u||_{\Sigma} $, with equality if and only if $ v = \lambda u $ or $ u = \lambda v $ for some $ \lambda \geq 0 $.
\item A {\em Minkowski norm} if $ \Sigma $ is smooth (at least $ C^2 $) and strongly convex (i.e. it has positive sectional curvature everywhere).
\end{itemize}
\end{definition}

\begin{remark}
    Observe that a Minkowski norm satisfies all the conditions (i)-(iv) from Section~\ref{subsec:finsler_metrics}, so it is effectively a constant Finsler metric, as we have anticipated in Definition~\ref{def:finsler_metrics}.
\end{remark}

Assume that $ \Sigma $ represents the velocity vectors in Zermelo's problem (one specific velocity for each direction), i.e. the moving object travels a distance $ ||v|| $ in one time unit by following a straight line in the direction of $ v \in \Sigma $. If $ A, B $ are two points in $ \mathbb{R}^n $, observe then that $ ||B-A||_{\Sigma} $ represents the travel time going straight from $ A $ to $ B $ with velocity $ \Sigma $. In general, if $\gamma: [a,b] \rightarrow \mathbb{R}^n$ is a (piecewise smooth) curve from $A$ to $B$, the travel time is given by
\begin{equation*}
    \mathcal{T}[\gamma] = \int_a^b ||\gamma'(s)||_{\Sigma} \ \mathrm{d}s.
\end{equation*}
Note that, in contrast with the time-dependent case, where the travel time had to be defined for a specific set of curves in the spacetime (recall Definition~\ref{def:travel_functional}), now every curve in $\mathbb{R}^n$ is a possible trajectory for the moving object. This provides a direct way to define the optimal paths in this setting, analogously to Definition~\ref{def:time_min_wavemap}.

\begin{definition}
    Let $A, B \in \mathbb{R}^n$. Choosing a norm $||\cdot||_{\Sigma}$, we say that a (regular) piecewise smooth curve $\gamma$ from $A$ to $B$ is {\em time-minimizing} (respectively, {\em strictly time-minimizing}) if $\mathcal{T}[\gamma] \leq \mathcal{T}[\varphi]$ (respectively, $\mathcal{T}[\gamma] < \mathcal{T}[\varphi]$) for any other piecewise smooth curve $\varphi$ from $A$ to $B$.
\end{definition}

When $ \Sigma $ is strictly convex, the strict triangle inequality ensures that the straight line between two points is always strictly time-minimizing (i.e. any other trajectory with velocity $ \Sigma $ takes more travel time). This is also an obvious conclusion of the classical Zermelo navigation, since straight lines are Finsler pregeodesics in the constant case. When $ \Sigma $ is convex, the straight line is still time-minimizing, but unicity is lost: there might be other trajectories with the same travel time. When $ \Sigma $ is non-convex, then we cannot ensure that the straight path is the fastest one. We will now focus our study on this last case.

\subsection{Multi-convex norms}
\label{subsec:multi-convex}
In principle, Definition~\ref{def:topological} allows for shapes that are too general for our purposes here. To endow the indicatrix with a more concrete structure---one that captures the real-world cases we aim to model---we will restrict ourselves to the following notion. In what follows, $ \widehat{\Sigma} \coloneqq \partial CH(\Sigma) $ will denote the boundary of the convex hull of $ \Sigma $.

\begin{definition}
A topological $ (n-1) $-sphere $ \Sigma $ in $ \mathbb{R}^n $ is {\em multi-convex} if it is the union of a finite number of hypersurfaces, each one a compact $C^1$ patch of the indicatrix of a strictly convex norm. This defines a norm $ ||\cdot||_{\Sigma} $ that will be called {\em multi-convex} (see Figure~\ref{fig:norms}).
\end{definition}

\begin{figure}
\centering
\begin{subfigure}{0.32\textwidth}
\includegraphics[width=\textwidth]{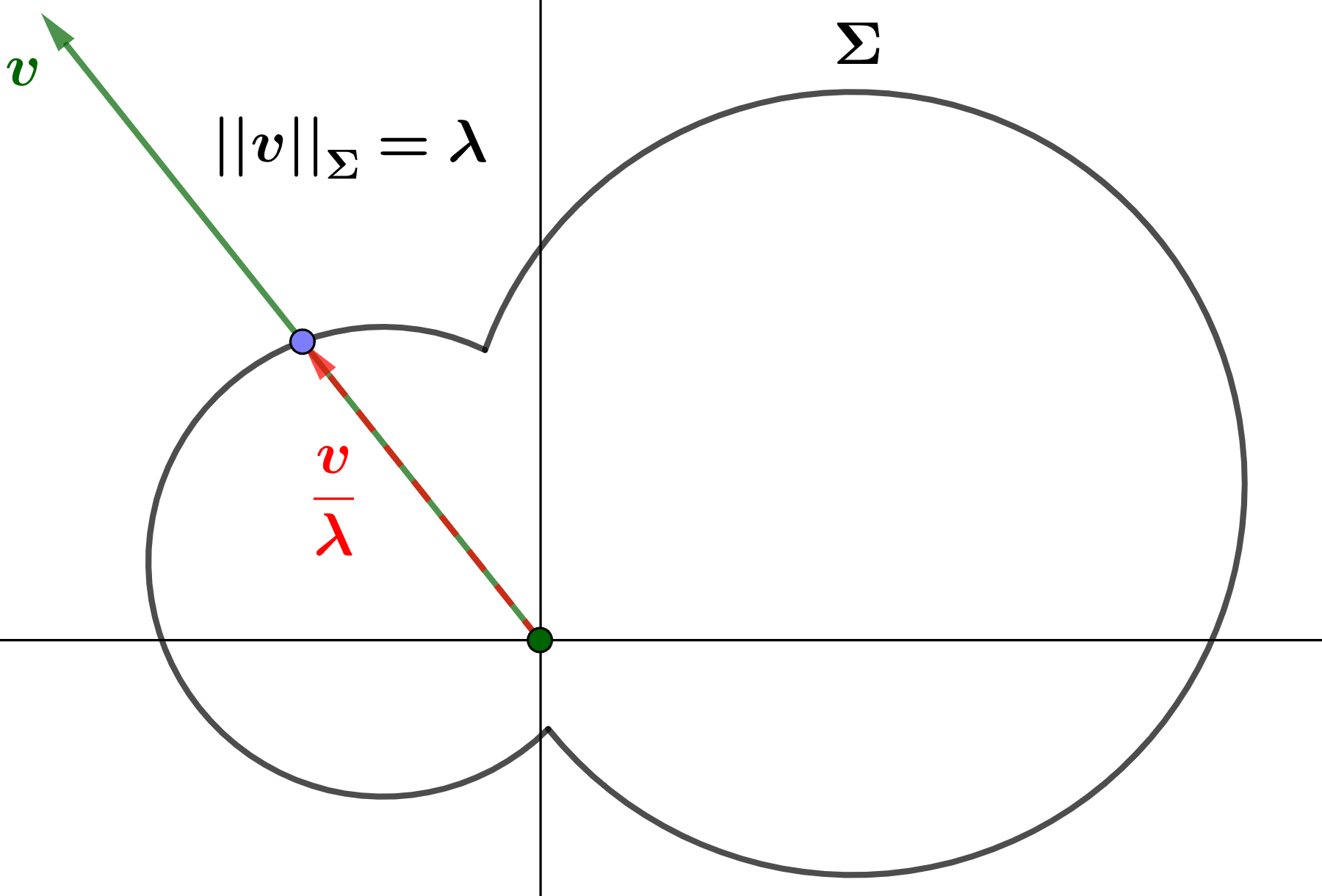}
\caption{Norm with indicatrix $\Sigma$.}
\label{fig:norm1}
\end{subfigure}
\begin{subfigure}{0.32\textwidth}
\includegraphics[width=\textwidth]{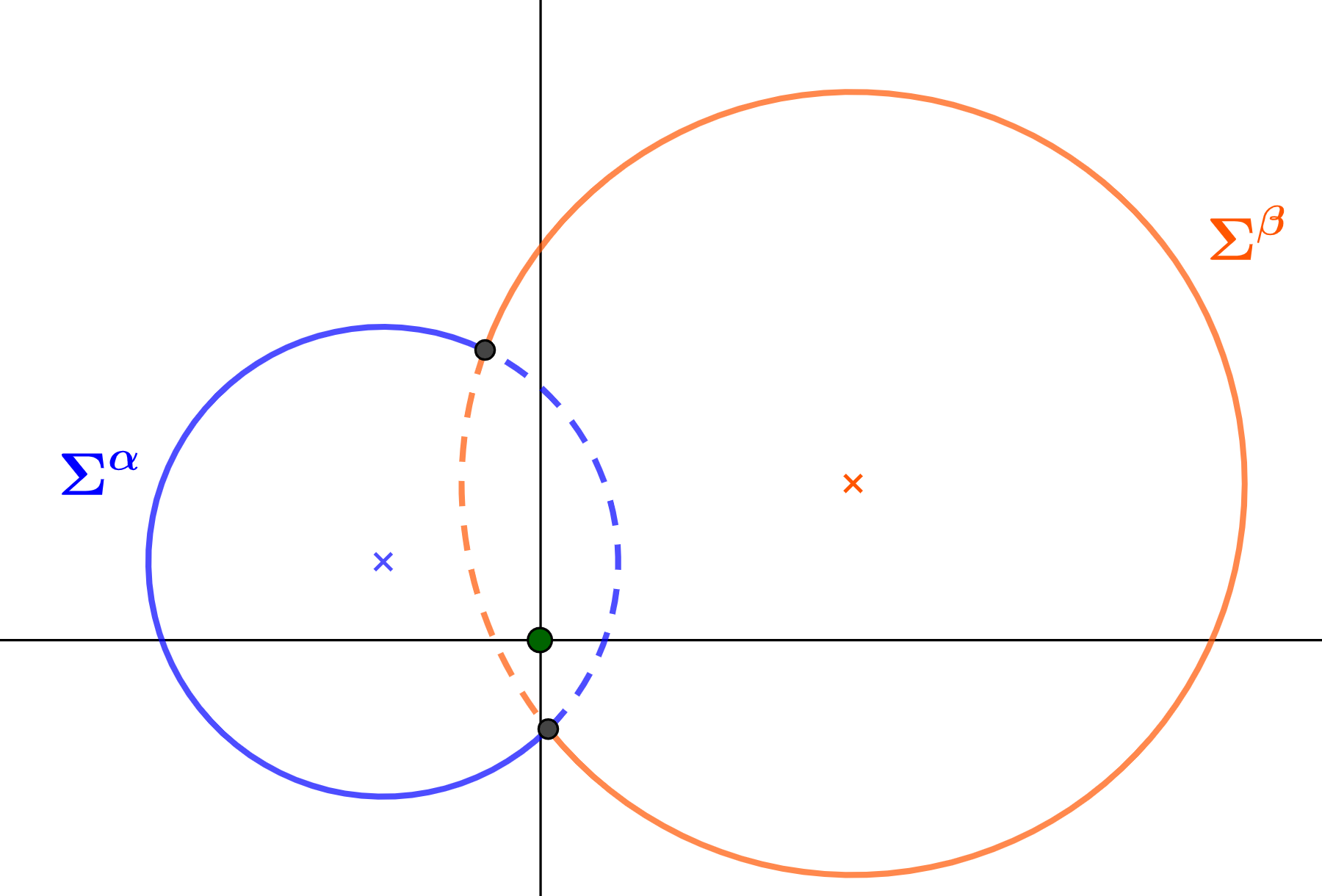}
\caption{Multi-convex indicatrix.}
\label{fig:norm2}
\end{subfigure}
\begin{subfigure}{0.32\textwidth}
\includegraphics[width=\textwidth]{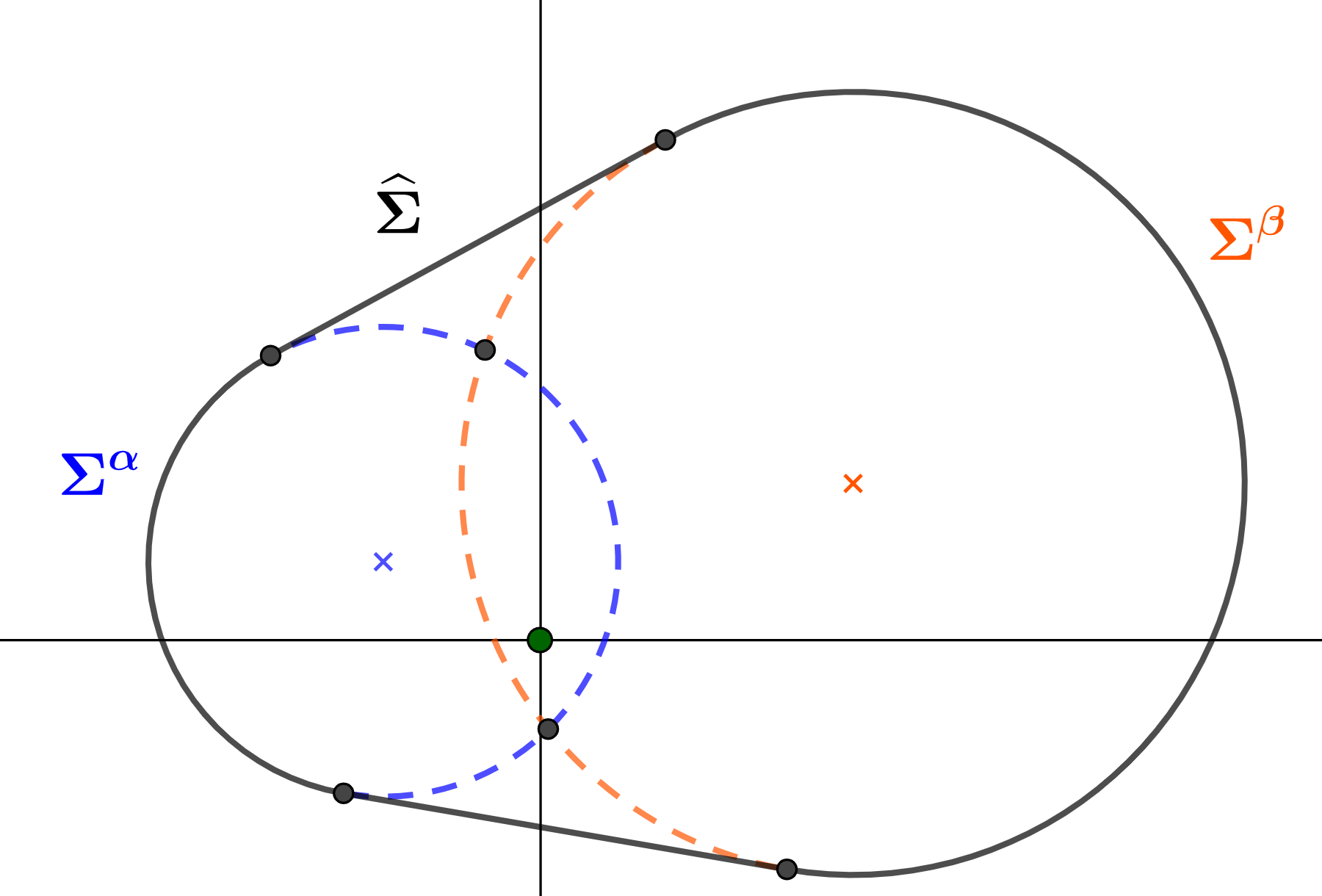}
\caption{Boundary of the convex hull.}
\label{fig:norm3}
\end{subfigure}
\caption{A norm $||\cdot||_{\Sigma}$ in $\mathbb{R}^2$ with a multi-convex indicatrix composed of two strictly convex indicatrices $\Sigma^{\alpha},\Sigma^{\beta}$. The boundary of the convex hull of $\Sigma$, denoted $\widehat{\Sigma}$, is shown in the last figure.}
\label{fig:norms}
\end{figure}

\begin{remark}
    Given the indicatrix $ \Sigma $ of a multi-convex norm, note that $ ||\cdot||_{\widehat{\Sigma}} $ is a convex norm. Moreover, since the (finitely many) indicatrices that make up $ \Sigma $ are smooth, $ \widehat{\Sigma} $ can only have a finite number of non-smooth points. These points will be called {\em cusps}.
\end{remark}

When the velocity is multi-convex, the solution to Zermelo's problem is always a piecewise segment, in the following sense.

\begin{definition} \label{def:Tacking}
A continuous curve $\gamma$ in $\mathbb{R}^n$ is called a {\em piecewise segment} if it is composed of a finite number of straight segments. Any break along $\gamma$ (i.e. any non-smooth point at the intersection of two segments) will be called a {\em tack point}.\footnote{We borrow the notion of {\emph{tacking}} from the ancient art of maneuvering a sailboat by zig-zagging in upwind (and downwind) directions.}
\end{definition}

\begin{lemma}
\label{lem:segments}
    Let $||\cdot||_{\Sigma}$ be a multi-convex norm, and fix $A,B \in \mathbb{R}^n$. If $\gamma$ is a time-minimizing trajectory from $A$ to $B$, then it must be a piecewise segment.
\end{lemma}
\begin{proof}
    Assume that $\gamma$ is a time-minimizing trajectory which is not a piecewise segment. Since $\gamma$ is piecewise smooth by definition, this means that $\gamma$ contains a smooth curve which is not a straight segment. In fact, by continuity, this smooth curve must contain in turn another (possibly shorter) smooth curve $\tilde{\gamma}=\gamma|_{[c,d]}$ whose velocity vector $\tilde{\gamma}'$ always belongs to one (and only one) of the strictly convex patches that compose $\Sigma$. Then, take a new curve $\varphi$ to be exactly the same as $\gamma$ from $A$ to $\gamma(c)$ and from $\gamma(d)$ to $B$, but replacing $\tilde{\gamma}$ with the straight segment $\tilde{\varphi}$ from $\gamma(c)$ to $\gamma(d)$. Note that, since $\Sigma$ is constant, $\tilde{\varphi}'$ must belong to the same strictly convex patch as $\tilde{\gamma}'$, so by the strict triangle inequality $\mathcal{T}[\tilde{\varphi}]<\mathcal{T}[\tilde{\gamma}]$ and thus, $\mathcal{T}[\varphi]<\mathcal{T}[\gamma]$, which contradicts the fact that $\gamma$ is time-minimizing.
\end{proof}

\begin{theorem}
\label{thm:time_min}
Let $ ||\cdot||_{\Sigma} $ be a multi-convex norm composed of $m \geq 1$ strictly convex norms, and fix $ A,B \in \mathbb{R}^n $. The ray from $A$ to $B$ intersects $ A+\widehat{\Sigma} = \{A+v\in \mathbb{R}^n: v\in \widehat{\Sigma}\} $ at a unique point $ Q \in \widehat{\Sigma} $. Then, if we travel from $A$ to $B$ with velocity $\Sigma$:
\begin{itemize}
    \item[(i)] If $ Q $ is a cusp, then $Q \in \Sigma$ and the straight line is strictly time-minimizing.
    \item[(ii)] Otherwise, the tangent plane $ T_Q\widehat{\Sigma} $ intersects $ \Sigma $ at a finite number of points $ \{Q_i\}_{i=1}^k $, with $ 1 \leq k \leq m $, and a curve from $A$ to $B$ is time-minimizing if and only if it is a piecewise segment with velocities given by the position vectors of $ \{Q_i\}_{i=1}^k \subset \Sigma $.
\end{itemize}
\end{theorem}
\begin{proof}
Assume first that $Q$ is not a cusp and note that the properties of the convex hull guarantee that the intersection between $ T_Q\widehat{\Sigma} $ and $ \Sigma $ is never empty, and the strict convexity of the (finitely many) indicatrices that form $ \Sigma $ ensures that this intersection is composed of a finite number of isolated points (at most $m$, since each intersection point must be on a different indicatrix). Observe also that $ T_Q\widehat{\Sigma} $ divides $ \mathbb{R}^n $ into two closed sets, one containing $ \widehat{\Sigma} $ entirely (since it is convex). Thanks to Lemma~\ref{lem:segments}, in order to prove (ii) it suffices to show that if we travel from $A$ to $B$ with velocity $ T_Q\widehat{\Sigma} $, then any piecewise segment spends the same travel time. This way, all the piecewise segments with velocity $ T_Q\widehat{\Sigma} \cap \Sigma $ share the same travel time, and any other (piecewise smooth) curve with velocity $\Sigma$ is slower.
 
Let us show this first for $ n = 2 $. In this case, $ T_Q\widehat{\Sigma} $ is a straight line. Let $ \gamma $ be the straight trajectory from $A$ to $B$, with velocity $v$ given by the position vector of $Q$, and let $ \varphi $ be any other piecewise segment with velocity $ T_Q\widehat{\Sigma} $. Assume that $ \varphi $ is composed of two segments with velocities $ v_1, v_2 $. This means that there is only one tack point $ p $ and let $ \tilde{p} $ be the intersection point between $ \gamma $ and the straight line parallel to $ T_Q\widehat{\Sigma} $ at $ p $ (see Figure \ref{fig:proof}). The travel time of $ \varphi $ is then
\begin{equation*}
\mathcal{T}[\varphi] = \frac{||p-A||}{||v_1||} + \frac{||B-p||}{||v_2||} = \frac{||\tilde{p}-A||}{||v||} + \frac{||B-\tilde{p}||}{||v||} = \mathcal{T}[\varphi],
\end{equation*}
and since $ \varphi $ is arbitrary, we conclude that every piecewise segment with one tack point and velocity $ T_Q\widehat{\Sigma} $ coincides in travel time with the straight line. Moreover, assume now that $ \{p_i\}_{i=1}^l $ are the tack points of $ \varphi $ ($ l+1 $ segments), with $ l > 1 $. Applying the above result for one tack point, the travel time of the straight line from $ A $ to $ p_2 $ coincides with the travel time of the part of $ \varphi $ from $ A $ to $ p_2 $. So, we can reduce $ \varphi $ to a trajectory $ \tilde{\varphi} $ with tack points $ \{p_i\}_{i=2}^l $ (one less than $ \varphi $) and $ \mathcal{T}[\varphi] = \mathcal{T}[\tilde{\varphi}] $. Applying this inductively to every successive tack point, we obtain that $ \mathcal{T}[\varphi] = \mathcal{T}[\gamma] $.

For $ n > 2 $, again let $ \gamma $ be the straight line and $ \varphi $ any piecewise segment with only one tack point $ p $ and velocities $ v_1, v_2 \in T_Q\widehat{\Sigma} $. Then, the points $ A, p, \tilde{p} $ form a two-dimensional plane containing $ \gamma $ and $ \varphi $, whose intersection with $ T_Q\widehat{\Sigma} $ is a straight line that contains the points $ A+v_1, A+v_2 $ and $ Q $. Therefore, we can apply the result for $ n = 2 $ and conclude that $ \mathcal{T}[\varphi] = \mathcal{T}[\gamma] $. If $ \varphi $ has more than one tack point, we can proceed inductively as before, reducing $ \varphi $ each time to a trajectory with one less tack point but the same travel time. Therefore, all the piecewise segments from $ A $ to $ B $ with velocity $ T_Q\widehat{\Sigma} $ share the same travel time.

Finally, if $Q$ is a cusp, then $Q$ is the intersection point of two of the strictly convex indicatrices that make up $\Sigma$, so $Q \in \Sigma$. Also, being $Q$ a cusp and taking into account the convexity of $\widehat{\Sigma}$, we can always select a hyperplane $\Pi$ such that $\Pi \cap \Sigma = Q$. So, $\Pi$ can effectively play the role of $T_Q\widehat{\Sigma}$ above, which means, following the same reasoning, that all the piecewise segments from $A$ to $B$ with velocity $\Pi$ share the same travel time. Since now $\Pi \cap \Sigma = Q$, the straight line (with velocity given by the position vector of $Q$) is faster than any other (piecewise smooth) curve with velocity $\Sigma$.
\end{proof}

\begin{figure}
\centering
\includegraphics[width=0.6\textwidth]{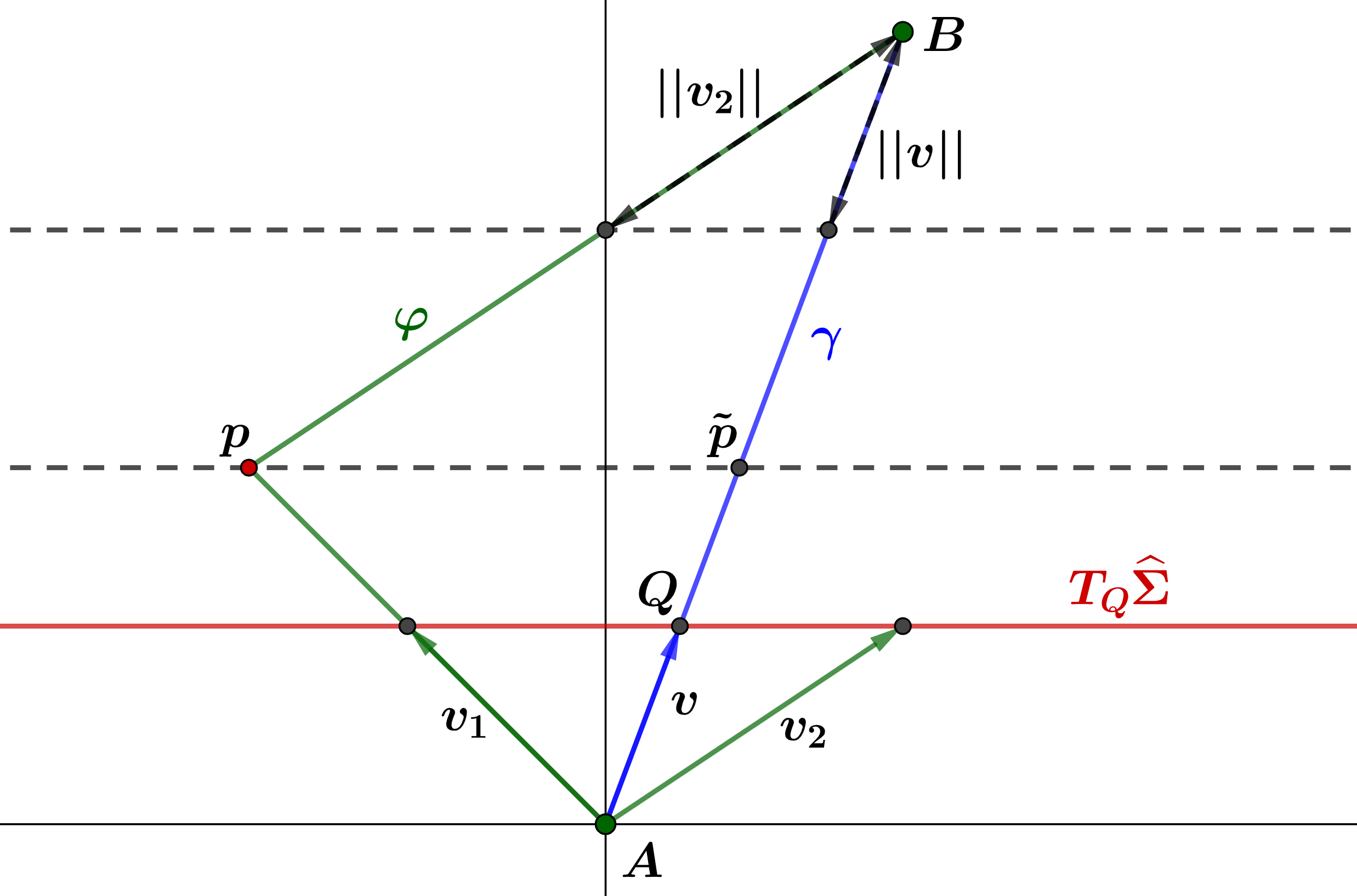}
\caption{Graphic representation of the proof of Theorem \ref{thm:time_min} for $ n = 2 $ and one tack point $ p $. When the velocity is given by $ T_Q\widehat{\Sigma} $, the travel times along the piecewise segment $ \varphi $ and the straight line $ \gamma $ coincide.}
\label{fig:proof}
\end{figure}

In particular, note that if $ k = 1 $ in (ii), then $ Q_1 = Q \in \Sigma $ and the straight line is strictly time-minimizing. Figures \ref{fig:sailboat} and \ref{fig:albatross} show this result in $ \mathbb{R}^2 $ and $ \mathbb{R}^3 $, respectively. These type of indicatrices can effectively model, for instance, the speed polar diagrams typical of sailboat navigation (see e.g. \cite[Figure~16.2]{larsson2000a}, \cite[Figures~179, 180 and 182]{marchaj1982a} and \cite[Figures~16 and 17]{pueschl2018a}) or the two-dimensional speed profiles of the wandering albatross (see \cite[Figures~5 and 8]{Richardson2018}), as well as its three-dimensional flight (see \cite[Figure~1]{Richardson2018}, \cite[Figure~1]{sachs2005a} and \cite[Figure~2]{sachs2016a}).

\begin{figure}
\centering
\includegraphics[width=0.6\textwidth]{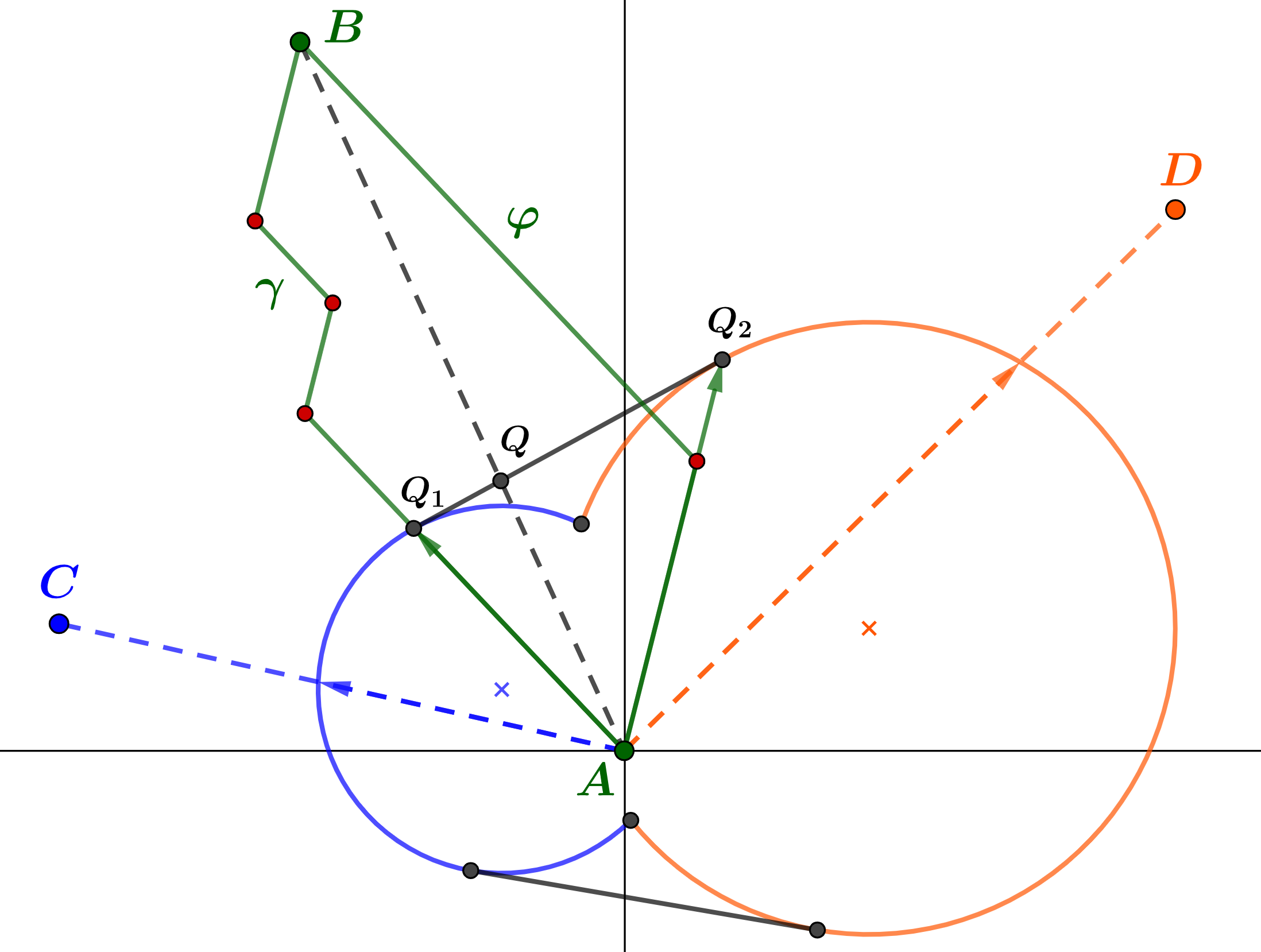}
\caption{Illustration of Theorem~\ref{thm:time_min} in $\mathbb{R}^2$ for the multi-convex indicatrix in Figure~\ref{fig:norms}. The strictly time-minimizing trajectory from $A$ to $ C $ or $ D $ is the straight line. However, $ \gamma $ and $ \varphi $ are two examples of time-minimizing trajectories from $A$ to $ B $ (their travel times are the same), the former with three tack points, the latter with only one (depicted as red dots). Each segment of $ \gamma $ and $ \varphi $ is traveled with velocities given by the position vectors of $ \{Q_1, Q_2\} $, which are the intersection points between the tangent line $ T_Q\widehat{\Sigma} $ and the multi-convex indicatrix. Any other combination of segments in these directions from $A$ to $B$ coincides in travel time with $ \gamma $ and $ \varphi $. In particular, there exist time-minimizing trajectories from $A$ to $B$ with any number of tack points.}
\label{fig:sailboat}
\end{figure}

\begin{figure}
\centering
\includegraphics[width=0.6\textwidth]{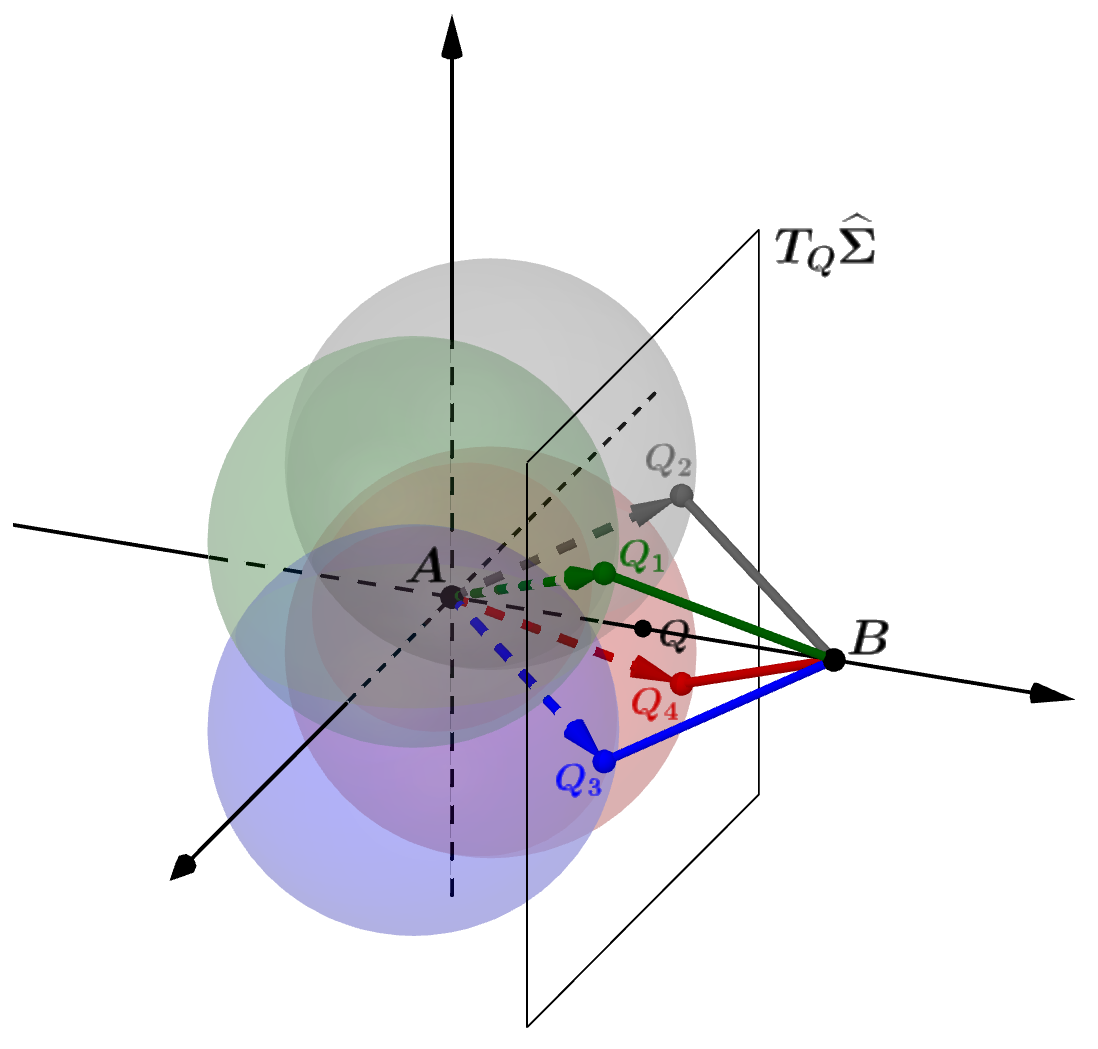}
\caption{Illustration of Theorem~\ref{thm:time_min} in $\mathbb{R}^3$ for a multi-convex indicatrix composed of patches from four spherical indicatrices, each with displaced centers with respect to the origin. The intersection between the tangent plane $ T_Q\widehat{\Sigma} $ and the multi-convex indicatrix gives four points $\{Q_i\}_{i=1}^4$. Any combination of segments in the direction of $\{Q_i\}_{i=1}^4$ from $A$ to $B$ coincides in travel time with the four trajectories depicted, which are examples of time-minimizing trajectories from $A$ to $B$.}
\label{fig:albatross}
\end{figure}

\begin{remark}
\label{rem:integral}
    At first glance, one might consider the following natural approach to generalizing Theorem~\ref{thm:time_min} to the time- and position-dependent case: the optimal velocities $\{Q_i\}_{i=1}^k$ provided by the theorem, when varying with time and position, would define time-dependent vector fields whose $t$-parametrized integral curves could serve as candidate paths to be concatenated---via tack points---to go from $A$ to $B$. However, in general, we have no control over the evolution of the non-convex parts of the indicatrix. For instance, the indicatrix may rotate over time, or its non-convex parts may transform into convex ones. As a result, the optimal velocities may change abruptly and even vary in number across time and space, so the associated vector fields---and hence their integral curves---may not even be well-defined. For this reason, we propose a simplified approach in the next section.
\end{remark}

\section{Restricted non-convex navigation}
\label{sec:restricted}
When we travel using a multi-convex norm, there are so many degrees of freedom available---e.g. the number of tack points or the choice of the initial direction---that, as seen in Theorem~\ref{thm:time_min}, the optimal trajectory might be highly non-unique. In practice, however, this is rarely the case: for instance, one is usually interested in tacking as little as possible, since this maneuver involves a temporary loss of the speed---one cannot instantly change direction, as we are assuming in the idealized theoretical situation. In addition, as noted in Remark~\ref{rem:integral}, the direct generalization of the non-convex navigation to the non-constant case becomes extremely difficult to handle.

Therefore, in this section we will introduce certain restrictions to the previous setting, with two main goals: first, to identify the minimal and simplest conditions under which we can ensure the existence and uniqueness of the optimal trajectory in the constant case; and second, to reformulate the problem so that it becomes more tractable in the non-constant case, and even when the initial and arrival regions are not single points.

\subsection{Single tacking with two Minkowski norms}
\label{subsec:restricted_constant}
Assume we have two constant Finsler metrics $F^{\alpha}, F^{\beta}$ (i.e. Minkowski norms) on $N=\mathbb{R}^n$, with their corresponding indicatrices $\Sigma^{\alpha}, \Sigma^{\beta}$, wavemaps $f^{\alpha}, f^{\beta}$ (recall Definition~\ref{def:wavemap}) and wavefronts $\mathcal{W}_t^{\alpha}, \mathcal{W}_t^{\beta}$ (recall Definition~\ref{def:time_min_wavemap}). Given an initial and arrival point $A,B \in N$, the goal is to find the fastest trajectory from $A$ to $B$, with the following restrictions:
\begin{itemize}
    \item[(I)] Only one tack point $p = (x^1,\ldots,x^n)$ is allowed.\footnote{This is not overly restrictive for modeling purposes, as any path with multiple tack points can be decomposed into a sequence of single-tack-point trajectories. However, in this case one would minimize the individual travel times between tack points instead of the total travel time.}
    \item[(II)] One must travel from $A$ to $p$ with velocity $\Sigma^{\alpha}$, and from $p$ to $B$ with velocity $\Sigma^{\beta}$.
\end{itemize}

Obviously, the fastest trajectory from $A$ to $p$ and from $p$ to $B$ are both straight line segments---because these are the Finsler pregeodesics in the constant case---represented shorthand by the Euclidean vector $p-A$ (from $A$ to $p$) and the vector $B-p$ (from $p$ to $B$), respectively. Hence, we can just work with these piecewise segments and express the total travel time as a positive function of $p = (x^1,\ldots,x^n)$:
\begin{equation}
    \label{eq:traveltime}
    \mathcal{T} = \mathcal{T}(p) = F^{\alpha}(p-A)+F^{\beta}(B-p) = ||p-A||_{\Sigma^{\alpha}} + ||B-p||_{\Sigma^{\beta}}.
\end{equation}
Now, since we have no control over the tack point $p$ but $A$ and $B$ are fixed, it is more convenient to express the travel time in terms of the (constant) Finsler metric $F^{\beta^{-}}$, defined as the {\em reverse Finsler metric} of $F^{\beta}$:
\begin{equation*}
    F^{\beta^{-}}(v) \coloneqq F^{\beta}(-v), \quad \forall v \in \mathbb{R}^n,
\end{equation*}
so that $\mathcal{T}(p) = F^{\alpha}(p-A) + F^{\beta^-}(p-B)$, i.e. we can compute the travel time going forward from $A$ to $p$ (using $\Sigma^\alpha$) and backward from $B$ to $p$ (using the velocity given by the indicatrix $\Sigma^{\beta^-}$ of $F^{\beta^-}$). We denote by $f^{\beta^-}$ the wavemap from $B$ associated with $F^{\beta^-}$.

\begin{definition}
    In the setting above:
    \begin{itemize}
        \item We define the {\em curve of candidate tack points} as the unique curve from $A$ to $B$ where a wavefront of $f^{\alpha}$ from $A$ is tangential to a wavefront of $f^{\beta^{-}}$ from $B$.
        \item An {\em optimal tack point} for the travel from $A$ to $B$ is a point $p \in N$ where $\mathcal{T}$ attains its global minimum.
    \end{itemize}
\end{definition}

\begin{remark}
    The uniqueness of the curve of candidate tack points follows directly from the strict convexity of the wavefronts inherited (by constancy of the metrics $F^{\alpha}$ and $F^{\beta^{-}}$) from the strong convexity of the corresponding indicatrices.
\end{remark}

The goal, of course, is to find the optimal tack point, for which the curve of candidate tack points will play a key role. To show this, let us restrict the problem even further for a moment. Let $T = F^{\alpha}(B)$ be the time $f^{\alpha}$ spends going from $A$ (at $t=0$) to $B$, i.e. $B \in \mathcal{W}_T^{\alpha}$. Suppose we fix a time $\tau \in [0,T]$ at which we want to tack and define a {\em $\tau$-optimal tack point} as a point $p \in N$ where $\mathcal{T}$ is minimum among all the possible tack points at time $t=\tau$.

\begin{lemma}\label{lem:tau_optimal}
    For any $\tau \in [0,T]$, there exists a unique $\tau$-optimal tack point. Moreover, the set $\{p \in N: p \text{ is } \tau\text{-optimal}\}_{\tau=0}^{T}$ is exactly the image of the curve of candidate tack points.
\end{lemma}
\begin{proof}
    Notice that the possible tack points at $t=\tau$ are those in the wavefront $\mathcal{W}^{\alpha}_{\tau}$. The intersection point $p$ between this wavefront and the curve of candidate tack points provides the unique $\tau$-optimal tack point, since the wavefront of $f^{\beta^-}$ must spend less time reaching $p$ from $B$ than any other point in $\mathcal{W}^{\alpha}_{\tau}$, due to the fact that both wavefronts are strictly convex and tangential at $p$. Conversely, any point $p$ on the curve of candidate tack points is reached by $f^{\alpha}$ after a certain time $\tau \in [0,T]$ and therefore, by the previous reasoning, $p$ must be $\tau$-optimal.
\end{proof}

\begin{remark}
\label{rem:tau_optimal}
    Each $\tau$-optimal tack point solves the classical Zermelo's problem for the travel from $\mathcal{W}_{\tau}^{\alpha}$ to $B$ with velocity $\Sigma^{\beta}$. Therefore, by Theorem~\ref{thm:sol_zermelo}, the $\tau$-optimal tack point is the unique point $p \in \mathcal{W}_{\tau}^{\alpha}$ such that $B-p \bot_{F^{\beta}} \mathcal{W}_{\tau}^{\alpha}$ or, equivalently, $p-B \bot_{F^{\beta^-}} \mathcal{W}_{\tau}^\alpha$.
\end{remark}

This provides a way to explicitly compute the curve of candidate tack points: the unique $\tau$-optimal tack points, when varying $\tau \in [0,T]$, provide the curve of candidate tack points $\varphi: [0,T] \rightarrow N$ parametrized by the tacking time $\tau$, i.e. $\varphi(\tau) \in \mathcal{W}_{\tau}^{\alpha}$. Then, the optimal tack point can be obtained simply by looking for the minimum of $\mathcal{T}$ among all the points on $\varphi$.

\begin{theorem}\label{thm:existence}
    Let $\varphi$ be the curve of candidate tack points. Then, the global minimum of $\mathcal{T} \circ \varphi$ always exists and is an optimal tack point.
\end{theorem}
\begin{proof}
    Consider the curve of candidate tack points $\varphi$ parametrized by the tacking time $\tau$, as above. The compactness of the interval $[0,T]$ ensures that $\mathcal{T}\circ \varphi$ attains its minimum at some (maybe non-unique) $\tau_0 \in [0,T]$. Suppose $\varphi(\tau_0)$ is non-optimal, i.e. there exists a point $p \in N$ such that $\mathcal{T}(p) < \mathcal{T}(\varphi(\tau_0))$. Let $\tau_1 \geq 0$ be the time $f^{\alpha}$ needs to reach $p$. Clearly $\tau_1 < T$ because $\mathcal{T}(\varphi(\tau_0)) \leq \mathcal{T}(\varphi(T)) = T$. Then, since $\varphi(\tau_1)$ is $\tau_1$-optimal by Lemma~\ref{lem:tau_optimal}, we obtain $\mathcal{T}(\varphi(\tau_1)) \leq \mathcal{T}(p) < \mathcal{T}(\varphi(\tau_0))$, which contradicts the fact that $\tau_0$ is a minimum of $\mathcal{T} \circ \varphi$.
\end{proof}

Finally, we show that, although we have slightly reformulated the Zermelo navigation from Section~\ref{sec:non-convex}---using separately $\Sigma^\alpha$ and $\Sigma^\beta$ instead of the multi-convex indicatrix $\Sigma$ they would form---the optimal trajectory obtained here is in fact a particular solution of Theorem~\ref{thm:time_min}, i.e. a solution for the travel with (non-convex) velocity $\Sigma$. Therefore, the problem proposed in this subsection is merely a restricted version of the one in Section~\ref{subsec:multi-convex}, rather than a different one---e.g. there is no difference between using $\Sigma$ from Figure~\ref{fig:norm1} or use separately $\Sigma^{\alpha}$ and $\Sigma^\beta$ from Figure~\ref{fig:norm2}. In fact, unlike in Theorem~\ref{thm:time_min}, the restrictions now ensure that the solution is unique.

\begin{theorem}
\label{thm:uniqueness}
    Let $p$ be an optimal tack point and consider the multi-convex indicatrix $\Sigma$ formed by (the outermost parts of) $\Sigma^{\alpha}$ and $\Sigma^{\beta}$. Then, the piecewise segment joining $A$, $p$ and $B$ is the unique time-minimizing trajectory provided by Theorem~\ref{thm:time_min} that satisfies the restrictions (I) and (II).
\end{theorem}
\begin{proof}
    Consider the travel from $A$ to $B$ with velocity $\Sigma$ (in particular, if $\Sigma^{\alpha}, \Sigma^{\beta}$ do not intersect transversally, then $\Sigma=\Sigma^{\alpha}$ or $\Sigma=\Sigma^{\beta}$). By Theorem \ref{thm:time_min}, the number of possible velocities one must use to achieve the time-minimizing trajectory is either one or two. In the former case, the straight line from $A$ to $B$ is strictly time-minimizing. In the latter case, one of the velocities belong to $\Sigma \cap \Sigma^{\alpha}$, and the other to $\Sigma \cap \Sigma^{\beta}$, so there are exactly two time-minimizing trajectories with one tack point, but only one of them is valid if we impose the restriction that one must travel first with $\Sigma^{\alpha}$, then with $\Sigma^{\beta}$. Therefore, considering the restrictions (I) and (II), there is a unique time-minimizing (in fact, strictly time-minimizing) trajectory $\gamma$ for the travel from $A$ to $B$ with velocity $\Sigma$.

    Now, let $\sigma$ be the piecewise segment joining $A$, $p$ and $B$. If $\sigma'$ always belongs to $\widehat{\Sigma}$, in particular it also belongs to $\Sigma$ (because it is composed of the outermost parts of $\Sigma^{\alpha}$ and $\Sigma^{\beta}$), so it is clear that $\sigma=\gamma$. Otherwise, $\gamma$ uses velocities that are not in $\widehat{\Sigma}$, so the travel is suboptimal by the proof of Theorem~\ref{thm:time_min} and we have that $\mathcal{T}[\sigma]<\mathcal{T}[\gamma]$, contradicting the fact that $p$ is an optimal tack point.
\end{proof}

As a direct consequence of the mentioned restrictions, the optimal tack point is also essentially unique, in the following sense.

\begin{corollary}
\label{cor:uniqueness}
An optimal tack point is either unique or else the set of optimal tack points is the whole straight segment from $A$ to $B$. The latter case occurs if and only if the straight segment is strictly time-minimizing and $F^{\alpha}(B) = F^{\beta}(B)$.
\end{corollary}
\begin{proof}
If the straight line from $A$ to $B$ is strictly time-minimizing, then, for each case $F^{\alpha}(B) > F^{\beta}(B)$, $F^{\alpha}(B) < F^{\beta}(B)$ and $F^{\alpha}(B) = F^{\beta}(B)$, the optimal tack points are, respectively, $A$ (it is faster to use only $F^{\beta}$), $B$ (it is faster to use only $F^{\alpha}$) or the whole segment (it does not matter where we put the tack point because the travels with $F^{\alpha}$ and $F^{\beta}$ are the same). Otherwise, the unique (non-straight) solution of Theorem~\ref{thm:uniqueness} corresponds to a unique optimal tack point.
\end{proof}

\subsection{Single tacking with two time-only dependent Finsler metrics}
\label{subsec:restricted_time}
The advantage of the formulation introduced in Section~\ref{subsec:restricted_constant} lies in its strong reliance on the uniqueness of the curve of candidate tack points, which, in turn, follows from the strict convexity of the wavefronts when two constant Finsler metrics are involved. Since this strict convexity persists over time when the metrics are time-only dependent and even when $A$ and $B$ are strictly convex, as established in Theorem~\ref{thm:convexity}, one may wonder whether the previous results concerning the curve of candidate tack points and the optimal tack point still hold in this setting.

So, assume we have two time-only dependent Finsler metrics $F^{\alpha}_t, F^{\beta}_t$ on $N=\mathbb{R}^n$ with their corresponding indicatrices $\Sigma^{\alpha}_t, \Sigma^{\beta}_t$, and the initial and arrival regions $A,B \subset N$ are strictly convex (of course, in particular they can be single points). The goal, as above, is to find the fastest trajectory between $A$ and $B$, providing we have to travel from $A$ to a single tack point $p$ with velocity $\Sigma^{\alpha}_t$, and from $p$ to $B$ with velocity $\Sigma^{\beta}_t$.

The first obvious observation is that, due to the time-dependence, the wavemaps $f^{\alpha}, f^{\beta}$ now have to be computed using the corresponding Lorentz-Finsler metrics $H^{\alpha} \coloneqq \mathrm{d}t^2-(F^{\alpha}_t)^2$, $H^{\beta} \coloneqq \mathrm{d}t^2-(F^{\beta}_t)^2$ on the spacetime $M = \mathbb{R} \times N$ (recall Section \ref{subsec:wavemap}). This way, the associated wavefronts $\mathcal{W}_t^{\alpha}, \mathcal{W}_t^{\beta}$ are given by the projection on $N$ of (future-directed) lightlike pregeodesics of $H^{\alpha}$ and $ H^{\beta}$, respectively. Also, as above, we will make use of the (time-only dependent) Finsler metric $F^{\beta^-}_t$, defined as the reverse of $F^{\beta}_t$:
\begin{equation*}
    F^{\beta^-}_t(v) \coloneqq F^{\beta}_t(-v), \qquad \forall t \in \mathbb{R}, \quad \forall v \in \mathbb{R}^n.
\end{equation*}

\begin{remark}
    The reverse metric $F^{\beta^-}_t$ provides the Lorentz-Finsler metric $H^{\beta^-} \coloneqq \mathrm{d}t^2 - (F^{\beta^-}_t)^2$ on $M$. In order to understand the relationship between $H^{\beta}$ and $H^{\beta^-}$, we need to introduce some notions regarding the future and past lightcones:
    \begin{itemize}
        \item For $H^{\beta}$, recall from Definition~\ref{def:causality} and Remark~\ref{rem:ligthcones} that its future lightcone is $\mathcal{C}^{\beta} \coloneqq (H^{\beta})^{-1}(0) \cap \mathrm{d}t^{-1}((0,\infty))$ and the past lightcone is given by $-\mathcal{C}^{\beta}$, i.e. $\hat{v}$ is past-directed for $H^{\beta}$ if and only if $-\hat{v}$ is future-directed.
        \item On the other hand, for $H^{\beta^-}$ we now define its past lightcone as $\mathcal{C}^{\beta^-} \coloneqq (H^{\beta^-})^{-1}(0) \cap \mathrm{d}t^{-1}((-\infty,0))$, being its future lightcone $-\mathcal{C}^{\beta^-}$. Again, $\hat{v}$ is past-directed for $H^{\beta^-}$ if and only if $-\hat{v}$ is future-directed.
    \end{itemize}
\end{remark}

Since $H^{\beta^-}(\hat{v}) = H^{\beta}(-\hat{v})$, it turns out that $\mathcal{C}^{\beta^-} = -\mathcal{C}^{\beta}$, i.e. $\hat{v}$ is future-directed (resp. past-directed) for $H^\beta$ if and only if it is also future-directed (resp. past-directed) for $H^{\beta^-}$. As a consequence, both metrics share the same causal curves. Moreover, we have the following relationship between their pregeodesics.

\begin{lemma}\label{lem:reverse_geod_H}
    A curve $\hat{\gamma}$ in $M$ is a future-directed pregeodesic of $H^{\beta}$ if and only if, when traveled in reverse, is a past-directed pregeodesic of $H^{\beta^-}$.
\end{lemma}
\begin{proof}
    By definition (recall Section~\ref{subsec:lorentz-finsler}), $\hat{\gamma}: [a,b] \rightarrow M$ is a geodesic of $H^{\beta}$ if it is a critical point of the $H^\beta$-energy functional
    \begin{equation}
    \label{eq:energy}
        \mathcal{E}_{H^{\beta}}[\hat{\gamma}] \coloneqq \int_a^b H^{\beta}(\hat{\gamma}'(s)) \ \mathrm{d}s,
    \end{equation}
    defined in the set of (piecewise smooth) curves with fixed with fixed endpoints $\hat{\gamma}(a), \hat{\gamma}(b)$. Now, consider the reparametrization $s(r) = a+b-r$, so that the curve $\tilde{\gamma}: [a,b] \rightarrow M$ given by $\tilde{\gamma}(r) = \hat{\gamma}(s(r))$ travels the same path as $\hat{\gamma}$, but in reverse---from $\hat{\gamma}(b)$ to $\hat{\gamma}(a)$. Then, applying the change of variables $s(r)$ in \eqref{eq:energy} and noting that $H^{\beta^-}(\hat{v}) = H^{\beta}(-\hat{v})$, we obtain
    \begin{equation*}
        \mathcal{E}_{H^{\beta}}[\hat{\gamma}] = -\int_{b}^a H^{\beta}(-\tilde{\gamma}'(r)) \ \mathrm{d}r = \int_a^b H^{\beta^-}(\tilde{\gamma}'(r)) \ \mathrm{d}r = \mathcal{E}_{H^{\beta^-}}[\tilde{\gamma}],
    \end{equation*}
    where $\mathcal{E}_{H^{\beta^-}}$ is the energy functional of $H^{\beta^-}$, defined in the set of (piecewise smooth) curves with fixed endpoints $\tilde{\gamma}(a)=\hat{\gamma}(b), \tilde{\gamma}(b)=\hat{\gamma}(a)$. Moreover, observe that the set of curves where $\mathcal{E}_{H^{\beta}}$ and $\mathcal{E}_{H^{\beta^-}}$ are defined coincide (up to a reverse reparametrization). This means that $\hat{\gamma}$ is a critical point of $\mathcal{E}_{H^{\beta}}$ if and only if $\tilde{\gamma}$ is a critical point of $\mathcal{E}_{H^{\beta^-}}$, i.e. $\hat{\gamma}$ is a future-directed geodesic of $H^{\beta}$ if and only if $\tilde{\gamma}$ is a past-directed geodesic of $H^{\beta^-}$.

    Finally, in the more general case when $\hat{\gamma}$ is a future-directed pregeodesic of $H^{\beta}$, we can reparametrize it as a geodesic. Then, applying the above, we can obtain a past-directed geodesic of $H^{\beta^-}$ and any further reparametrization that keeps the orientation (so that it travels the same path as $\hat{\gamma}$ but in reverse) makes it a past-directed pregeodesic of $H^{\beta^-}$. The converse holds analogously, i.e. any past-directed pregeodesic of $H^{\beta^-}$ can be reversely reparametrized to be a future-directed pregeodesic of $H^{\beta}$.
\end{proof}

In order to find the fastest trajectory in this setting, we need a way to compute the travel time. Observe that the expression \eqref{eq:traveltime} is, of course, not valid here. Instead, for any tack point $p \in N$, we have to proceed with the following steps:
\begin{enumerate}
    \item Take the wavemap $f^{\alpha}$ of $H^\alpha$ departing from $A$ at time $t=0$, and let $t_0 \geq 0$ be the time $f^{\alpha}$ needs to arrive at $p$.
    \item Take the wavemap $f^{\beta}$ of $H^{\beta}$ departing from $p$ at time $t=t_0$, and let $\mathcal{T} \geq t_0$ be arrival time at $B$ (when the wavefront is tangent to $B$), i.e. $f^{\beta}$ takes a time $\mathcal{T}-t_0$ to go from $p$ to $B$.
    \item The time $\mathcal{T}$ is the total travel time for the path $A-p-B$. This way, we obtain a smooth function $\mathcal{T}=\mathcal{T}(p)=\mathcal{T}(x^1,\ldots,x^n)$ and an {\em optimal tack point} is a point where if $\mathcal{T}$ attains its global minimum.
\end{enumerate}

Now, fix a tack point $p \in N$ with its corresponding total travel time $\mathcal{T}(p)$ and consider the (reverse) wavemap $f^{\beta^-}$ of $H^{\beta^-}$ given by the projection of its past-directed lightlike pregeodesics departing from $B$ at time $t=\mathcal{T}(p)$. By Lemma~\ref{lem:reverse_geod_H}, $f^{\beta^-}$ arrives at $p$ at time $t_0$, i.e. $f^{\beta^-}$ spends the same (backward) time going from $B$ to $p$ as $f^{\beta}$ does going (forward in time) from $p$ to $B$. This allows us to construct the curve of candidate tack points in a similar way as in Section~\ref{subsec:restricted_constant} and derive analogous results.

\begin{proposition}
    Let $T > 0$ be the time $f^{\alpha}$ spends going from $A$ at $t=0$ to $B$. For any $\tau \in [0,T]$, there exists a unique time $\mathcal{T} \in [\tau,T]$ such that the wavefront of $f^{\alpha}$ departing from $A$ and running from $t=0$ to $t=\tau$ is tangential to the wavefront of $f^{\beta^-}$ departing from $B$ and running (backward in time) from $t=\mathcal{T}$ to $t=\tau$. Obviously, $\mathcal{T}$ coincides with the total travel time of the corresponding tangent point---when used as a tack point. The set of all these tangent points defines a unique {\em curve of candidate tack points} from $A$ to $B$.
\end{proposition}
\begin{proof}
    This follows immediately from the previous observation and the strict convexity of the wavefronts of $f^{\alpha}$ and $f^{\beta^-}$, given by Theorem~\ref{thm:convexity}, due to the fact that $F^\alpha_t, F^{\beta^-}_t$ are time-only dependent and $A,B$ are strictly convex.
\end{proof}

Then, observe that Lemma~\ref{lem:tau_optimal} and Theorem~\ref{thm:existence} hold exactly the same in this time-only dependent setting. In contrast, Theorem~\ref{thm:uniqueness} and Corollary~\ref{cor:uniqueness} cannot be directly generalized, as they strongly rely on Theorem~\ref{thm:time_min}. In particular, the uniqueness of the optimal tack point is not guaranteed now, as we will see in the following section.

\begin{remark}
\label{rem:trajectories}
    Once we know the optimal tack point $p \in N$, the solution to Zermelo's problem---the fastest trajectory between $A$ and $B$---is the unique spatial trajectory of $f^{\alpha}(z,t)$ from $A$ at $t=0$ such that $\partial_t f^{\alpha}(z_0,0) \bot_{F^{\alpha}_0} A$ and $f^\alpha(z_0,t_0) = p$ for some $t_0 \geq 0$ (thanks to Theorems~\ref{thm:no_cut_points} and \ref{thm:convexity}, $z_0 \in A$ and $t_0$ are unique), concatenated with the unique spatial trajectory of $f^{\beta}(\theta,t)$ from $p$ at $t=t_0$ (as noted in Remark~\ref{rem:wavemap} and showed in Section~\ref{subsec:circles}, $\theta \in [0,2\pi)$ parametrizes directions in $T_pN$) such that $f^{\beta}(\theta_0,\mathcal{T})\bot_{F^{\beta}_{\mathcal{T}}}$ for some $\mathcal{T}\geq t_0$ (again, $\theta_0$ and $\mathcal{T}$ are unique). Note that each trajectory is, respectively, the solution to Zermelo's problem for the travel from $A$ at $t=0$ to $p$, and from $p$ at $t=t_0$ to $B$ (in particular, they satisfy Theorem~\ref{thm:sol_zermelo}(ii)).
\end{remark}

Finally, as an immediate generalization, observe that all the results in this subsection hold true if we replace the time-only dependent Finsler metrics with general time- and position-dependent ones, as long as their wavefronts always remain strictly convex.

\subsection{Non-uniqueness counterexamples}
\label{subsec:counter_examples}
In contrast to the results presented in Theorem~\ref{thm:uniqueness} and Corollary~\ref{cor:uniqueness}, when we consider two time-only dependent Finsler metrics $F^{\alpha}_t, F^{\beta}_t$, an optimal tack point $p$ for the bi-metric travel from $A$ to $B$, i.e. first using $F^{\alpha}_t$ and then $F^{\beta}_t$, need not be unique. Indeed, as already indicated, time- and position-dependent metrics do not in general have convex wavefronts, so two wavefronts can meet and be identical along any open subset of the two wavefronts and thence produce a large system of positions for optimal tack points. We now observe via simple examples, that even time-only dependent (as well as position-only dependent) \emph{Riemannian metrics} can produce such non-uniqueness of optimal tack points.

\begin{example}
Following the example presented in Section~\ref{subsec:circles}, we consider Riemannian norms $F^{\alpha}$ and $F^{\beta}_t$ in $\mathbb{R}^{2}$ with circular indicatrix fields and both without wind, i.e. $W^{\alpha}(t) = 0$ and $W^{\beta}(t) = 0$ for all $t$. We assume constant radius $R^{\alpha}$ for $F^{\alpha}$, but a time-varying radius $R^{\beta}$ for $F^{\beta}_t$ as follows, where $\lambda \in (0,1)$ is any constant:
\begin{equation*}
\begin{aligned}
R^{\alpha}(t) &=  1, \quad \forall t \in \mathbb{R},\\
R^{\beta}(t) &= 1 - \lambda \sin(t), \quad \forall t \in \mathbb{R}. 
\end{aligned}
\end{equation*}
By symmetry, the time-optimal track between $A = (0,0)$ and $B = (5\pi, 0)$ is the straight line segment on the $x^{1}$-axis connecting these two points. For this trajectory, the travel time with $F^{\alpha}$ from $A$ to $B$ is $T = 5\pi$, but the travel time with $F^{\beta}_t$ from $A$ to $B$ is larger, because the ``reach'' from $A$ towards $B$ by using $F^{\beta}_t$ during the total time $T= 5\pi$ is only:
\begin{equation*}
\int_{0}^{5\pi} R^{\beta}(r) \ \mathrm{d}r  < 5\pi.
\end{equation*}

Suppose now that we force $(c,0)$ to be a tack point, where $c \in [0, 5\pi]$. Then, the time to reach $(c,0)$ from $A$ by $F^{\alpha}$ is $c$ and the time to reach $B$ from $(c,0)$ by $F^{\beta}_t$ is some time $\tilde{T} - c$. We claim that $\tilde{T} \geq 5\pi$ and that $T = 5\pi$ is the optimal travel time with tacking at specific tack points for the total travel from $A$ to $B$.

First, the travel time $T=5\pi$ can be realized by using the endpoint $B = (5\pi,0)$ as tack point, i.e. by using $c=5\pi$. Second, for any (other) tack point $(c,0)$, the $t$-parametrized track from $A$ to $B$ is $(x^{1}(t), 0)$, where, by the explicit integral expression \eqref{eq:time_only_point} (in the direction $\theta=0$):
\begin{equation*}
x^{1}(t) = c + \int_{c}^{t} R^{\beta}(r) \ \mathrm{d}r, \quad \forall t\geq c.
\end{equation*}
Since we want to reach the target point $B = (5\pi,0)$, we therefore need a travel time $\tilde{T}$ satisfying:
\begin{equation*}
x^{1}(\tilde{T}) = c + \int_{c}^{\tilde{T}} R^{\beta}(r) \ \mathrm{d}r = 5\pi.
\end{equation*}
However, suppose we only have total time $T=5\pi$ at our disposal. Then, by the specific periodicity of $R^{\beta}$, we get the following smaller reach $\ell$ of the track as a function of $c \in [0, 5\pi]$, using the time $T$:
\begin{equation*}
    x^{1}(5\pi) = \ell(b) = c + \int_{c}^{5\pi} R^{\beta}(r) \ \mathrm{d}r \leq  5\pi,
\end{equation*}
and equality holds if and only if $c \in \{\pi, 3\pi, 5\pi \}$. In consequence, there are two optimal tack points along the interior of the line segment from $A$ to $B$, namely $(\pi,0)$ and $(3\pi,0)$, each giving the total travel time $T = 5\pi$. This shows that the optimal tack point for a one-tack-point travel from $A$ to $B$, using first the metric $F^{\alpha}$ and then $F^{\beta}_t$, is not unique.
\end{example}

\begin{example}
In a vein similar to the example above, choosing $R^{\alpha}(x^{1}, x^{2}) = 1/(2+\cos(x^{1}))$ and $R^{\beta}(x^{1} , x^{2})= 1/2$, they now provide the position-only dependent Riemannian norms $F^{\alpha} = \Vert \cdot \Vert(2+\cos(x^{1})) $ and $ F^{\beta} = 2\Vert \cdot \Vert$, where $\Vert \cdot \Vert$ is the standard Euclidean norm in $\mathbb{R}^{2}$.  It is again clear that the optimal path for a travel from $A=(0,0)$ to $B=(5\pi,0)$ is the corresponding straight line segment on the $x^{1}$-axis. By \eqref{eq:traveltime_F}, the total travel time when choosing $(c,0)$ as the tack point---parametrizing the path as $\gamma(s) = (s,0)$, $s \in [0,5\pi]$---is
\begin{equation*}
\mathcal{T}(c) = \int_0^c (2+\cos(s)) \ \mathrm{d}s + \int_c^{5\pi} 2 \  \mathrm{d}s = \sin(c)+10\pi.
\end{equation*}
The travel time therefore attains its minimum when $\sin(c) = -1$, i.e. at $c\in \{\frac{3\pi}{2},\frac{7\pi}{2}\}$, which again shows the non-uniqueness of the optimal tack point, this time in the position-only dependent case.
\end{example}

\subsection{Connection with Snell's law}
\label{subsec:snell}
There is an interesting connection between the non-convex navigation we have introduced here and the law of refraction in physics, commonly known as {\em Snell's law}. In the usual Snell's setting, the space $N$ is divided in two by an interface $\eta$ (a hypersurface of $N$), representing the separation between two different media where a wave is spreading. By Fermat's principle, the wave rays passing through the interface from one medium to another follow the path that minimizes (or more generally, makes stationary) the travel time between their endpoints. Since, in general, the speed profile of the wave is discontinuous along $\eta$, the wave ray must undergo a specific ``break'' at the interface, resulting in a ``refracted'' trajectory. The condition for this break (which we can associate with a tack point) to satisfy Fermat's principle is what we properly call Snell's law. This law has been generalized for anisotropic time-independent waves in \cite{MP2023}, and for the general time-dependent case in \cite{JMPS}.

So, essentially, the setting in Snell's law is a more restricted version of our non-convex Zermelo navigation, in the sense that in the Snell setting you are allowed to tack only along the interface $\eta$. However, observe that the model is completely analogous. Indeed, the wave has a different speed profile at each side of $\eta$, resulting in two different Finsler metrics $F^\alpha_t, F^\beta_t$, and the goal is to find the ray path between two points (or regions) $A,B$ separated by $\eta$, i.e. to find the solution to Zermelo's problem for the travel from $A$ to $B$, with the restrictions that we must use $F^\alpha_t$ on one side of $\eta$ and $F^\beta_t$ on the other side, and we are allowed to cross $\eta$ only once. As a consequence, we have the following straightforward connection.

\begin{theorem}
    Let $p$ be a (possibly non-unique) optimal tack point at time $t_0 \geq 0$ for the one-tack-point travel from $A$ to $B$, using first $F^\alpha_t$ and then $F^\beta_t$ (possibly time- and position-dependent Finsler metrics). Let $\gamma_\alpha, \gamma_\beta$ be the optimal trajectories from $A$ (at $t=0$) to $p$ and from $p$ (at $t=t_0$) to $B$, respectively. Then, the complete concatenated trajectory $\gamma_\alpha \cup \gamma_\beta$ satisfies Snell's law:
    \begin{equation}
    \label{eq:snell}
        \left( \frac{\partial F^{\alpha}_{t_0}}{\partial v^i}(\gamma_\alpha'(t_0^-)) - \frac{\partial F^{\beta}_{t_0}}{\partial v^i}(\gamma_\beta'(t_0^+)) \right) u^i = 0, \quad \forall u = (u^1,\ldots,u^n) \in T_p\eta,
    \end{equation}
    for any interface $\eta$ (passing through $p$) that divides $\gamma_\alpha$ and $\gamma_\beta$.
\end{theorem}
\begin{proof}
    Snell's formula \eqref{eq:snell} comes directly from \cite[Equation~(9)]{MP2023} at $t=t_0$. If $\gamma_\alpha \cup \gamma_\beta$ does not satisfy this condition, there is a variation in Snell's setting that locally reduces the travel time, in contradiction with the optimality of $\gamma_\alpha \cup \gamma_\beta$ in the more general setting of Zermelo navigation with tacking.
\end{proof}

\section{Computational estimation of the tacking point}
\label{sec:computations}
Although the theoretical analysis of the non-convex Zermelo navigation becomes significantly more complex when the wavefronts are not strictly convex---and this will not be addressed in the present work---we will show in this section that, from a computational standpoint, it is still possible to find solutions in the most general case, involving any number of time-dependent Finsler metrics and tack points. However, in situations not supported by theoretical results (e.g., when the wavefronts are not strictly convex), it is important to note that we cannot ensure that the computational solution is truly optimal---it may merely correspond to a local minimum of the travel time.

To compute solutions numerically, we will focus on the formulation of Zermelo’s problem introduced in Section~\ref{subsec:finsler_metrics}. Specifically, our goal is to minimize the (discretized) travel time functional~\eqref{eq:traveltime_F} over (discretized) curves $\hat{\gamma}(s) = (t(s), \gamma(s))$ satisfying the constraint~\eqref{eq:dot_t_F}---i.e., $\hat{\gamma} \in \mathcal{Q}_{A,B}$ as in Definition~\ref{def:travel_functional}---where the initial and arrival regions $A, B$ will be assumed to be single points.\footnote{Point-like initial and arrival regions constitute the most relevant case in practice and will allow us to reduce the computational complexity of the algorithms.} In practice though, for computational convenience, it is preferable to work with the \emph{squared travel time functional}:
\begin{equation}
\label{eq:squared_traveltime}
    \widehat{\mathcal{T}}[\hat{\gamma}] \coloneqq \int_a^b F(t(s),\gamma(s),\gamma'(s))^2 \ \mathrm{d}s = \int_a^b g_{\gamma'(s)}^{F_{t(s)}}(\gamma'(s),\gamma'(s)) \ \mathrm{d}s,
\end{equation}
where $g^{F_t}$ is the fundamental tensor of $F_t$, given by \eqref{eq:fund_tensor_F}. This is because being a critical point of $\widehat{\mathcal{T}}$ imposes a specific parametrization, thereby reducing the degrees of freedom and making the numerical computation more robust. At the same time, critical points of $\widehat{\mathcal{T}}$ remain critical for $\mathcal{T}$ and thus, they are also pregeodesics of $H=\mathrm{d}t^2-F_t^2$ (defined in Section~\ref{subsec:lorentz-finsler}), as shown in the following result.

\begin{proposition} \label{prop:h_geodesics}
    Let $F_t$ be a time-dependent Finsler metric. Then, given two points $A,B \in N$ and a curve $\hat{\gamma}(s)=(t(s),\gamma(s))$ in $ \mathcal{Q}_{A,B}$, the following statements are equivalent:
    \begin{itemize}
        \item[(i)] $\hat{\gamma}$ is a critical point of the squared travel time functional \eqref{eq:squared_traveltime} on $\mathcal{Q}_{A,B}$.
        \item[(ii)] $\hat{\gamma}$ is a critical point of the travel time functional \eqref{eq:traveltime_F} on $\mathcal{Q}_{A,B}$ such that $t''(s)=0$.
        \item[(iii)] $\hat{\gamma}$ is a pregeodesic of the Lorentz-Finsler metric $H=\mathrm{d}t^2-F_t^2$ such that $t''(s)=0$.\footnote{In particular, this means that $\hat{\gamma}$ can be reparametrized to be a critical point of the $H$-energy functional \eqref{eq:H_energy} among any curve from $\hat{\gamma}(a)$ to $\hat{\gamma}(b)$.}
    \end{itemize}
\end{proposition}
\begin{proof}
    Let us first prove that (i) and (ii) are equivalent. Consider a variation $\hat{\gamma}_{\omega}(s)=(t_{\omega}(s),\gamma_{\omega}(s))$ of $\hat{\gamma}$ through curves in $\mathcal{Q}_{A,B}$ (as in Definition~\ref{def:travel_functional}). Then, writing $F_{\omega}(s) \equiv F(t_\omega(s),\gamma_{\omega}(s),\gamma'_{\omega}(s))$ for short:
    \begin{equation*}
    \begin{split}
        \left. \frac{d}{d\omega} \widehat{\mathcal{T}}[\hat{\gamma}_{\omega}] \right\rvert_{\omega=0} & = \int_a^b 2t'(s) \left. \frac{\partial}{\partial \omega} F_{\omega}(s) \right\rvert_{\omega=0} \ \mathrm{d}s = 2t'(b) \left.\frac{d}{d \omega} t_{\omega}(b)\right\rvert_{\omega=0} - 2\int_a^b t''(s) \left.\frac{\partial}{\partial \omega} t_{\omega}(s)\right\rvert_{\omega=0} \ \mathrm{d}s,
    \end{split}
    \end{equation*}
    where we have used \eqref{eq:dot_t_F}, integration by parts with $\int \left.\frac{\partial}{\partial \omega} F_{\omega}(s)\right\rvert_{\omega=0} \ \mathrm{ds} = \left.\frac{d}{d \omega} t_{\omega}(s)\right\rvert_{\omega=0}$, and the fact that $t_{\omega}(a)=0$ for all $\omega$. Therefore, $\hat{\gamma}$ is a critical point of $\widehat{\mathcal{T}}$ on $\mathcal{Q}_{A,B}$ if and only if the right-hand side of the above equation vanishes for any arbitrary variation $\hat{\gamma}_{\omega} \in \mathcal{Q}_{A,B}$ of $\hat{\gamma}$. This means that
    \begin{equation*}
    \begin{split}
        0 & = t'(b) \left.\frac{d}{d \omega} t_{\omega}(b)\right\rvert_{\omega=0}, \quad \forall \hat{\gamma}_{\omega}\\
        0 & = t''(s) \left.\frac{\partial}{\partial \omega} t_{\omega}(s)\right\rvert_{\omega=0}, \quad \forall s\in[a,b], \quad \forall \hat{\gamma}_{\omega}.
    \end{split}
    \end{equation*}
    Since $t'(b)=F(t(b),\gamma(b),\gamma'(b)) \not= 0$, the first equation is equivalent to $0=\left.\frac{d}{d \omega} t_{\omega}(b)\right\rvert_{\omega=0}=\left.\frac{d}{d \omega} \mathcal{T}[\hat{\gamma}_{\omega}]\right\rvert_{\omega=0}$ for any variation, i.e. $\hat{\gamma}$ is a critical point of $\mathcal{T}$ on $\mathcal{Q}_{A,B}$. The second equation, since it must hold at any $s \in [a,b]$ and for any variation, is equivalent to $t''(s)=0$. Finally, the equivalence between (ii) and (iii) is immediate by Theorem~\ref{thm:sol_zermelo}.
\end{proof}

\begin{remark}
Observe that, for time-independent Finsler metrics, the squared travel time functional \eqref{eq:squared_traveltime} is just the standard energy functional \eqref{eq:F_energy}, whose critical points are Finsler geodesics (i.e. pregeodesics with the appropriate affine parametrization). In the time-dependent case, as Proposition~\ref{prop:h_geodesics} shows, the parametrization that makes a curve a critical point of the squared travel time functional (namely, $t''(s) = 0$) does not coincide with the one required to be an actual geodesic of $H$. In any case, this distinction will not be relevant here, as we are only concerned with the path itself, rather than the specific parametrization used to describe it.
\end{remark}

Since the travel time \eqref{eq:traveltime_F} is invariant under reparametrizations, any minimum of $\mathcal{T}$ can be reparametrized so that $t''(s)=0$ while still remaining a minimum. Therefore, when searching for time-minimizing trajectories, working with $\widehat{\mathcal{T}}$ is actually equivalent to working with $\mathcal{T}$, as long as there is only one (time-dependent) Finsler metric in play. In the following, we will say in general that we are looking for ``pregeodesics'' when looking for minima of $\widehat{\mathcal{T}}$, due to the correspondence of these solutions with Lorentz-Finsler pregeodesics. However, when tack points are included and thus, more than one Finsler metric is used, Proposition~\ref{prop:h_geodesics} no longer holds. In this case, we can still compute the trajectories between tack points as pregeodesics---because these paths have to be solutions of Zermelo's problem themselves---but the tack points need to be found by minimizing the travel time (not the squared travel time) of the total trajectory. We will say in general that we are looking for the ``time-minimizing'' tack points and tack curves when looking for minima of the total travel time.

\subsection{Estimating pregeodesics for time-dependent Finsler metrics}
\label{subsec:pregeodesics_estimation}
At first, we propose an algorithm for computing pregeodesics connecting two boundary points for a time-dependent Finsler metric, in order to estimate the time-minimizing trajectory (without tacking). We aim to extend the \textit{GEORCE} algorithm \citep{georce}, which computes geodesics for Riemannian and Finsler manifolds, but does not take time-dependence into account. Following the computations in \citep{georce}, we consider the discretized version of the squared travel time functional \eqref{eq:squared_traveltime} as a control problem in the form
\begin{equation} \label{eq:disc_energy}
    \begin{split}
        \min_{(t_{s},x_{s},v_{s})} \sum_{s=0}^{T-1}\left(\Delta t_{s}\right)^{2} &= \min_{(t_{s},x_{s},v_{s})} \sum_{s=0}^{T-1}v_{s}^{\top}G(t_{s},x_{s},v_{s})v_{s}, \\
        x_{s+1} &= x_{s}+v_{s}, \quad s=0,\dots,T-1, \\
        t_{s+1}-t_{s} &= F(t_{s},x_{s},v_{s}), \quad s=0,\dots,T-1, \\
        t_{0}&=0, x_{0} = A, x_{T-1}=B,
    \end{split}
\end{equation}
where $G(t,x,v) \coloneqq \{g^{F_t}_{ij}(v)\}$ denotes the coordinate matrix of $g^{F_{t}}_v$, given by \eqref{eq:matrix_gF}. The control formulation allows us to decompose the minimization problem in \eqref{eq:disc_energy} into convex problems, which allows us to derive a system of equations of necessary conditions for the minimum. However, since the state equation in time is potentially non-linear and therefore not convex, we will instead consider a first-order Taylor approximation of $F_{t}$ to derive the necessary conditions for a minimum.
\begin{proposition} \label{prop:disc_energy}
    Assume the following first-order Taylor approximation of $F_t$ for the state equation in time:
    \begin{equation*} \label{eq:f_metric_taylor}
        F(t,x,u) = L^{(0)}+L^{(v)}v+L^{(x)}x+L^{(t)}t,
    \end{equation*}
    where $L^{(v)} \coloneqq \frac{\partial F}{\partial u}\left(t^{(0)},x^{(0)},v^{(0)}\right)$, $L_{s}^{(x)} \coloneqq \frac{\partial F}{\partial x}\left(t^{(0)},x^{(0)},v^{(0)}\right)$, $L^{(t)} \coloneqq \frac{\partial F}{\partial t}\left(t^{(0)},x^{(0)},v^{(0)}\right)$ and $L^{(0)} \coloneqq F\left(t^{(0)},x^{(0)},v^{(0)}\right)$ for some point $\left(t^{(0)},x^{(0)},v^{(0)}\right)$. The necessary conditions for a minimum in \eqref{eq:disc_energy} with the first-order Taylor approximation are then
    \begin{equation} \label{eq:necessary_cond}
        \begin{split}
            &2G(t_{s},x_{s},v_{s})v_{s}+\restr{\nabla_{y}\left[v_{s}^{\top}G(t_{s},x_{s},y)v_{s}\right]}{y=v_{s}}+\mu_{s}+\pi_{s}L_{s}^{(u)}=0, \quad s=0,\dots,T-1, \\
            &v_{s}^{\top}\frac{\partial G\left(t_{s}^{(i)},x_{s}^{(i)},v_{s}^{(i)}\right)}{\partial t}v_{s}+\pi_{s}L_{s}^{(t)}+\pi_{s}=\pi_{s-1}, \quad s=1,\dots,T-1, \\
            &0=\pi_{T-1}, \\
            &\restr{\nabla_{y}\left[v_{s}^{\top}G(t_{s},y,v_{s})v_{s}\right]}{y=x_{s}}+\pi_{s}L_{s}^{(x)}+\mu_{s}=\mu_{s-1}, \quad s=1,\dots,T-1, \\
            &x_{s+1}=x_{s}+v_{s}, \quad s=0,\dots,T-1, \\
            &t_{s+1}=t_{s}+F(t_{s},x_{s},v_{s}),\quad s=0,\dots,T-1, \\
            &\sum_{s=0}^{T-1}v_{s}=B-A, \\
            &t_{0}=0,x_{0}=A,
        \end{split}
    \end{equation}
    where $\mu_{s} \in \mathbb{R}^{n}$ and $\pi_{s} \in \mathbb{R}$ denote the dual prices corresponding to position and time, respectively, for $s=0,\dots,T-1$.
\end{proposition}
\begin{proof}
    See Appendix~\ref{ap:necessary_cond}.
\end{proof}
The problem with the necessary conditions in \eqref{eq:necessary_cond} is that they are not easily solved for a general time-dependent Finsler metric $F_{t}$. We will instead consider an iterative scheme to solve \eqref{eq:necessary_cond}, where at iteration $i$ we consider the variables $\left(t_{s}^{(i)},x_{s}^{(i)},v_{s}^{(i)}\right)$. Inspired by \citep{georce}, we fix the following variables at each iteration $i$:
\begin{equation} \label{eq:georceh_fixed_variables}
    \begin{split}
        \xi_{s} &\coloneqq v_{s}^{\top}\frac{\partial G\left(t_{s}^{(i)},x_{s}^{(i)},v_{s}^{(i)}\right)}{\partial t}v_{s}, \quad s=1,\dots,T-1,\\
        \nu_{s} &\coloneqq \restr{\nabla_{x}v_{s}^{\top}G(t_{s},x,v_{s})v_{s}}{t_{s}=t_{s}^{(i)},x=x_{s}^{(i)},v_{s}=v_{s}^{(i)}}, \quad s=1,\dots,T-1, \\
        \zeta_{s} &\coloneqq \restr{\nabla_{y}v_{s}^{\top}G(t_{s},x_{s},y)v_{s}}{t_{s}=t_{s}^{(i)},x_{s}=x_{s}^{(i)}, v_{s}=v_{s}^{(i)}, y=v_{s}^{(i)}}, \quad s=1,\dots,T-1, \\
        L_{s}^{(v)} &\coloneqq \frac{\partial F}{\partial v}\left(t_{s}^{(i)},x_{s}^{(i)},v_{s}^{(i)}\right), \quad s=1,\dots,T-1, \\
        L_{s}^{(x)} &\coloneqq \frac{\partial F}{\partial x}\left(t_{s}^{(i)},x_{s}^{(i)},v_{s}^{(i)}\right), \quad s=1,\dots,T-1, \\
        L_{s}^{(t)} &\coloneqq \frac{\partial F}{\partial t}\left(t_{s}^{(i)},x_{s}^{(i)},v_{s}^{(i)}\right), \quad s=1,\dots,T-1, \\
        L_{s}^{(0)} &\coloneqq F\left(t_{s}^{(i)},x_{s}^{(i)},v_{s}^{(i)}\right), \quad s=1,\dots,T-1, \\
        G_{s} &\coloneqq G\left(t_{s}^{(i)},x_{s}^{(i)},v_{s}^{(i)}\right), \quad s=0,\dots,T-1, \\
    \end{split}
\end{equation}
such that the system of equations in \eqref{eq:necessary_cond} reduces to
\begin{equation} \label{eq:georceh_equations}
    \begin{split}
        &2G_{s}v_{s}+\zeta_{s}+\mu_{s}+\pi_{s}L_{s}^{(u)}=0, \quad s=0,\dots,T-1, \\
        &\xi_{s}+\pi_{s}L_{s}^{(t)}+\pi_{s}=\pi_{s-1}, \quad s=1,\dots,T-1, \\
        &0=\pi_{T-1}, \\
        &\nu_{s}+\pi_{s}L_{s}^{(x)}+\mu_{s}=\mu_{s-1}, \quad s=1,\dots,T-1, \\
        &x_{s+1}=x_{s}+v_{s},\quad s=0,\dots,T-1, \\
        &t_{s+1}=t_{s}+F(t_{s},x_{s},v_{s}),\quad s=0,\dots,T-1, \\
        &\sum_{s=0}^{T-1}v_{s}=B-A, \\
        &t_{0}=0,x_{0}=A.
    \end{split}
\end{equation}
For the system of equations in \eqref{eq:georceh_equations}, we see that the time dual prices $\{\pi_{s}\}$ can be determined as
\begin{equation} \label{eq:dual_prices}
    \begin{split}
        \pi_{T-1} &= 0, \\
        \pi_{s-1} &= \xi_{s}+\pi_{s}L_{s}^{(t)}+\pi_{s}, \quad s=T-1,\dots,1.
    \end{split}
\end{equation}
Using this, the position dual prices $\{\mu_{s}\}$ can be found exactly as in \citep{georce}:
\begin{equation} \label{eq:georceh_update_scheme}
    \begin{split}
        \mu_{T-1} &= \left(\sum_{s=0}^{T-1}G_{s}^{-1}\right)^{-1}\left(2(A-B)-\sum_{s=0}^{T-1}G_{s}^{-1}\left(\sum_{j>s}^{T-1}\left(\nu_{j}+\pi_{j}L_{j}^{(x)}\right)+\zeta_{s}+\pi_{s}L_{s}^{(v)}\right)\right), \\
        v_{s} &= -\frac{1}{2}G_{s}^{-1}\left(\mu_{T-1}+\sum_{j>t}^{T-1}\left(\nu_{j}+\pi_{j}L_{j}^{(x)}\right)+\zeta_{s}+\pi_{s}L_{s}^{(v)}\right), \quad s=T-1,\dots,0,
    \end{split}
\end{equation}
with the forward equations
\begin{equation*}
    \begin{split}
        &t_{0}=0,x_{0}=A, \\
        &x_{s+1}=x_{s}+v_{s}, \quad s=0,\dots,T-1, \\
        &t_{s+1}=t_{s}+F(t_{s},x_{s},v_{s}), \quad s=0,\dots,T-1.
    \end{split}
\end{equation*}
Using the update scheme in \eqref{eq:dual_prices} and \eqref{eq:georceh_update_scheme}, we denote the algorithm \textit{GEORCE-H} and display it in pseudo-code in Algorithm~\ref{al:georceh}, following the same approach as in \citep{georce} with soft line search to determine the step-size using backtracking with $\rho=0.5$.
\begin{algorithm}[!ht]
    \caption{GEORCE-H for Time-Dependent Finsler Metrics}
    \label{al:georceh}
    \begin{algorithmic}[1]
        \STATE{\textbf{Input}: $\mathrm{tol}$, $T$, $\rho$}
        \STATE{\textbf{Output}: Pregeodesic estimate $x_{0:T}$}
        \STATE{Set $x_{s}^{(0)} \leftarrow A+\frac{B-A}{T}s$, $v_{s}^{(0)} \leftarrow \frac{B-A}{T}$, $t_{s+1} \leftarrow t_{s}+F(t_{s},x_{s},v_{s})$ for $s=0,\dots,T$}
        \WHILE{stop criteria > $\mathrm{tol}$}
        \STATE{$\xi_{s} \leftarrow v_{s}^{\top}\frac{\partial G\left(t_{s}^{(i)},x_{s}^{(i)},v_{s}^{(i)}\right)}{\partial t}v_{s}$ for $s=1,\dots,T-1$} \\
        \STATE{Compute $\nu_{s}, \xi_{s}, \zeta_{s}, L_{s}^{(v)}, L_{s}^{(x)}, L_{s}^{(t)}, L_{s}^{(0)}, G_{s}$ using \eqref{eq:georceh_fixed_variables} for $\left(t_{s}^{(i)},x_{s}^{(i)}, v_{s}^{(i)}\right)$} \\
        \STATE{$\pi_{T-1} \leftarrow 0$} \\
        \STATE{$\pi_{s-1} \leftarrow \xi_{s}+\pi_{s}L^{(t)}+\pi_{s}$ for $s=T-1,\dots,1$} \\
        \STATE{$\mu_{T-1} \leftarrow \left(\sum_{s=0}^{T-1}G_{s}^{-1}\right)^{-1}\left(2(A-B)-\sum_{s=0}^{T-1}G_{s}^{-1}\left(\sum_{j>s}^{T-1}\left(\nu_{j}+\pi_{j}L_{j}^{(x)}\right)+\zeta_{s}+\pi_{s}L_{s}^{(v)}\right)\right)$} \\
        \STATE{$v_{s} \leftarrow -\frac{1}{2}G_{s}^{-1}\left(\mu_{T-1}+\sum_{j>t}^{T-1}\left(\nu_{j}+\pi_{j}L_{j}^{(x)}\right)+\zeta_{s}+\pi_{s}L_{s}^{(v)}\right)$ for $s=T-1,\dots,0$}
        \STATE{$t_{s+1} \leftarrow t_{s}+F(t_{s},x_{s},v_{s})$ for $s=0,\dots,T-1$}
        \STATE{$x_{s+1} \leftarrow x_{s}+v_{s}$ for $s=0,\dots,T-1$}
        \STATE{$\tilde{\alpha} \leftarrow 1$}
        \WHILE{$E(t_{0:T}, x_{0:T}, v_{0:T}) >  E\left(t^{(i)}_{0:T}, x^{(i)}_{0:T}, v^{(i)}_{0:T}\right)$}
        \STATE{$\tilde{\alpha} \leftarrow \rho \tilde{\alpha}$}
        \STATE{$x_{s+1} \leftarrow x_{s}+\tilde{\alpha} v_{s}+(1-\tilde{\alpha})v_{s}^{(i)}$ for $s=0,\dots,T-1$}
        \STATE{$u_{s+1} \leftarrow \tilde{\alpha} v_{s}+(1-\tilde{\alpha})v_{s}^{(i)}$ for $s=0,\dots,T-1$}
        \STATE{$t_{s+1} = t_{s}+F(t_{s},x_{s},v_{s})$ for $s=0,\dots,T-1$}
        \ENDWHILE
        \STATE{Set $v_{s}^{(i+1)} \leftarrow v_{s}$ for $s=0,\dots,T-1$}
        \STATE{Set $x_{s}^{(i+1)} \leftarrow x_{1}$ for $s=1,\dots,T-1$}
        \STATE{Set $t_{s}^{(i+1)} \leftarrow t_{s}$ for $s=0,\dots,T-1$}
        \STATE{$i \leftarrow i+1$}
        \ENDWHILE
        \STATE{return $\left(t_{s}^{(i)}, x_{s}^{(i)}\right)$ for $t=0,\dots,T-1$}
    \end{algorithmic}
\end{algorithm}

\paragraph{Convergence results} For \textit{GEORCE-H} in Algorithm~\ref{al:georceh}, we prove global convergence and local quadratic convergence similar to the original \textit{GEORCE}-algorithm \citep{georce} by generalizing the proofs to a time-dependent Finsler metric.

\begin{proposition}[Global Convergence]\label{prop:global_convergence}
    Let $E^{(i)}$ be the value of the discretized squared travel time functional in \eqref{eq:disc_energy} for the solution after iteration $i$ (with line search) in \textit{GEORCE-H}. If the starting point $\left(t^{(0)}, x^{(0)}, v^{(0)}\right)$ is feasible in the sense that the pregeodesic candidate has the correct start and end points, then the series $\left\{E^{(i)}\right\}_{i>0}$ will converge to a (local) minimum. 
\end{proposition}
\begin{proof}
    See Appendix~\ref{ap:global_convergence}.
\end{proof}

\begin{proposition}[Local Quadratic Convergence] \label{prop:quad_conv}
    Under sufficient regularity conditions of the discretized squared travel time functional (see Appendix~\ref{ap:assumptions}), it holds that if the discretized squared travel time functional in \eqref{eq:disc_energy} has a strongly unique (local) minimum point $z^{*}=\left(t^{*}, x^{*}, v^{*}\right)$ and locally the step-size satisfies $\alpha^{*}=1$, then \textit{GEORCE-H} has local quadratic convergence, i.e.
    \begin{equation*}
            \exists \epsilon>0:\quad \exists c>0: \quad \forall z^{(i)} \in B_{\epsilon}(z^{*}): \quad \norm{z^{(i+1)}-z^{*}} \leq c \norm{z^{(i)}-z^{*}}^{2},
    \end{equation*}
    where $z^{(i)}=(t^{(i)}, x^{(i)}, v^{(i)})$ is the solution from \textit{GEORCE-H} at iteration $i$, and $B_{\epsilon}\left(z^{*}\right)$ is the ball with radius $\epsilon>0$ and center $z^{*}$.
\end{proposition}
\begin{proof}
    See Appendix~\ref{ap:quadratic_convergence} for the proof and Appendix~\ref{ap:assumptions} for assumptions.
\end{proof}

\subsection{Estimating tack points for time-dependent Finsler metrics}
\label{subsec:tack_points_estimation}
We consider a fixed number of tack points $n_{\mathrm{tacks}}$ and a fixed order of time-dependent Finsler metrics $\{F_{t}^{i}\}_{i=0}^{n_{\mathrm{tacks}}}$. We consider the pregeodesics between each tack point parameterized by the unit interval, $s \in [0,1]$, where we split the unit interval by the number of tack points, i.e. $[0,s_{1}) \,\cup\, [s_{1},s_{2}) \, \cup \, \dots \, \cup \, [s_{i},s_{i+1}) \, \cup \, \dots \, \cup \, [s_{n_{\mathrm{tacks}}},1]$. Estimating time-minimizing tack points and curves can be formulated as
\begin{equation} \label{eq:min_tacking}
    \begin{split}
        \min_{z_{1:n_{\mathrm{tacks}}}} \quad &t(1) \\
        \text{s.t.} \quad & t'(s)= F^{i}(t(s),\gamma_{i}(s),\gamma'_{i}(s)), \quad \forall s \in [s_{i},s_{i+1}], \\
        &\gamma(0)=A,\gamma(1)=B,t(0)=0.
    \end{split}
\end{equation}
where $t$ denotes the travel time of the curve $\gamma$. To minimize travel time, each curve connecting two tack points must be a pregeodesic. We propose an iterative estimation of the pregeodesics between each tack point using Algorithm~\ref{al:georceh} for fixed tack points and then update the tack points by minimizing the integrated travel time along the pregeodesics, using standard optimization solvers. We present the algorithm in pseudo-code in Algorithm~\ref{al:tacking_optimization}. We use \textit{ADAM} \citep{kingma2017adammethodstochasticoptimization} to minimize travel time with respect to tack points for the integrated time along the estimated pregeodesics found by Algorithm~\ref{al:georceh}.
\begin{algorithm}[H]
    \SetAlgoLined
    \textbf{Input}: tol \\
    \textbf{Output}: Estimation of tack points $z_{1:n_{\mathrm{tacks}}}^{(i)}$, pregeodesics $\gamma_{1:(n_{\mathrm{tacks}}+1)}^{(i)}$ and travel time $t(1)$ \\
    $i = 0$ \\
    \While{$\mathrm{stop}$ $\mathrm{criteria}$ > $\mathrm{tol}$} {
        Compute pregeodesics, $\gamma_{1:(n_{\mathrm{tacks}}+1)}^{(i)}$, connecting each tack point using Algorithm~\ref{al:georceh} \\
        Update the tack points, $z_{1:n_{\mathrm{tacks}}}^{(i)}$, by minimizing \eqref{eq:min_tacking} \\ 
        $i \leftarrow i+1$
    }
    return $z_{1:n_{\mathrm{tacks}}}^{(i)}, \gamma_{1:(n_{\mathrm{tacks}}+1)}^{(i)}, t(1)$ \\
    \caption{Tack Points Estimation}
    \label{al:tacking_optimization}
\end{algorithm}

We note that in Algorithm~\ref{al:tacking_optimization} we estimate pregeodesics by minimizing the squared travel time functional, and the tack points by minimizing the total travel time. Overall, this will give a (local) minimum of the travel time. Unless the minimum is unique, there might be multiple solutions to the optimization problem. To investigate numerically other possible local minima, the optimization scheme can be run for different initialization of the candidate tack points. In Algorithm~\ref{al:tacking_optimization} we estimate the time-minimizing tack curves by sequentially estimating pregeodesics to a candidate tack point, and then update the tack point by minimizing the travel time. The entire tack curve, including the tack points, can be found by minimizing only the squared travel time using \textit{GEORCE-H} directly on both the curve points and tack points. Although the tack points and pregeodesics can be estimated in this way in one iteration, which significantly decreases the runtime, minimizing the squared travel time of the tack curve will not correspond in general to minimizing the total travel time. This is due to the fact that Proposition~\ref{prop:h_geodesics} is not valid when there is more than one Finsler metric in play, as explained at the beginning of the section.

\section{Numerical examples of Zermelo navigation with tacking}
\label{sec:examples}
In this section, we apply Algorithm~\ref{al:georceh} and Algorithm~\ref{al:tacking_optimization} to illustrate the Zermelo navigation tracks with tacking for a wide range of different time-dependent Finsler metrics. Since we will work in $\mathbb{R}^{2}$ throughout, for simplicity we will adopt the notation $(x,y) \coloneqq \left(x^{1},x^{2}\right)$ for the position coordinates, and $(u,v) \coloneqq \left(v^{1},v^{2}\right)$ for the velocity coordinates.

To construct explicit examples, we will consider the Zermelo metric \eqref{eq:zermelo_metric} derived in Section~\ref{subsec:ellipses}, whose indicatrices are shifted ellipses. We denote by $a,b$ the semi-axes of the ellipses, which are rotated by an angle $\theta$ in the clockwise direction and displaced by a ``wind'' vector field $W=(w_1,w_2)$, as in Section~\ref{subsec:ellipses}. In this context, however, $W$ will not represent an actual wind---we will use two different vector fields to illustrate the non-convex navigation---and should merely be interpreted as the displacement of the indicatrix field. For our purposes here, it is more convenient to write the components of $W$ with respect to the (orthonormal) basis defined by the axes of the ellipses. These components $c_1,c_2$ (relative to the $ab$-axes) relate to the Euclidean ones $w_1,w_2$ (relative to the $xy$-axes) via
\begin{equation*}
    \begin{split}
        w_1 & = c_1\cos{\theta}+c_2\sin{\theta}, \\
        w_2 & = -c_1\sin{\theta}+c_2\cos{\theta}.
    \end{split}
\end{equation*}
That is, $(w_1,w_2)$ is the result of rotating the vector $(c_1,c_2)$ clockwise by an angle $\theta$. Then, in terms of $a,b,c_{1},c_{2},\theta$, the expression of the Zermelo metric \eqref{eq:zermelo_metric} takes the expanded form
\begin{equation} \label{eq:shiftetd_ellipses_metric}
    \begin{split}
        F(t,x,y,u,v) &= \left(\frac{1}{a^{2}b^{2}-a^{2}c_{2}^{2}-b^{2}c_{1}^{2}}\right) \cdot \\
        &\phantom{xxxxx}\Biggl(-a^{2}c_{2}\left(u\sin\theta+v\cos\theta\right)-b^{2}c_{1}\left(u\cos\theta-v\sin\theta\right)\\
    &\phantom{xxxxxxxxxx}+\Bigl(a^{4}b^{2}\left(u\sin\theta+v\cos\theta\right)^{2}+a^{2}b^{4}\left(u\cos\theta-v\sin\theta\right)^{2} \\
        &\phantom{xxxxxxxxxxxxx}-a^{2}b^{2}c_{1}^{2}\left(u\sin\theta+v\cos\theta\right)^{2} \\
        &\phantom{xxxxxxxxxxxxx}+2a^{2}b^{2}c_{1}c_{2}\left(u\cos\theta-v\sin\theta\right)\left(u\sin\theta+v\cos\theta\right) \\
        &\phantom{xxxxxxxxxxxxx}-a^{2}b^{2}c_{2}^{2}\left(u\cos\theta-v\sin\theta\right)^{2}\Bigr)^{1/2}\Biggr)\quad .
    \end{split}
\end{equation}
Note that $a,b,c_{1},c_{2},\theta$ can be simple constants, but also functions of time and position. We summarize in Table~\ref{tab:example_params} the various choices of bi-metrics $(F^{\alpha}_t, F^{\beta}_t)$, i.e. the various choices of corresponding parameter functions, that we will use for the numerical experiments below.


\begin{table}[ht!]
\centering
\begin{tabular}{>{\raggedright\arraybackslash}p{2.5cm} *{5}{>{\centering\arraybackslash}p{2.5cm}}}
\midrule \midrule
{\bf{Parameters \  $\longrightarrow$}} & $a(t,x,y)$ & $b(t,x,y)$ & $c_{1}(t,x,y)$ & $c_{2}(t,x,y)$ & $\theta(t,x,y)$ \\
\midrule \midrule
{\bf{Constant:}}  \\ \midrule  $\phantom{xxxxx}F^{\alpha} $              & $2$          & $2$          & $(3/2)\cos(\pi/10)$          & $(3/2)\sin(\pi/10)$          & $\pi$          \\
\midrule    $\phantom{xxxxx}F^{\beta} $                 & $1$          & $1$          & $3/4$          & $0$          & $0$          \\
\midrule \midrule
{\bf{Time-only:}}   \\  \midrule   $\phantom{xxxxx}F_{t}^{\alpha} $        & $2 + (1/2)\sin(t)$ & $(3/4)a$ & $-\cos(t \pi/4)$ & $-\sin(t \pi/4)$ & $0$ \\
\midrule    $\phantom{xxxxx}F^{\beta} $                 & $7$          & $7/4$          & $-(3/2)$          & $0$          & $\pi/4$          \\
\midrule \midrule
{\bf{Time \& position:}}  &For $t \geq 0$:\\    \midrule $\phantom{xxxxx}F_{t}^{\alpha} $       & $1+t+x^{2} + y^{2}$ & $1+t+x^{2} + y^{2}$ & $1/2$ & $1/2$ & $0$ \\
 \midrule   $\phantom{xxxxx}F_{t}^{\beta} $             & $1+t+x^{2} + y^{2}$ & $1+t+x^{2} + y^{2}$ & $1/2$ & $-1/2$ & $0$ \\
\midrule \midrule
{\bf{Position-only:}}  & For $R(y) > 0$: \\  \midrule   $\phantom{xxxxx}F^{\alpha} $       & $R(y)$ & $R(y)$ & $(1/2)R(y)$ & $(1/2)R(y)$ & $0$ \\
 \midrule   $\phantom{xxxxx}F^{\beta} $                & $R(y)$ & $R(y)$ & $(1/2)R(y)$ & $-(1/2)R(y)$ & $0$ \\
\midrule \midrule
\end{tabular}
\caption{Parameters for \eqref{eq:shiftetd_ellipses_metric} to obtain the different, fairly simple, explicit  Finsler bi-metrics $(F_{t}^{\alpha}, F_{t}^{\beta})$ that will be used in the numerical experiments below.}
\label{tab:example_params}
\end{table}

\paragraph{Constant bi-metric} As a first simple example, we consider two Finsler metrics $F^\alpha, F^\beta$ independent of position and time, using the ``Constant'' parameters in Table~\ref{tab:example_params} for the expression \eqref{eq:shiftetd_ellipses_metric}.
In this case, the metrics are Minkowski norms and therefore the pregeodesics will be straight lines. We illustrate the indicatrix field in Figure~\ref{fig:direction_only_indicatrix}, which is similar to the speed profiles seen in sailing \citep{miles2022a}, modeled as double indicatrices, where here we consider two Finsler metrics with different scaling.

\begin{figure}
    \centering
    \includegraphics[width=0.4\textwidth]{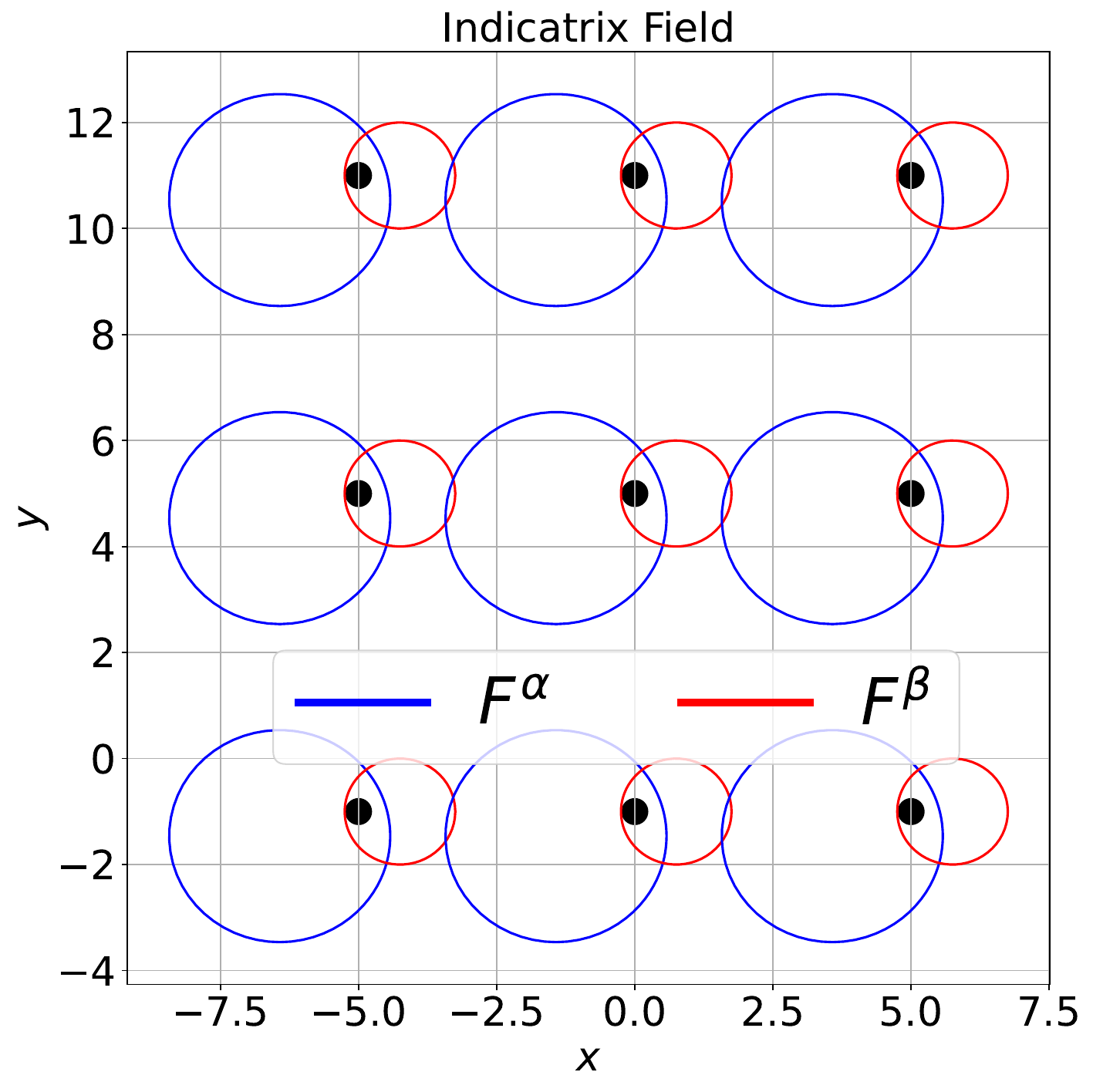}
    \caption{The indicatrix field for the ``Constant'' bi-metric.}
    \label{fig:direction_only_indicatrix}
\end{figure}

We estimate the corresponding tacking curves and points using Algorithm~\ref{al:tacking_optimization} for $n_{\mathrm{tacks}}=1,2,3,4$ tack points between $A=(0,0)$ and $B=(2,8)$, as illustrated in Figure~\ref{fig:direction_only}. The blue straight segment from $A$ to $B$ corresponds to the optimal trajectory using only $F^{\alpha}$, while the red straight segment corresponds to the optimal trajectory using only $F^{\beta}$. When there are multiple tack points, the metrics change at each tacking, which is marked by a dotted point. As expected from the theoretical results in the previous sections, Figure~\ref{fig:direction_only} shows that it is not optimal to go directly towards the target, but rather to go in the largest directions of the indicatrix fields---technically, the directions prescribed by Theorem~\ref{thm:time_min}---and change metric similar to a sailboat traveling upwind (see e.g. \citep{miles2022a}). We further see that as the number of tack points increases, the tacking curve gets closer to the straight line. However, the total travel time remains the same for any number of tack points (as predicted by Theorem~\ref{thm:time_min}) and it does not converge to the travel time of the straight line pregeodesic using either $F^{\alpha}$ or $F^{\beta}$.

\begin{figure}
    \centering
    \includegraphics[width=1.0\textwidth]{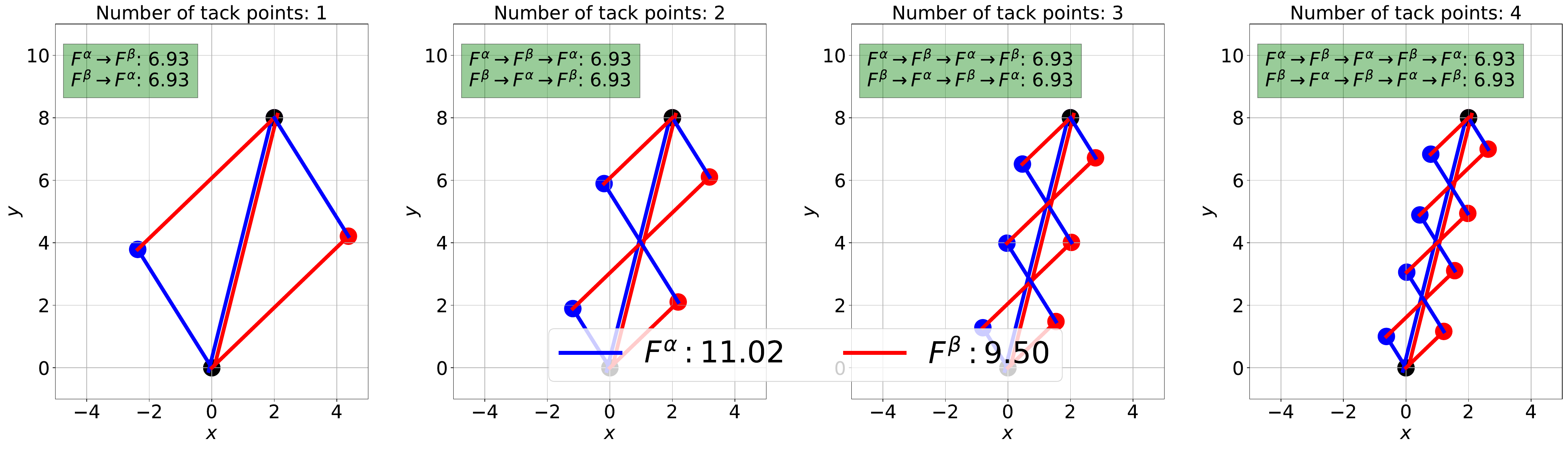}
    \caption{The figure shows the corresponding tacking curves and points using Algorithm~\ref{al:tacking_optimization} for $n_{\mathrm{tacks}}=1,2,3,4$ tack points from left to right for the ``Constant'' bi-metric between the points $(0,0)$ and $(2,8)$. The blue curves are computed using $F^{\alpha}$, while the red curves are computed using $F^{\beta}$. The blue and red labels at the bottom indicate the travel time without tacking, while the green boxes indicate the common total travel time for the different tack curves. Tack points are highlighted by circular dots with the same color as the ``incoming'' segment. Note also, that the common turning angle at the tack points is not $\pi/2$.}
    \label{fig:direction_only}
\end{figure}

\paragraph{Time-only dependent bi-metric} To see the application of Algorithm~\ref{al:tacking_optimization} in the non-constant case, we consider now a time-only dependent Finsler metric, whose pregeodesics are not necessarily straight lines. We consider again the general shifted ellipse indicatrix field in \eqref{eq:shiftetd_ellipses_metric}, with the ``Time-only'' parameters from Table~\ref{tab:example_params}.
Note that $F^{\alpha}_{t}$ is a time-only dependent Finsler metric, while $F^{\beta}$ is a constant metric. We show the indicatrix field for different times in Figure~\ref{fig:time_only_indicatrix}, where we see that the indicatrix field of $F^{\alpha}_{t}$ changes in time, while the indicatrix field of $F^{\beta}$ is constant.

\begin{figure}
    \centering
    \includegraphics[width=1.0\textwidth]{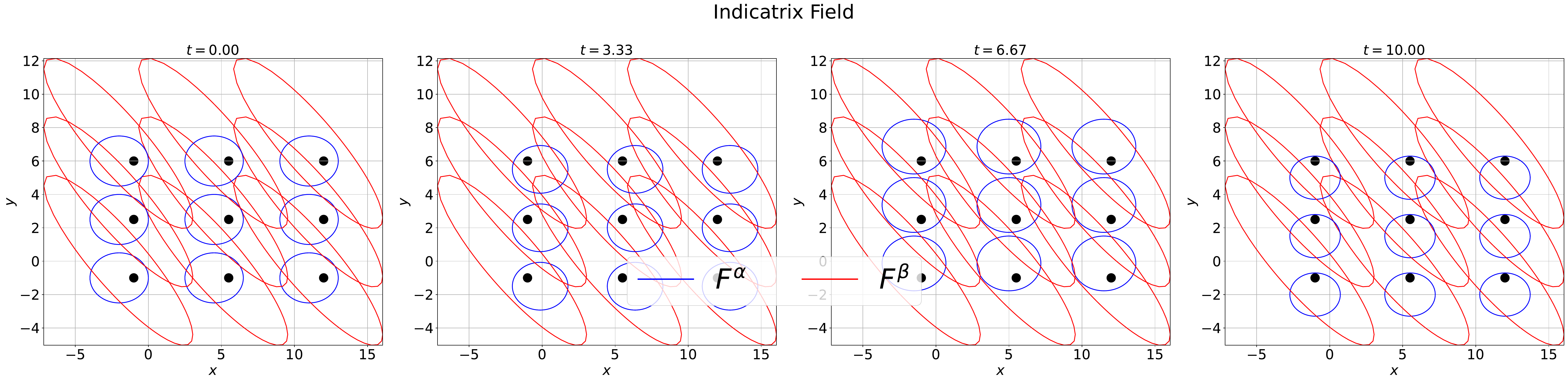}
    \caption{The indicatrix field for the ``Time-only'' bi-metric.}
    \label{fig:time_only_indicatrix}
\end{figure}

In Figure~\ref{fig:time_only} we display the corresponding tack curves for $n_{\mathrm{tacks}}=1$, where $F^{\alpha}_{t}$ is used for the blue curve segments, while $F^{\beta}$ is used for the red curve segments. We see again that the tack curves have shorter travel times than the direct pregeodesics without tacking. We also see that the travel time is not invariant to the order of the metrics, where in this case using first $F^{\beta}$ and then $F^{\alpha}_{t}$ gives the tack curve with the lowest total travel time.

\begin{figure}
    \centering
    \includegraphics[width=0.4\textwidth]{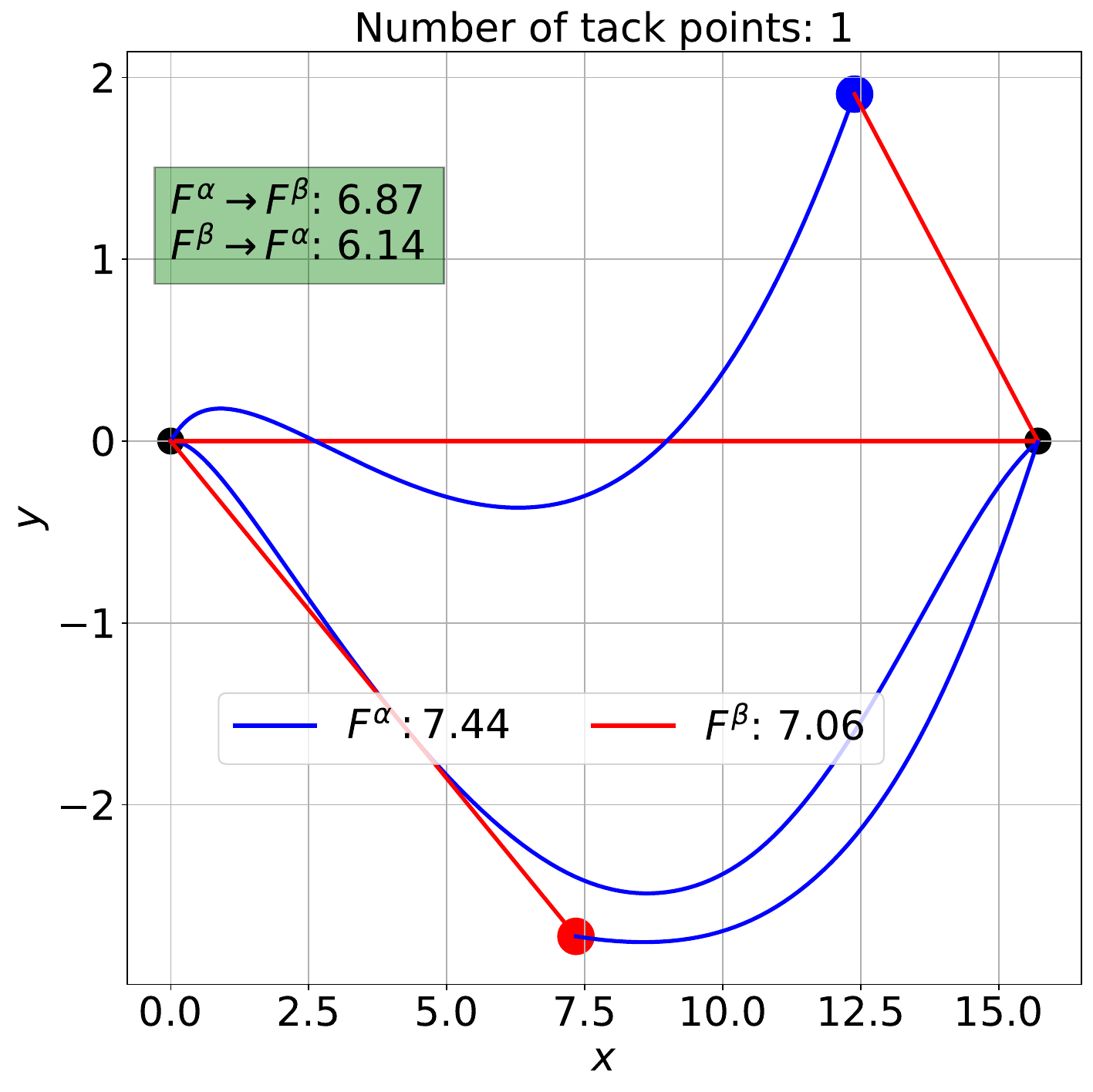}
    \caption{Tacking curves and points using Algorithm~\ref{al:tacking_optimization} for the ``Time-only'' bi-metric between the points $(0,0)$ and $(5\pi,0)$. The travel time without tacking can be seen in the label, while the green box shows the total travel time for the tack curves. It is clearly an advantage here to choose the metric $F^{\beta}$ first. Moreover, the corresponding turning angle at the tack point is conveniently smaller.}
    \label{fig:time_only}
\end{figure}

\paragraph{Time- and position-dependent  bi-metric} In this example, we will consider the most general case with time- and position-dependent Finsler metrics. We set the ``Time \& position'' parameters from Table~\ref{tab:example_params} in \eqref{eq:shiftetd_ellipses_metric}.
Observe that both $F^{\alpha}_t$ and $F^{\beta}_t$ depend on time and position and are symmetric with respect to each other.

In Figure~\ref{fig:time_position_indicatrix} we display the indicatrix field of $F^{\alpha}_t$ for different values of time. In Figure~\ref{fig:time_position} we show the computed direct pregeodesics using Algorithm~\ref{al:georceh}, as well as the tack curves with one tack point using Algorithm~\ref{al:tacking_optimization}. We see that the tack curves and pregeodesics are symmetric, and the tack curves have the lowest travel time as expected.

\begin{figure}
    \centering
    \includegraphics[width=1.0\textwidth]{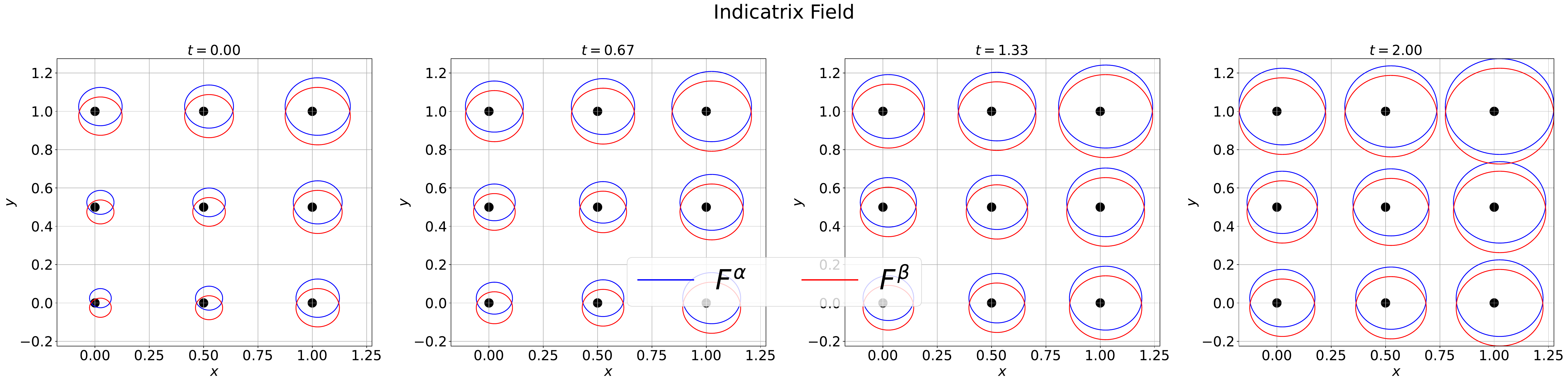}
    \caption{The indicatrix field for the ``Time \& position'' bi-metric.}
    \label{fig:time_position_indicatrix}
\end{figure}

\begin{figure}
    \centering
    \includegraphics[width=0.4\textwidth]{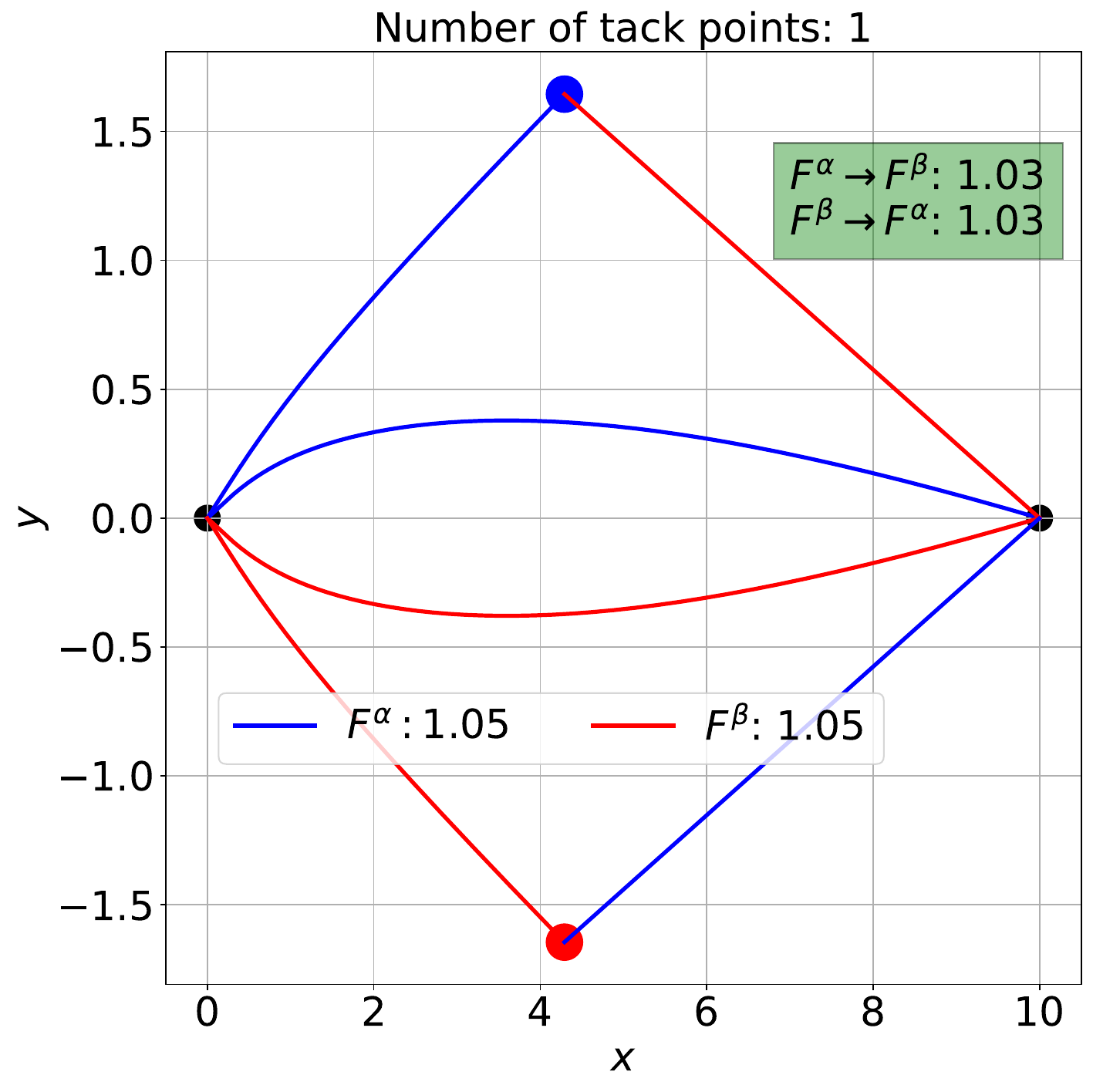}
    \caption{Tacking curves and points using Algorithm~\ref{al:tacking_optimization} for the ``Time \& position'' bi-metric. The (common) travel time for the direct pregeodesics from $(0,1)$ to $(10,1)$ without tacking can be seen in the label of the figure, while the green box shows the (common) total travel time for the tack curves. The displayed symmetry follows from the symmetry of the two metrics that we  use for the bi-metric construction.}
    \label{fig:time_position}
\end{figure}


\paragraph{Position-only dependent bi-metric}
As a final example, we consider the bi-metric expressed in \eqref{eq:poincarre_north} below. It is directly motivated by the long distance flight paths of the Juan Fernandez petrels (west of Chile)  that have been GPS-monitored and documented by Clay et al. in \cite[Figure 2]{Clay2023}. In this simplistic model the bi-metric indicatrix field, i.e. the ground speed profile of the petrels, consists of two shifted circles in each tangent plane (as in Section~\ref{subsec:circles}). The radius $R(y)$ of the circles depends only on the position coordinate $y$ and is modeled by $R(y)= 3\arctan(y)$. This is at the same time a model (modulo a constant factor and orientation) of the wind speed profile in the domain of the petrels' flight paths seen in \cite[Figure 2]{Clay2023}. We apply now the modeling assumption that the ground speed profile indicatrix is a polar curve which is everywhere homothetically scaled by the wind speed on each given location---an assumption, that is made plausible (for the similar flights of the wandering albatross) by Richardson et al. in \cite[Figure 8]{Richardson2018}.

In the last row of Table~\ref{tab:example_params} we thus choose $R(y)= 3\arctan(y)$ and obtain the following compact expressions for the Zermelo metric in \eqref{eq:shiftetd_ellipses_metric}:
\begin{equation} \label{eq:poincarre_north}
    \begin{split}
        F^{\alpha}(x,y,u,v) &= \frac{\sqrt{3u^{2}+3v^{2}-2uv}+v-u}{3\arctan(y)}, \\
        F^{\beta}(x,y,u,v) &= \frac{\sqrt{3u^{2}+3v^{2}+2uv}-v-u}{3\arctan(y)}. \\
    \end{split}
\end{equation}
We illustrate the indicatrix field of the bi-metric in Figure~\ref{fig:poincarre_north_indicatrix}. Note that as $y$ increases, the indicatrix field increases correspondingly in size. Since $R(y)$ is small in the domain where $y$ is small, it takes a long travel time to move from $A$ to $B$ in that domain. 

For explicit comparison, suppose the track chosen from $A=(0,1)$ to $B=(40,1)$ is the straight line $\gamma(s) = (s, 1)$, $s\in [0, 40]$, irrespective of both the Zermelo navigation advantage and the tacking advantage. Then the travel time would be---using any one of the two metrics $F^{\alpha}$ or $F^{\beta}$, because $F^{\alpha}(s,1, 1,0) = F^{\alpha}(s,1, 1,0)$ for all $s$:
\begin{equation}\label{eq:StraightT}
\mathcal{T}[\gamma] =\mathcal{T}^{\alpha}[\gamma] = \mathcal{T}^{\beta}[\gamma] = \int_{0}^{40}   \frac{1}{F^{\alpha}(s,1, 1,0)}\, \mathrm{d}s = 128.7.
\end{equation}

In accordance with our main results, this travel time can be reduced considerably by taking full advantage of the Zermelo navigation option with tacking---as we illustrate explicitly below. 

\begin{figure}
    \centering
    \includegraphics[width=0.4\textwidth]{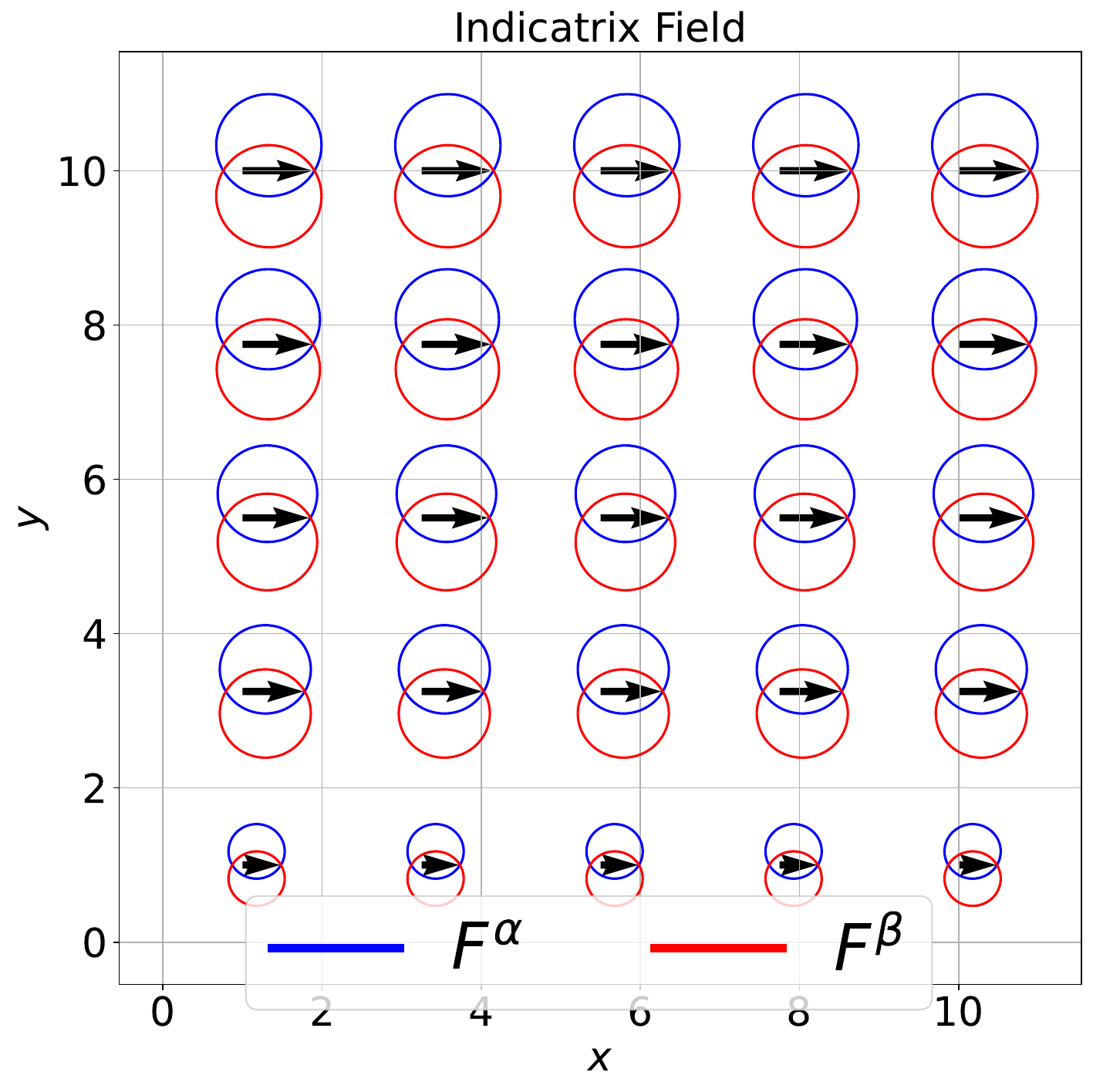}
    \caption{The indicatrix field for the ``Position-only'' bi-metric in \eqref{eq:poincarre_north}.}
    \label{fig:poincarre_north_indicatrix}
\end{figure}

In Figure~\ref{fig:poincarre_north} we show the travel time optimal pregeodesics connecting $A=(0,1)$ and $B=(40,1)$, using Algorithm~\ref{al:georceh}, and the corresponding estimated tack curves given by Algorithm~\ref{al:tacking_optimization}. The various allowed number of tack points give essentially the same time-optimal tracks and the same optimal travel time as the case with only one tack point. In all cases, as expected, this total travel time, $6.73$, is significantly shorter than the travel time $\mathcal{T}[\gamma]$ along the straight line, and is indeed also shorter than the (in this case identical) travel time, $8.24$, obtained by the direct pregeodesic tracks without tack points from $A$ to $B$.

Observe that allowing more tack points should never increase the total travel time, since doing so increases the degrees of freedom. When additional tack points do not reduce the travel time---as in this example---we see a coalescence of some of the tack points, effectively reducing the trajectory to one with fewer tackings. Specifically, Figure~\ref{fig:poincarre_north} shows that the algorithm either compresses two tack points into a single one---in the third and fourth subfigures---or selects the endpoint $B$ as a tack point---in the second and fourth subfigures. The latter is due to the fact that---in this scenario---with an even number of tack points, the last segment of the track would otherwise be forced to be an $F^{\alpha}$-pregeodesic, which is more time-consuming around $B$ than a corresponding end-segment consisting of an $F^{\beta}$-pregeodesic.

A more refined analysis shows a much greater variation in the set of \emph{almost-time-optimal paths} for the bi-metric under consideration---in particular among the tracks with many tack points in the region with high values of $y$, where $R(y)$, and hence the bi-metric, is essentially constant, as in Figure \ref{fig:direction_only}. In fact, if we allow---still in the very concrete and simple setting of \eqref{eq:poincarre_north}---a $7\%$ increase in total travel time from the optimal value, then the possibility (and relative advantage) of applying tacking shows a significant variation in the ensuing set of relaxed tack curves. This is explicitly clear from the displays of selected such tack curves in Figure \ref{fig:NewPoincare}, which, as expected, shows tacking phenomena that are qualitatively very similar to those observed in the constant bi-metric case in Figure \ref{fig:direction_only}.

As a final remark, and as noted at the beginning of Section~\ref{sec:computations}, in general we cannot guarantee that the numerical solutions obtained are true global minima. This naturally raises the interesting standard question: When is a local minimum (or a local stationary point) a global minimum? In our present setting, one first step in the direction of an answer is to consider the second variation of the travel time functional---as done in the general-relativistic Lorentzian case by Perlick in \cite{perlick1998a}. The extra dependence of the second variation on not only the flag curvatures of the metrics involved but also on the positioning of the tack points is clearly an interesting topic for future theoretical and numerical work on time- and position-dependent Zermelo navigation with tacking.


\begin{figure}
    \centering
    \includegraphics[width=1.0\textwidth]{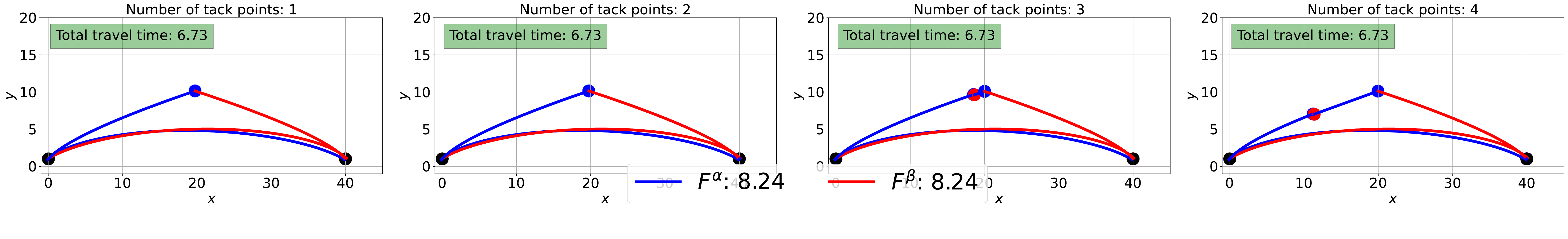}
    \caption{The computed tack curves with $n_{\mathrm{tacks}}=1,2,3,4$ tack points using Algorithm~\ref{al:tacking_optimization} for the ``Position-only'' bi-metric in \eqref{eq:poincarre_north} between the points $(0,1)$ and $(40,1)$. The blue curves are computed using $F^{\alpha}$, while the red curves are computed using $F^{\beta}$. The (common) total travel time of the respective tack curves is shown in the green boxes. Note that the travel time $\mathcal{T}[\gamma] = 128.7$ found in \eqref{eq:StraightT} for the direct straight line travel from $A$ to $B$ is significantly larger than the (common) travel time along the $F^{\alpha}$- and $F^{\beta}$-pregeodesics, and that these travel times are reduced even further when tacking is allowed.} 
    \label{fig:poincarre_north}
\end{figure}

\begin{figure}
    \centering
    \includegraphics[width=1.0\textwidth]{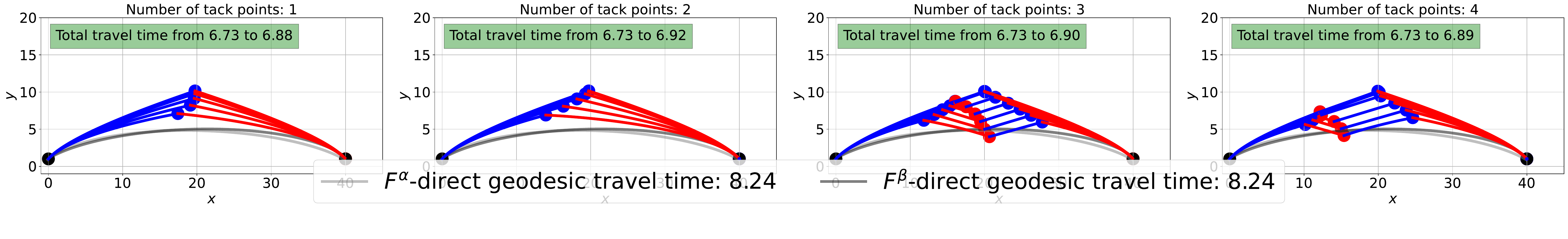}
    \caption{Tack curves with only close-to-optimal travel times for the same bi-metric \eqref{eq:poincarre_north} as in Figure \ref{fig:poincarre_north}. The large increase in the span of possible track curves comes with only a small extra cost in travel time. In this case, the tracks take advantage of the flexibility that is automatically given in the domain where the bi-metric is almost constant---a flexibility that is already revealed by the constant bi-metric in Figure \ref{fig:direction_only}.} 
    \label{fig:NewPoincare}
\end{figure}


\section*{Acknowledgments}
EPR was supported by: Project PID2021-124157NB-I00, funded by MCIN/AEI/10.13039/501100011033/ and ``ERDF A way of making Europe"; Project PID2020-116126GB-I00, funded by MCIN/AEI/10.13039/501100011033/; The framework IMAG-Mar\'{i}a de Maeztu grant CEX2020-001105-M/AEI/10.13039/501100011033; ``Ayudas a proyectos para el desarrollo de investigaci\'{o}n cient\'{i}fica y t\'{e}cnica por grupos competitivos (Comunidad Aut\'{o}noma de la Regi\'{o}n de Murcia)'', included in ``Programa Regional de Fomento de la Investigaci\'{o}n Cient\'{i}fica y T\'{e}cnica (Plan de Actuaci\'{o}n 2022)'' of Fundaci\'{o}n S\'{e}neca-Agencia de Ciencia y Tecnolog\'{i}a de la Regi\'{o}n de Murcia, REF. 21899/PI/22; and ``Ayudas para la Formación del Profesorado Universitario (FPU)'' from the Spanish Government.

This research was carried out during EPR’s stays at DTU Compute in the summers of 2023–2025, partially funded by ``Ayudas de Movilidad para Beneficiarios del Subprograma FPU'' from the Spanish Government (2023), DTU (2024), and ``Ayudas de Movilidad de Excelencia para Docentes e Investigadores de la Universidad de Oviedo'', cofinanced by Banco Santander (2025).

The authors warmly thank Professor Miguel Ángel Javaloyes (Universidad de Murcia) for his valuable comments and insightful discussions during the early stages of this work.


\clearpage
\bibliographystyle{plain}
\bibliography{main} 

\begin{thebibliography}{10}

\bibitem{antonelli2003}
P.~L. Antonelli, A.~B\'{o}na, and M.~A. Slawi\'{n}ski.
\newblock Seismic rays as {F}insler geodesics.
\newblock {\em Nonlinear Anal. Real World Appl.}, 4(5):711--722, 2003.

\bibitem{BCS}
D.~Bao, S.-S. Chern, and Z.~Shen.
\newblock {\em An introduction to {R}iemann-{F}insler geometry}, volume 200 of
  {\em Graduate Texts in Mathematics}.
\newblock Springer-Verlag, New York, 2000.

\bibitem{BaoRob2}
D.~Bao, C.~Robles, and Z.~Shen.
\newblock Zermelo navigation on {R}iemannian manifolds.
\newblock {\em J. Differential Geom.}, 66(3):377--435, 2004.

\bibitem{biferale2019a}
L.~Biferale, F.~Bonaccorso, M.~Buzzicotti, P.~Clark Di~Leoni, and
  K.~Gustavsson.
\newblock Zermelo's problem: {O}ptimal point-to-point navigation in 2{D}
  turbulent flows using reinforcement learning.
\newblock {\em Chaos}, 29(10):103138, 2019.

\bibitem{bonnard2023a}
B.~Bonnard, O.~Cots, and B.~Wembe.
\newblock Zermelo navigation problems on surfaces of revolution and geometric
  optimal control.
\newblock {\em ESAIM Control Optim. Calc. Var.}, 29:376--814, 2023.

\bibitem{buell1996a}
I.~Büll, T.~Ramcke, and U.~Mares.
\newblock Per {C}omputer hart am {W}ind oder die {P}hysik des {S}egelns.
\newblock {\em Physik in Unserer Zeit}, 27(4):154--160, 1996.

\bibitem{CJS2011}
E.~Caponio, M.~\'{A}. Javaloyes, and M.~S\'{a}nchez.
\newblock On the interplay between {L}orentzian causality and {F}insler metrics
  of {R}anders type.
\newblock {\em Rev. Mat. Iberoam.}, 27(3):919–952, 2011.

\bibitem{caponio2021a}
E.~Caponio, M.~Á. Javaloyes, and M.~Sánchez.
\newblock Wind {F}inslerian {S}tructures: From {Z}ermelo's {N}avigation to the
  {C}ausality of {S}pacetimes.
\newblock {\em Mem. Amer. Math. Soc.}, 300(2):1501, 2024.

\bibitem{Caratheodory1965}
C.~Carath\'eodory.
\newblock {\em Calculus of variations and partial differential equations of the
  first order. {P}art {I}: {P}artial differential equations of the first
  order}.
\newblock Holden-Day, Inc., San Francisco-London-Amsterdam, 1965.
\newblock Translated from the {G}erman by R. B. Dean and J. J. Brandstatter.

\bibitem{Caratheodory1967}
C.~Carath\'eodory.
\newblock {\em Calculus of variations and partial differential equations of the
  first order. {P}art {II}: {C}alculus of variations}.
\newblock Holden-Day, Inc., San Francisco-London-Amsterdam, 1967.
\newblock Translated from the {G}erman by R. B. Dean and J. J. Brandstatter.

\bibitem{Clay2023}
T.~A. Clay, P.~Hodum, E.~Hagen, and M.~de~L. Brooke.
\newblock Adjustment of foraging trips and flight behaviour to own and partner
  mass and wind conditions by a far-ranging seabird.
\newblock {\em Animal {B}ehaviour}, 198:165--179, 2023.

\bibitem{hays2014a}
G.~C. Hays, A.~Christensen, S.~Fossette, G.~Scofield, J.~Talbot, and
  P.~Mariani.
\newblock Route optimisation and solving {Z}ermelo’s navigation problem
  during long distance migration in cross flows.
\newblock {\em Ecology Letters}, 17(2):137--143, 2014.

\bibitem{JMPS}
M.~\'{A}. Javaloyes, S.~Markvorsen, E.~Pend\'{a}s-Recondo, and M.~S\'{a}nchez.
\newblock Generalized {F}ermat's principle and {S}nell's law for cone
  structures and applications.
\newblock In preparation.

\bibitem{JP}
M.~\'{A}. Javaloyes and E.~Pend\'{a}s-Recondo.
\newblock {L}ightlike {H}ypersurfaces and {T}ime-{M}inimizing {G}eodesics in
  {C}one {S}tructures.
\newblock In A.~L. Albujer, M.~Caballero, A.~García-Parrado, J.~Herrera, and
  R.~Rubio, editors, {\em {D}evelopments in {L}orentzian {G}eometry}, volume
  389 of {\em Springer Proceedings in Mathematics \& Statistics}, pages
  159--173. {S}pringer {N}ature {S}witzerland AG, Cham, 2023.

\bibitem{JPS2023}
M.~\'{A}. Javaloyes, E.~Pend\'{a}s-Recondo, and M.~S\'{a}nchez.
\newblock An {A}ccount on {L}inks {B}etween {F}insler and {L}orentz
  {G}eometries for {R}iemannian {G}eometers.
\newblock In A.~Alarc\'{o}n, V.~Palmer, and C.~Rosales, editors, {\em New
  Trends in Geometric Analysis}, volume~10 of {\em RSME Springer Series}, pages
  259--303. {S}pringer {N}ature {S}witzerland AG, Cham, 2023.

\bibitem{JPS2025}
M.~\'{A}. Javaloyes, E.~Pend\'{a}s-Recondo, and M.~S\'{a}nchez.
\newblock {G}ielis {S}uperformula and {W}ildfire {M}odels.
\newblock In J.~Gielis and S.~Brasili, editors, {\em Proceedings of the 2nd
  International Symposium on Square Bamboos and the Geometree, ISSBG 2023},
  pages 167--184. {G}eniaal B. V., Antwerp, 2025.

\bibitem{JS2014}
M.~\'{A}. Javaloyes and M.~S\'{a}nchez.
\newblock On the definition and examples of {F}insler metrics.
\newblock {\em Ann. Sc. Norm. Super. Pisa Cl. Sci.}, 13(3):813–858, 2014.

\bibitem{JS2020}
M.~\'{A}. Javaloyes and M.~S\'{a}nchez.
\newblock On the definition and examples of cones and {F}insler spacetimes.
\newblock {\em Rev. R. Acad. Cienc. Exactas Fís. Nat. Ser. A Mat. RACSAM},
  114:30, 2020.

\bibitem{javaloyes2021a}
M.~Á. Javaloyes, E.~Pendás-Recondo, and M.~Sánchez.
\newblock Applications of cone structures to the anisotropic rheonomic
  {H}uygens’ principle.
\newblock {\em Nonlinear Anal.}, 209:112337, 2021.

\bibitem{Javaloyes2023a}
M.~Á. Javaloyes, E.~Pendás-Recondo, and M.~Sánchez.
\newblock A {G}eneral {M}odel for {W}ildfire {P}ropagation with {W}ind and
  {S}lope.
\newblock {\em SIAM J. Appl. Algebra Geom.}, 7(2):414--439, 2023.

\bibitem{kingma2017adammethodstochasticoptimization}
D.~P. Kingma and J.~L. Ba.
\newblock Adam: {A} {M}ethod for {S}tochastic {O}ptimization.
\newblock {\em 3rd International Conference on Learning Representations, ICLR
  2015, Conference Track Proceedings}, 2015.

\bibitem{larsson2000a}
L.~Larsson and R.~E. Eliasson.
\newblock {\em Principles of {Y}acht {D}esign}.
\newblock Adlard Coles Nautical, 2000.

\bibitem{levi-civita1931a}
T.~Levi-Civita.
\newblock Über {Z}ermelo's {L}uftfahrtproblem.
\newblock {\em Z. Angew. Math. Mech.}, 11(4):314--322, 1931.

\bibitem{Garcia2025}
D.~D. Lujerio~Garcia, N.~M. Solórzano~Chávez, M.~A. Molina~Morales, and B.~M.
  Cerna~Maguiña.
\newblock Euclidean space perturbed by a constant vector field and its relation
  to a zermelo navigation problem.
\newblock {\em Selecciones Matemáticas}, 12(1):15--32, 2025.

\bibitem{mania1937}
B.~Mani\`{a}.
\newblock Sopra un problema di navigazione di {Z}ermelo.
\newblock {\em Math. Ann.}, 113:584--599, 1937.

\bibitem{marchaj1982a}
C.~A. Marchaj.
\newblock {\em Sailing {T}heory and {P}ractice}.
\newblock Dodd, Mead \& Company, 1982.

\bibitem{SM2016}
S.~Markvorsen.
\newblock A {F}insler geodesic spray paradigm for wildfire spread modelling.
\newblock {\em Nonlinear Anal. Real World Appl.}, 28:208--228, 2016.

\bibitem{MP2023}
S.~Markvorsen and E.~Pend\'{a}s-Recondo.
\newblock Snell’s law revisited and generalized via {F}insler geometry.
\newblock {\em Int. J. Geom. Methods Mod. Phys.}, 20(8):2350138, 2023.

\bibitem{mclaren2014a}
J.~D. McLaren, J.~Shamoun-Baranes, A.~M. Dokter, R.~H.~G. Klaassen, and
  W.~Bouten.
\newblock Optimal orientation in flows: providing a benchmark for animal
  movement strategies.
\newblock {\em Journal of the Royal Society Interface}, 11(99):0588, 2014.

\bibitem{miles2022a}
C.~Miles and A.~Vladimirsky.
\newblock Stochastic {O}ptimal {C}ontrol of a {S}ailboat.
\newblock {\em IEEE Control Syst. Lett.}, 6:2048--2053, 2022.

\bibitem{peal1880a}
S.~E. Peal.
\newblock Soaring of {B}irds.
\newblock {\em Nature}, 23(575):10--11, 1880.

\bibitem{EPR2024}
E.~Pend\'{a}s-Recondo.
\newblock On the application of {L}orentz-{F}insler geometry to model wave
  propagation.
\newblock {\em arXiv:2408.03206 [math.DG]}, 2024.

\bibitem{perlick2006a}
V.~Perlick.
\newblock Fermat principle in {F}insler spacetimes.
\newblock {\em Gen. Relativity Gravitation}, 38(2):365--380, 2006.

\bibitem{perlick1998a}
V.~Perlick and P.~Piccione.
\newblock A general-relativistic fermat principle for extended light sources
  and extended receivers.
\newblock {\em Gen. Relativity Gravitation}, 30(10):1461--1476, 1998.

\bibitem{pinti2020a}
J.~Pinti, A.~Celani, U.~H. Thygesen, and P.~Mariani.
\newblock Optimal navigation and behavioural traits in oceanic migrations.
\newblock {\em Theoretical Ecology}, 13(4):583--593, 2020.

\bibitem{piro2022a}
L.~Piro, B.~Mahault, and R.~Golestanian.
\newblock Optimal navigation of microswimmers in complex and noisy
  environments.
\newblock {\em New J. Phys.}, 24(9):093037, 2022.

\bibitem{piro2021a}
L.~Piro, E.~Tang, and R.~Golestanian.
\newblock Optimal navigation strategies for microswimmers on curved manifolds.
\newblock {\em Phys. Rev. Research}, 3(2):023125, 2021.

\bibitem{Pokhrel2023}
S.~Pokhrel and S.~A. Eisa.
\newblock A novel hypothesis for how albatrosses optimize their flight physics
  in real-time: an extremum seeking model and control for dynamic soaring.
\newblock {\em Bioinspiration \& {B}iomimetics}, 18(1):016014, 2023.

\bibitem{pueschl2018a}
W.~Püschl.
\newblock High-speed sailing.
\newblock {\em European J. Phys.}, 39(4):044002, 2018.

\bibitem{rayleigh1883a}
Lord Rayleigh.
\newblock The {S}oaring of {B}irds.
\newblock {\em Nature}, 27(701):534--535, 1883.

\bibitem{richards1993a}
G.~D. Richards.
\newblock The {P}roperties of {E}lliptical {W}ildfire {G}rowth for {T}ime
  {D}ependent {F}uel and {M}eteorological {C}onditions.
\newblock {\em Combust. Sci. and Tech.}, 95(1-6):357--383, 1993.

\bibitem{Richardson2011}
P.~L. Richardson.
\newblock How do albatrosses fly around the world without flapping their wings?
\newblock {\em Progress in Oceanography}, 88(1-4):46--58, 2011.

\bibitem{Richardson2017}
P.~L. Richardson.
\newblock Da {V}inci's observations of soaring birds.
\newblock {\em Phys. Today}, 70(11):78--79, 2017.

\bibitem{Richardson2019}
P.~L. Richardson.
\newblock Leonardo da {V}inci's discovery of the dynamic soaring by birds in
  wind shear.
\newblock {\em Notes and {R}ecords: {T}he {R}oyal {S}ociety {J}ournal of the
  {H}istory of {S}cience}, 73(3):285--301, 2019.

\bibitem{Richardson2018}
P.~L. Richardson, E.~D. Wakefield, and R.~A. Phillips.
\newblock Flight speed and performance of the wandering albatross with respect
  to wind.
\newblock {\em Movement {E}cology}, 6:3, 2018.

\bibitem{georce}
F.~M. Rygaard and S.~Hauberg.
\newblock {GEORCE}: {A} {F}ast {N}ew {C}ontrol {A}lgorithm for {C}omputing
  {G}eodesics.
\newblock {\em arXiv:2505.05961 [math.DG]}, 2025.

\bibitem{sachs2005a}
G.~Sachs.
\newblock Minimum shear wind strength required for dynamic soaring of
  albatrosses.
\newblock {\em Ibis}, 147(1):1--10, 2005.

\bibitem{sachs2016a}
G.~Sachs.
\newblock In-flight measurement of upwind dynamic soaring in albatrosses.
\newblock {\em Progress in Oceanography}, 142:47--57, 2016.

\bibitem{sachs2013a}
G.~Sachs, J.~Traugott, A.~P. Nesterova, and F.~Bonadonna.
\newblock Experimental verification of dynamic soaring in albatrosses.
\newblock {\em Journal of Experimental Biology}, 216(22):4222--4232, 2013.

\bibitem{serres2009}
U.~Serres.
\newblock On {Z}ermelo-like problems: {G}auss-{B}onnet inequality and {E}.
  {H}opf theorem.
\newblock {\em J. Dyn. Control Syst.}, 15(1):99--131, 2009.

\bibitem{shen2003a}
Z.~Shen.
\newblock Finsler {M}etrics with {K} = 0 and {S} = 0.
\newblock {\em Canad. J. Math.}, 55(1):112--132, 2003.

\bibitem{techy2011a}
L.~Techy.
\newblock Optimal navigation in planar time-varying flow: {Z}ermelo's problem
  revisited.
\newblock {\em Intelligent Service Robotics}, 4(4):271--283, 2011.

\bibitem{thorne2023a}
L.~H. Thorne, T.~A. Clay, R.~A. Phillips, L.~G. Silvers, and E.~D. Wakefield.
\newblock Effects of wind on the movement, behavior, energetics, and life
  history of seabirds.
\newblock {\em Marine Ecology Progress Series}, 723:73--117, 2023.

\bibitem{mises1931a}
R.~v.~Mises.
\newblock Zum {N}avigationsproblem der {L}uftfahrt.
\newblock {\em Z. Angew. Math. Mech.}, 11(5):373--381, 1931.

\bibitem{ventura2022a}
F.~Ventura, P.~Catry, M.~P. Dias, G.~A. Breed, A.~Folch, and J.~P. Granadeiro.
\newblock A central place foraging seabird flies at right angles to the wind to
  jointly optimize locomotor and olfactory search efficiency.
\newblock {\em Proceedings of the Royal Society B: Biological Sciences},
  289(1981):20220895, 2022.

\bibitem{ventura2020a}
F.~Ventura, J.~P. Granadeiro, O.~Padget, and P.~Catry.
\newblock Gadfly petrels use knowledge of the windscape, not memorized foraging
  patches, to optimize foraging trips on ocean-wide scales.
\newblock {\em Proceedings of the Royal Society B: Biological Sciences},
  287(1918):20191775, 2020.

\bibitem{yajima2009}
T.~Yajima and H.~Nagahama.
\newblock Finsler geometry of seismic ray path in anisotropic media.
\newblock {\em Proc. R. Soc. A}, 469:1763--1777, 2009.

\bibitem{Z1931}
E.~Zermelo.
\newblock {\"U}ber das {N}avigationsproblem bei ruhender oder
  ver{\"a}nderlicher {W}indverteilung.
\newblock {\em Z. Angew. Math. Mech.}, 11:114--124, 1931.

\end{thebibliography}

\clearpage
\begin{appendix}
    \section{Convergence results for GEORCE-H}
    We generalize the proofs in \citep{georce} for time-dependent Finsler metrics to show global and quadratic convergence. For notation, let $E$ denote the discretized squared travel time functional in \eqref{eq:disc_energy}, and let $\langle \cdot,\cdot \rangle$, $||\cdot||$ and $\nabla$ be the standard Euclidean metric, norm and gradient, respectively. 
    \subsection{Assumptions} \label{ap:assumptions}
For completeness, we restate the assumptions in \citep{georce} that we will use for the proof of local quadratic convergence in Appendix~\ref{ap:quadratic_convergence} below.
\begin{assumption} \label{assum:quad_conv_assumptions}
    We assume the following regarding the discretized squared travel time functional:
    \begin{itemize}
        \item[(a)] Any (local) minimum point $z^{*}=(t^{*},x^{*},v^{*})$ of the discretized squared travel time functional will be assumed to be a {\em strongly unique minimum point}, meaning that
        \begin{equation*}
            \exists \epsilon>0: \quad \exists K>0:\quad \forall z \in B_{\epsilon}(z^{*}): \quad E(z)-E(z^{*}) \geq K\norm{z-z^{*}},
        \end{equation*}
        where $B_{\epsilon}(z^{*})=\left\{z: \norm{z- z^{*}}<\epsilon\right\}$.
        \item[(b)] We assume that the discretized squared travel time functional $E(z)$ is locally Lipschitz.
        \item[(c)] The first-order Taylor approximation of the discretized squared travel time functional is
        \begin{equation*}
            \Delta E = \langle \nabla E(z_{0}), \Delta z \rangle + \mathcal{O}\left(\Delta z\right)\norm{\Delta z},
        \end{equation*}
        i.e., we assume that $\mathcal{O}\left(\Delta z\right)$ is independent of $z_{0}$, and the term $\mathcal{O}\left(\Delta z\right)\norm{\Delta z}$ can be rewritten as $\mathcal{\tilde{O}}\left(\norm{\Delta z}^{2}\right)$ locally (which follows from the Lipschitz condition).
    \end{itemize}
\end{assumption}

Note that, since the metric varies smoothly on the manifold, then the discretized squared travel time functional is also smooth and hence Lipschitz. Thus, all Finsler and Riemannian manifolds satisfy the above conditions, but an even larger class of squared travel time functionals that does not necessarily stems from a Riemannian or Finsler metric will satisfy the requirements.
    \subsection{Necessary Conditions} \label{ap:necessary_cond}
    \begin{proposition}
    Assume the following first-order Taylor approximation of $F_{t}$ for the state equation in time:
    \begin{equation*}
        F(t,x,v) = L^{(0)}+L^{(v)}v+L^{(x)}x+L^{(t)}t,
    \end{equation*}
    where $L^{(v)} \coloneqq \frac{\partial F}{\partial v}\left(t^{(0)},x^{(0)},v^{(0)}\right)$, $L_{s}^{(x)} \coloneqq \frac{\partial F}{\partial x}\left(t^{(0)},x^{(0)},v^{(0)}\right)$, $L^{(t)} \coloneqq \frac{\partial F}{\partial t}\left(t^{(0)},x^{(0)},v^{(0)}\right)$ and $L^{(0)} \coloneqq F\left(t^{(0)},x^{(0)},v^{(0)}\right)$ for some point $\left(t^{(0)},x^{(0)},v^{(0)}\right)$. The necessary conditions for a minimum in \eqref{eq:disc_energy} with the first-order Taylor approximation are then
    \begin{equation*}
        \begin{split}
            &2G(t_{s},x_{s},v_{s})v_{s}+\restr{\nabla_{y}\left[v_{s}^{\top}G(t_{s},x_{s},y)v_{s}\right]}{y=v_{s}}+\mu_{s}+\pi_{s}L_{s}^{(v)}=0, \quad s=0,\dots,T-1, \\
            &v_{s}^{\top}\frac{\partial G\left(t_{s}^{(i)},x_{s}^{(i)},v_{s}^{(i)}\right)}{\partial t}v_{s}+\pi_{s}L_{s}^{(t)}+\pi_{s}=\pi_{s-1}, \quad s=1,\dots,T-1, \\
            &0=\pi_{T-1}, \\
            &\restr{\nabla_{y}\left[v_{s}^{\top}G(t_{s},y,v_{s})v_{s}\right]}{y=x_{s}}+\pi_{s}L_{s}^{(x)}+\mu_{s}=\mu_{s-1}, \quad s=1,\dots,T-1, \\
            &x_{s+1}=x_{s}+v_{s}, \quad s=0,\dots,T-1, \\
            &t_{s+1}=t_{s}+F(t_{s},x_{s},v_{s}),\quad s=0,\dots,T-1, \\
            &\sum_{s=0}^{T-1}v_{s}=B-A, \\
            &t_{0}=0,x_{0}=A.
        \end{split}
    \end{equation*}
    where $\mu_{s} \in \mathbb{R}^{n}$ for $s=0,\dots,T-1$.
\end{proposition}
\begin{proof}
    Using the first-order approximation of the Finsler metric for state equation in time, the approximated Hamiltonian is
    \begin{equation*} 
        \tilde{H}(t_{s},x_{s},v_{s}) = v_{s}^{\top}G_{s}v_{s}+\mu_{s}^{\top}(x_{s}+v_{s})+\pi_{s}\left(t_{s}+L^{(0)}_{s}+L^{(v)}_{s}v_{s}+L^{(x)}_{s}x_{s}+L^{(t)}_{s}t_s\right).
    \end{equation*}
    Since $G$ is symmetric and positive definite, then the approximation of the Hamiltonian function in each iteration is convex, i.e. the stationary point with respect to $v_{s}$ is the global minimum. By a discretized version of Pontryagin's maximum principle we obtain the following necessary conditions for a local minimum:
    \begin{equation*}
        \begin{split}
            &2G(t_{s},x_{s},v_{s})v_{s}+\restr{\nabla_{y}\left[v_{s}^{\top}G(t_{s},x_{s},y)v_{s}\right]}{y=v_{s}}+\mu_{s}+\pi_{s}L_{s}^{(v)}=0, \quad s=0,\dots,T-1, \\
            &v_{s}^{\top}\frac{\partial G\left(t_{s}^{(i)},x_{s}^{(i)},v_{s}^{(i)}\right)}{\partial t}v_{s}+\pi_{s}L_{s}^{(t)}+\pi_{s}=\pi_{s-1}, \quad s=1,\dots,T-1, \\
            &0=\pi_{T-1}, \\
            &\restr{\nabla_{y}\left[v_{s}^{\top}G(t_{s},y,v_{s})v_{s}\right]}{y=x_{s}}+\pi_{s}L_{s}^{(x)}+\mu_{s}=\mu_{s-1}, \quad s=1,\dots,T-1, \\
            &x_{s+1}=x_{s}+v_{s}, \quad s=0,\dots,T-1, \\
            &t_{s+1}=t_{s}+F(t_{s},x_{s},v_{s}),\quad s=0,\dots,T-1, \\
            &\sum_{s=0}^{T-1}v_{s}=B-A, \\
            &t_{0}=0,x_{0}=A.
        \end{split}
    \end{equation*}
\end{proof}
    \subsection{Global Convergence} \label{ap:global_convergence}
We follow the same approach as in \citep{georce} to prove global convergence of \textit{GEORCE-H}. For completeness, we first restate a result used in the proof of global convergence in the original \textit{GEORCE}-paper almost ad verbatim \citep{georce}, but adapt it to time-dependent Finsler metrics.

\begin{lemma}\label{lemma:global_conv_feasible}
    Assume that $(t^{(i)}, x^{(i)}, v^{(i)})$ is a feasible solution (i.e., it has the correct start and end points). Then, the following properties hold:
    \begin{enumerate}
        \item There exists a unique solution $\left(t^{(i+1)}, x^{(i+1)}, v^{(i+1)}\right)$ to the system of equations in \eqref{eq:georceh_equations} based on $\left(t^{(i)}, x^{(i)}, v^{(i)}\right)$.
        \item All linear combinations $(1-\alpha)\left(x^{(i)},v^{(i)}\right)+\alpha\left(x^{(i+1)},v^{(i+1)}\right)$ for $0 < \alpha \leq 1$ are feasible solutions.
    \end{enumerate}
\end{lemma}
\begin{proof}
    Since the matrices $\left\{G\left(t_{s}^{(i)}, x_{s}^{(i)}, v_{s}^{(i)}\right)\right\}_{s=0}^{T}$ are positive definite, then they are regular with unique inverse, and so are the sum of the inverse matrices. This means that there exists a unique solution to the update scheme in \eqref{eq:georceh_update_scheme} proving 1.

    Since the solution $\left(x^{(i+1)}, v^{(i+1)}\right)$ is also a feasible solution for the update scheme in \eqref{eq:georceh_update_scheme}, then the linear combination of two feasible solutions will also be a feasible solution, i.e.
    \begin{equation*}
        \sum_{s=0}^{T-1}v_{s}^{(j)}=(B-A), \quad j=1,\dots,(i+1).
    \end{equation*}
    The new solution based on the linear combination gives
    \begin{equation*}
        \begin{split}
            (1-\alpha)\sum_{s=0}^{T-1}v_{s}^{(i)}+\alpha\sum_{s=0}^{T-1}v_{s}^{(i+1)} & =(1-\alpha)(B-A)+\alpha(B-A)=(B-A), \\
            (1-\alpha)x^{(i)}+\alpha x^{(i+1)} & = (1-\alpha)\left(\alpha, x_{1}^{(i)}, \dots, x_{T-1}^{(i)},B\right) +\alpha \left(A, x_{1}^{(i+1)}, \dots, x_{T-1}^{(i+1)},B\right) \\
            & = \left(A, (1-\alpha)x_{1}^{(i)}+\alpha x_{1}^{(i+1)},\dots,(1-\alpha)x_{T-1}^{(i)}+\alpha x_{T-1}^{(i+1)},B\right),
        \end{split}
    \end{equation*}
    where for any $s=0,\dots,T-1$,
    \begin{equation*}
        \begin{split}
            (1-\alpha)x_{s}^{(i)}+\alpha x_{s}^{(i+1)} &= (1-\alpha)\left(A+\sum_{j=0}^{s-1}v_{j}^{(i)}\right)+\alpha\left(A+\sum_{j=0}^{s-1}v_{j}^{(i+1)}\right) \\
            &= A+\sum_{j=0}^{s-1}\left((1-\alpha)v_{j}^{(i)}+\alpha v_{j}^{(i+1)}\right).
        \end{split}
    \end{equation*}
    This shows that the linear combination in the state is feasible as it produces consistent state variables in terms of start and end points. Furthermore, each state variable is feasible, if they were determined from the linear combination of the control vectors, which proves that the linear combinations of the state and control variables are also feasible solutions.
\end{proof}

We then generalize the following lemma to a time-dependent Finsler metric with affine state equation in time using the same approach as in \citep{georce}.

\begin{lemma}\label{lemma:global_taylor}
    Let $\left\{t_{s}^{(i)}, x_{s}^{(i)}, v_{s}^{(i)}\right\}_{s=0}^{T}$ denote the feasible solution after iteration $i$ in \textit{GEORCE-H}. If $\left\{t_{s}^{(i)}, x_{s}^{(i)}, v_{s}^{(i)}\right\}_{s=0}^{T}$ is not a (local) minimum point, then the feasible solution from iteration $\left\{t_{s}^{(i+1)}, x_{s}^{(i+1)}, v_{s}^{(i+1)}\right\}_{s=0}^{T}$ will decrease the objective function in the sense that there exists an $\eta>0$ such that for all $\alpha$ with $0<\alpha\leq\eta\leq 1$, then
    \begin{equation*}
        E\left(t^{(i+1)}(\alpha), x^{(i)}+\alpha \left(x^{(i+1)}-x^{(i)}\right), v^{(i)}+\alpha \left(v^{(i+1)}-v^{(i)}\right)\right) < E\left(t^{(i)}, x^{(i)}, v^{(i)}\right),
    \end{equation*}
    where $x^{(i)}=\left(A,x_{1}^{(i)}, \dots, x_{T-1}^{(i)}, B\right)$ and $v^{(i)} = \left(v_{0}^{(i)}, v_{1}^{(i)}, \dots, v_{T-2}^{(i)}, v_{T-1}^{(i)}\right)$.
\end{lemma}
\begin{proof}
    Since $E(t,x,v)$ is a smooth function, then the first-order Taylor approximation in $(t,x,v)$ of the discretized squared travel time functional $E(t,x,v)$ in \eqref{eq:disc_energy} is
    \begin{equation} \label{eq:georceh_taylor_proof}
        \begin{split}
            \Delta E(t,x,v) &= \sum_{s=1}^{T-1}\left(\left\langle\restr{\partial_{t_{s}}E(t_{s},x_{s},v_{s})}{(t_{s}, x_{s},v_{s})=\left(t_{s}^{(i)}, x_{s}^{(i)}, v_{s}^{(i)}\right)}, \Delta t_{s}\right\rangle +\mathcal{O}(\Delta t_{s})\norm{\Delta t_{s}}\right) \\
            &+ \sum_{s=1}^{T-1}\left(\left\langle\restr{\nabla_{x_{s}}E(t_{s},x_{s},v_{s})}{(t_{s},x_{s},v_{s})=\left(t_{s}^{(i)}, x_{s}^{(i)}, v_{s}^{(i)}\right)}, \Delta x_{s}\right\rangle +\mathcal{O}(\Delta x_{s})\norm{\Delta x_{s}}\right) \\
            &+\sum_{s=0}^{T-1}\left(\left\langle \restr{\nabla_{v_{s}}E(t_{s}, x_{s},v_{s})}{(t_{s},x_{s},v_{s})=\left(t_{s}^{(i)}, x_{s}^{(i)}, v_{s}^{(i)}\right)}, \Delta v_{s}\right\rangle+\mathcal{O}(\Delta v_{s})\norm{\Delta v_{s}}\right),
        \end{split}
    \end{equation}
    where $\Delta t_{s} \coloneqq t_{s}^{(i+1)}-t_{s}^{(i)}$, $\Delta v_{s} \coloneqq v_{s}^{(i+1)}-v_{s}^{(i)}$ and $\Delta x_{s} \coloneqq x_{s}^{(i+1)}-x_{s}^{(i)}$. For \textit{GEORCE-H}, the optimality conditions are
    \begin{equation} \label{eq:f_cond}
        \begin{split}
            \restr{\partial_{t_{s}}E(t_{s},x_{s},v_{s})}{(t_{s},x_{s},v_{s})=\left(t_{s}^{(i)},x_{s}^{(i)},v_{s}^{(i)}\right)} &= \pi_{s-1}-\pi_{s}L_{s}^{(t)}-\pi_{s}, \quad s=1,\dots,T-1, \\
            \restr{\nabla_{x_{s}}E(t_{s},x_{s},v_{s})}{(t_{s},x_{s},v_{s})=\left(t_{s}^{(i)},x_{s}^{(i)},v_{s}^{(i)}\right)} &= \mu_{s-1}-\pi_{s}L_{s}^{(x)}-\mu_{s}, \quad s=1,\dots,T-1, \\
            \restr{\nabla_{v_{s}}E(t_{s},x_{s},v_{s})}{(t_{s},x_{s},v_{s})=\left(t_{s}^{(i)},x_{s}^{(i)}, v_{s}^{(i+1)}\right)} &= 2\nu_{s}v_{s}^{(i+1)}+\zeta_{s}, \quad s=0,\dots,T-1.
        \end{split}
    \end{equation}
    Applying \eqref{eq:f_cond} to \eqref{eq:georceh_taylor_proof}, we get
    \begin{equation*}
        \begin{split}
            \Delta E(t,x,v) &= \sum_{s=1}^{T-1}\left(\langle \pi_{s-1}-\pi_{s}L_{s}^{(t)}-\pi_{s}, \Delta t_{s}\rangle + \mathcal{O}\left(\Delta t_{s}\right)\norm{\Delta t_{s}}\right) \\
            &+ \sum_{s=1}^{T-1}\left(\langle\mu_{s-1}-\pi_{s}L_{s}^{(x)}-\mu_{s}, \Delta x_{s}\rangle +\mathcal{O}(\Delta x_{s})\norm{\Delta x_{s}}\right) \\
            &+\sum_{s=0}^{T-1}\left(\left\langle \restr{\nabla_{v_{s}}E(t_{s},x_{s},v_{s})}{(t_{s},x_{s},v_{s})=\left(t_{s}^{(i)}x_{s}^{(i)},v_{s}^{(i)}\right)}, \Delta v_{s}\right\rangle+\mathcal{O}(\Delta v_{s})\norm{\Delta v_{s}}\right).
        \end{split}
    \end{equation*}
    Since $\Delta x_{s} = \sum_{j=0}^{s-1}\Delta v_{j}$, it follows that
    \begin{equation*}
        \begin{split}
            \Delta E(t,x,v) &= \sum_{s=1}^{T-1}\left(\langle \pi_{s-1}-\pi_{s}L_{s}^{(t)}-\pi_{s}, \Delta t_{s}\rangle + \mathcal{O}\left(\Delta t_{s}\right)\norm{\Delta t_{s}}\right) \\
            &+ \sum_{s=1}^{T-1}\left(\left\langle \mu_{s-1}-\pi_{s}L_{s}^{(x)}-\mu_{s}, \sum_{j=0}^{s-1}\Delta v_{j}\right\rangle +\mathcal{O}\left(\sum_{j=0}^{s-1}\Delta v_{j}\right)\norm{\sum_{j=0}^{s-1}\Delta v_{j}}\right) \\
            &+\sum_{s=0}^{T-1}\left(\left\langle \restr{\nabla_{v_{s}}E(t_{s},x_{s},v_{s})}{(t_{s},x_{s},v_{s})=\left(t_{s}^{(i)},x_{s}^{(i)},v_{s}^{(i)}\right)}, \Delta v_{s}\right\rangle+\mathcal{O}(\Delta v_{s})\norm{\Delta v_{s}}\right).
        \end{split}
    \end{equation*}
    Rearranging the terms, we get
    \begin{equation*}
        \begin{split}
            \Delta E(t,x,v) &= \sum_{s=0}^{T-2}\pi_{s}\Delta t_{s+1}-\sum_{s=1}^{T-1}\left(1+L_{s}^{(t)}\right)\pi_{s}\Delta t_{s} + \sum_{s=1}^{T-1}\mathcal{O}\left(\Delta t_{s}\right)\norm{t_{s}}+\sum_{s=0}^{T-2}\left\langle \mu_{s}, \Delta x_{s+1} \right \rangle \\
            &- \sum_{s=1}^{T-1}\left \langle \pi_{s}L_{s}^{(x)}+\mu_{s}, \Delta x_{s} \right\rangle + \sum_{s=1}^{T-1}\mathcal{O}\left(\Delta x_{s}\right)\norm{\Delta x_{s}} \\
            &+\sum_{s=0}^{T-1}\left(\left\langle \nabla_{v} \restr{E(t,x,v)}{(t,x,v)=\left(t_{s}^{(i)},x_{s}^{(i)},v_{s}^{(i)}\right)}, \Delta v_{s}\right\rangle + \mathcal{O}\left(\Delta v_{s}\right)\norm{\Delta v_{s}}\right) \\
            &= \pi_{0}\Delta t_{1} + \sum_{s=1}^{T-2}\pi_{s}\left(\Delta t_{s+1}-\Delta t_{s}\right)+\pi_{T-1}\Delta t_{T-1}-\sum_{s=1}^{T-1}L_{s}^{(t)}\pi_{s}\Delta t_{s} \\
            &+ \sum_{s=1}^{T-1}\mathcal{O}\left(\Delta t_{s}\right)\norm{t_{s}}+\langle u_{0}, \Delta x_{1} \rangle + \sum_{s=1}^{T-1}\left\langle \mu_{s}, \Delta x_{s+1}-\Delta x_{s} \right\rangle \\
            &- \left\langle \mu_{T-1}, \Delta x_{T-1} \right \rangle - \sum_{s=1}^{T-1}\left\langle \pi_{s}L_{s}^{(x)}, \Delta x_{s} \right\rangle+\sum_{s=1}^{T-1}\mathcal{O}\left(\Delta x_{s}\right)\norm{\Delta x_{s}} \\
            &+ \sum_{s=0}^{T-1} \left(\left\langle\nabla_{v}\restr{E(t,x,v)}{(t,x,v)=\left(t_{s}^{(i)},x_{s}^{(i)},v_{s}^{(i)}\right)}, \Delta v_{s} \right\rangle + \mathcal{O}\left(\Delta v_{s}\right)\norm{\Delta v_{s}}\right).
        \end{split}
    \end{equation*}
    We observe that
    \begin{equation*}
        \begin{split}
            \Delta t_{s+1}-\Delta t_{s} &= \left(t_{s+1}^{(i+1)}-t_{s+1}^{(i)}\right)-\left(t_{s}^{(i+1)}-t_{s}^{(i)}\right) = \left(t_{s+1}^{(i+1)}-t_{s}^{(i)}\right)-\left(t_{s+1}^{(i+1)}-t_{s}^{(i)}\right) \\
            &= E\left(t_{s}^{(i+1)}, x_{s}^{(i+1)}, v_{s}^{(i+1)}\right)-E\left(t_{s}^{(i)}, x_{s}^{(i)}, v_{s}^{(i)}\right) \\
            &= L_{s}^{(t)}\Delta t_{s} + L_{s}^{(x)}\Delta x_{s} + L_{s}^{(v)}\Delta v_{s} + \mathcal{O}\left(\Delta E\right)\norm{\Delta E}.
        \end{split}
    \end{equation*}
    Similarly, we see that
    \begin{equation*}
        \begin{split}
            \Delta x_{s+1}-\Delta x_{s} &= \left(x_{s+1}^{(i+1)}-x_{s+1}^{(i)}\right)-\left(x_{s}^{(i+1)}-x_{s}^{(i)}\right) = \left(x_{s+1}^{(i+1)}-x_{s}^{(i)}\right)-\left(x_{s+1}^{(i+1)}-x_{s}^{(i)}\right) \\
            &= v_{s}^{(i+1)}-v_{s}^{(i)} = \Delta v_{s},
        \end{split}
    \end{equation*}
    which implies that
    \begin{equation*}
        \left\langle \mu_{T-1}, \Delta x_{T-1} \right\rangle = \left\langle \mu_{T-1}, 0-\sum_{s=0}^{T-2}\Delta v_{s}\right\rangle
        = \left\langle \mu_{T-1}, \sum_{s=0}^{T-1}\Delta v_{s}-\sum_{s=0}^{T-2}\Delta v_{s}\right\rangle = \left\langle \mu_{T-1}, \Delta \mu_{T-1}\right\rangle.
    \end{equation*}
    The latter is due to the fact that the solution in iteration $i$ and $i+1$ are both feasible, i.e. $\sum_{s=0}^{T-1}v_{s}=B-A$, which means that $\sum_{s=0}^{T-1}\Delta v_{s}=0$. Further note that $\Delta t_{1}=\Delta t_{1}-\Delta t_{0}$ as $\Delta t_{0}=0$, since the starting point is fixed. Similarly, $\Delta x_{1}= \Delta x_{1}-\Delta x_{0}$, since $\Delta x_{0}=0$, and $\pi_{T-1}=0$. This implies that
    \begin{equation*}
        \begin{split}
            \Delta E(t,x,v) &= \sum_{s=0}^{T-1}\pi_{s}\left(L_{s}^{(t)}\Delta t_{s} + L_{s}^{(x)}\Delta x_{s} + L_{s}^{(v)}\Delta v_{s}+\mathcal{O}\left(\Delta E\right)\norm{\Delta E}\right) \\
            &- \sum_{s=1}^{T-1}L_{s}^{(t)}\pi_{s}\Delta t_{s} + \sum_{s=1}^{T-1}\mathcal{O}\left(\Delta t_{s}\right)\norm{\Delta t_{s}}+\sum_{s=0}^{T-1}\left(\left\langle \mu_{s}, \Delta v_{s}\right\rangle + \mathcal{O}\left(\Delta x_{s}\right)\norm{\Delta x_{s}}\right) - \sum_{s=0}^{T}\left\langle \pi_{s}L_{s}^{(x)}, \Delta x_{s}\right\rangle \\
            &+ \sum_{s=0}^{T-1}\left(\left\langle \nabla_{v}\restr{E(t,x,v)}{(t,x,v)=\left(t_{s}^{(i+1)},x_{s}^{(i+1)},v_{s}^{(i+1)}\right)}, \Delta v_{s}\right\rangle+\mathcal{O}\left(\Delta x_{s}\right)\norm{\Delta x_{s}}+\pi_{s}\mathcal{O}\left(\Delta E\right)\right) \\
            &+ \sum_{s=0}^{T-1}\mathcal{O}\left(\Delta v_{s}\right)\norm{v_{s}}.
        \end{split}
    \end{equation*}
    Since for $s=0,\dots,T-1$,
    \begin{equation*}
        \begin{split}
            &\pi_{s}L_{s}^{(v)}+\mu_{s}+\nabla_{v}\restr{E(t,x,v)}{(t,x,v)=\left(t_{s}^{(i)},x_{s}^{(i)}, v_{s}^{(i)}\right)} \\
            &= -\nabla_{v}\restr{E(t,x,v)}{(t,x,v)=\left(t_{s}^{(i)},x_{s}^{(i)}, v_{s}^{(i+1)}\right)}+\nabla_{v}\restr{E(t,x,v)}{(t,x,v)=\left(t_{s}^{(i)},x_{s}^{(i)}, v_{s}^{(i)}\right)} \\
            &= -2G\left(t_{s}^{(i)}, x_{s}^{(i)},v_{s}^{(i)}\right)-\zeta_{s}+2G\left(t_{s}^{(i)},x_{s}^{(i)},v_{s}^{(i)}\right)v_{s}^{(i+1)}+\zeta_{s} \\
            &= -2G\left(t_{s}^{(i)}, x_{s}^{(i)}, v_{s}^{(i)}\right)\left(v_{s}^{(i+1)}-v_{s}^{(i)}\right),
        \end{split}
    \end{equation*}
    then, the first-order Taylor approximation of the discretized squared travel time functional becomes
    \begin{equation*}
        \Delta E(t,x,v) = \sum_{s=0}^{T-1}\left(\left\langle -2G\left(t_{s}^{(i)},x_{s}^{(i)}, v_{s}^{(i)}\right)\Delta v_{s}, \Delta v_{s}\right\rangle + \mathcal{O}\left(\Delta t_{s}\right)\norm{\Delta t_{s}}+\mathcal{O}\left(\Delta x_{s}\right)\norm{\Delta x_{s}}+\mathcal{O}\left(\Delta v_{s}\right)\norm{\Delta v_{s}}\right).
    \end{equation*}
    Since $G\left(t_{s}^{(i)},x_{s}^{(i)},v_{s}^{(i)}\right)$ is positive definite for $s=0,\dots,T-1$, the first summation is negative assuming at least one non-zero $\{\Delta v_{s}\}_{s=0}^{T-1}$, which is the case, since otherwise $\left\{t_{s}^{(i)},x_{s}^{(i)},v_{s}^{(i)}\right\}_{s=0}^{T-1}$ would be a local minimum.

    Note that $\left\{\norm{\Delta v_{s}}\right\}_{s=0}^{T-1}$ converging to $0$ implies that $\norm{\Delta x_{s}}$ also converges to $0$ for $s=0,\dots,T-1$. Since $t_{s+1}=t_{s}+E(t_{s},x_{s},v_{s})$ and $E$ is smooth, then $\Delta t_{s}$ will also converge to $0$ for $s=0,\dots,T-1$. The $\mathcal{O}$-terms can therefore be simplified to
    \begin{equation} \label{eq:energy_taylor}
        \Delta E(t,x,v) = \sum_{s=0}^{T-1}\left(\left\langle -2G\left(t_{s}^{(i)},x_{s}^{(i)}, v_{s}^{(i)}\right)\Delta v_{s}, \Delta v_{s}\right\rangle + \mathcal{O}\left(\Delta v_{s}\right)\norm{\Delta v_{s}}\right).
    \end{equation}
    Now scale all $\{\Delta v_{s}\}_{s=0}^{T-1}$ with a scalar $ \alpha \leq 1$. Note that the term $ -2G\left(t_{s}^{(i)},x_{s}^{(i)}, v_{s}^{(i)}\right)\Delta v_{s}$ is unaffected by this, since this term is equal to $\pi_{s}L_{s}^{(v)}+\mu_{s}+\nabla_{v}\restr{E(t,x,v)}{(t,x,v)=\left(t_{s}^{(i)},x_{s}^{(i)},v_{s}^{(i)}\right)}$, which is independent of $\alpha$. The first-order Taylor expansion is then
    \begin{equation*}
        \Delta E(t,x,v) = \sum_{s=0}^{T-1}\left(\left\langle -2G\left(t_{s}^{(i)},x_{s}^{(i)}, v_{s}^{(i)}\right)\Delta v_{s}, \alpha \Delta v_{s}\right\rangle + \mathcal{O}\left(\alpha\Delta v_{s}\right)\norm{\alpha\Delta v_{s}}\right).
    \end{equation*}
    In the limit it follows that there exists a scalar $\eta>0$ such that $\frac{E(t,x,v)}{\alpha}<0$ for some $\alpha$ with $0<\alpha \leq \eta \leq 1$, i.e. $\Delta E(t,x,v)<0$.
\end{proof}

With the generalized version of Lemma~\ref{lemma:global_taylor} for a time-dependent Finsler metric, we prove the main theorem equivalent to the original theorem for global convergence in \citep{georce}, which follows the version in \citep{georce} almost ad verbatim adapted to a time-dependent Finsler metric.

\begin{theorem}
    Let $E^{(i)}$ be the value of the discretized squared travel time functional for the solution after iteration $i$ (with line search) in \textit{GEORCE-H}. If the starting point $\left(t^{(0)}, x^{(0)}, v^{(0)}\right)$ is feasible, then the series $\left\{E^{(i)}\right\}_{i>0}$ will converge to a (local) minimum. 
\end{theorem}
\begin{proof}
    If the starting point $\left(t^{(0)}, x^{(0)}, v^{(0)}\right)$ in \textit{GEORCE-H} is feasible, then the points $\left(t^{(i)}, x^{(i)}, v^{(i)}\right)$ after each iteration (with line search) in \textit{GEORCE-H} will also be feasible solutions by Lemma~\ref{lemma:global_conv_feasible}.

    The discretized squared travel time functional is a positive function, i.e. a lower bound is at least 0. Assume that the series $\left\{E^{(i)}\right\}_{i>0}$ is not converging to a (local) minimum. From Lemma~\ref{lemma:global_taylor}, it follows that $\left\{E^{(i)}\right\}_{i>0}$ is decreasing for increasing $i$, and since $E^{(i)}$ has a lower bound, the series $\left\{E^{(i)}\right\}_{i}$ cannot be non-converging. Assume now that the convergence value for $\left\{E^{(i)}\right\}_{i}$, denoted $\hat{E}$, is not a (local) minimum, and denote the convergence point $(\hat{t}, \hat{x}, \hat{v})$. According to Lemma~\ref{lemma:global_taylor}, \textit{GEORCE-H} will from the solution $(\hat{t}, \hat{x}, \hat{v})$ in the following iteration produce a new point $(\tilde{t}, \tilde{x},\tilde{v})$ such that $E(\tilde{t}, \tilde{x},\tilde{v})<\hat{E}$, which is a contradiction as the series $\left\{E^{(i)}\right\}_{i}$ is assumed to be converging to $\hat{E}$. The series, $\left\{E^{(i)}\right\}_{i}$, will therefore converge, and the convergence point is a (local) minimum.
\end{proof}
    \subsection{Local Quadratic Convergence} \label{ap:quadratic_convergence}
    In this section we generalize the proof of local quadratic convergence in \citep{georce}, showing that \textit{GEORCE-H} has locally quadratic convergence to a (local) minimum under certain assumptions satisfied by the discretized squared travel time functional. We will work under the assumptions stated in Appendix~\ref{ap:assumptions}. We first recall the following result from \citep{georce}.
\begin{lemma}\label{lemma:quad_conv_bound}
    Assume $f(x)$ is a smooth function. Let $z^{*}$ be a strongly unique (local) minimum point for $f$. Then:
    \begin{equation*}
        \exists \epsilon>0:\quad \forall z \in B_{\epsilon}(z^{*})\setminus\{z^{*}\}:\quad \langle \nabla f(z), z^{*}-z \rangle < 0. 
    \end{equation*}
\end{lemma}
We now generalize the following proposition from \citep{georce} to time-dependent Finsler metrics, following the same approach as in \citep{georce} almost ad verbatim, showing local quadratic convergence.
\begin{proposition}
    If the discretized squared travel time functional has a strongly unique (local) minimum point $z^{*}$ and locally $\alpha^{*}=1$, i.e. no line search, then \textit{GEORCE-H} has local quadratic convergence, i.e.
    \begin{equation*}
        \exists \epsilon>0:\quad \exists c>0: \quad \forall z^{(i)} \in B_{\epsilon}(z^{*}): \quad \norm{z^{(i+1)}-z^{*}} \leq c \norm{z^{(i)}-z^{*}}^{2},
    \end{equation*}
    where $z^{(i)}=(t^{(i)},x^{(i)}, v^{(i)})$ is the solution from \textit{GEORCE-H} at iteration $i$.
\end{proposition}
\begin{proof}
    Assume that $z^{*}$ is a (local) minimum point. Since \textit{GEORCE-H} has global convergence to a (local) minimum point by Theorem~\ref{prop:global_convergence}, we will assume that the algorithm has a convergence point $F(z^{*})$. Consider iteration $i$ with $z^{(i)} = (t^{(i)},x^{(i)},v^{(i)})$ from \textit{GEORCE-H}, and assume that all the following iterations in \textit{GEORCE-H} are in the ball $B_{\epsilon_{1}}(z^{*})$, and fulfill Assumption~\ref{assum:quad_conv_assumptions}(a). Note that as the series $\left\{z^{(i)}\right\}$ converges to $z^{*}$, then $\left\{z^{(i)}\right\}_{i \geq k} \subset B_{\epsilon_{1}}(z^{*})$ for some $k>0$. It then follows from Assumption~\ref{assum:quad_conv_assumptions}(a) that for some $K>0$:
    \begin{equation*}
        E\left(z^{(i+1)}\right)-E\left(z^{(i)}\right) \geq K \norm{z^{(i+1)}-z^{*}},
    \end{equation*}
    since $z^{(i+1)} \in B_{\epsilon_{1}}(z^{*})$. The left-hand term above can be rearranged into
    \begin{equation*}
        E\left(z^{(i+1)}\right)-E\left(z^{*}\right) = \left(E\left(z^{(i+1)}\right)-E\left(z^{(i)}\right)\right)+\left(E\left(z^{(i)}\right)-E\left(z^{*}\right)\right).
    \end{equation*}
    By \eqref{eq:energy_taylor} in the proof of Lemma~\ref{lemma:global_taylor}, and using $\alpha^{*}=1$, it follows that
    \begin{equation*}
        \begin{split}
            &E\left(z^{(i+1)}\right)-E\left(z^{*}\right) \\
            &= \left(E\left(z^{(i+1)}\right)-E\left(z^{(i)}\right)\right)+\left(E\left(z^{(i)}\right)-E\left(z^{*}\right)\right) \\
            &= \sum_{s=0}^{T-1}\left(\left\langle -2G\left(t_{s}^{(i)}, x_{s}^{(i)}, v_{s}^{(i)}\right)\Delta v_{s}, \Delta v_{s} \right\rangle + \mathcal{O}\left(\Delta v_{s}\right)\norm{\Delta v_{s}}\right) \\
            &- \sum_{s=0}^{T-1}\left(\left\langle -2G\left(t_{s}^{(i)}, x_{s}^{(i)}, v_{s}^{(i)}\right)\Delta v_{s}, \left(v_{s}^{*}-v_{s}^{(i)}\right) \right\rangle + \mathcal{O}\left(v_{s}^{*}-v_{s}^{(i)}\right)\norm{v_{s}^{*}-v_{s}^{(i)}}\right) \\
            &= \sum_{s=0}^{T-1}\left(\left\langle -2G\left(t_{s}^{(i)}, x_{s}^{(i)}, v_{s}^{(i)}\right)\Delta v_{s}, \Delta v_{s}-\left(v_{s}^{*}-v_{s}^{(i)}\right) \right\rangle + \mathcal{O}\left(\Delta v_{s}\right)\norm{\Delta v_{s}}+\mathcal{O}\left(v_{s}^{*}-v_{s}^{(i)}\right)\norm{v_{s}^{*}-v_{s}^{(i)}}\right) \\
            &= \sum_{s=0}^{T-1}\left(\left\langle -2G\left(t_{s}^{(i)}, x_{s}^{(i)}, v_{s}^{(i)}\right)\Delta v_{s}, v_{s}^{*}-v_{s}^{(i+1)} \right\rangle + \mathcal{O}\left(\Delta v_{s}\right)\norm{\Delta v_{s}}+\mathcal{O}\left(v_{s}^{*}-v_{s}^{(i)}\right)\norm{v_{s}^{*}-v_{s}^{(i)}}\right). \\
        \end{split}
    \end{equation*}
    Observe now that the term $2G\left(t_{s}^{(i)}, x_{s}^{(i)}, v_{s}^{(i)}\right)\Delta v_{s}$ is the gradient for $z_{s}^{(i)}$. Define $\epsilon_{2}$ such that the $\mathcal{O}$-terms are sufficiently close to zero for all $z \in B_{\epsilon_{2}}(z^{*})$. Thus, for $\epsilon=\min\left\{\epsilon_{1}, \epsilon_{2}\right\}$ we have from Lemma~\ref{lemma:quad_conv_bound} that
    \begin{equation*}
        \sum_{s=0}^{T-1}\left\langle 2G\left(t_{s}^{(i)},x_{s}^{(i)},v_{s}^{(i)}\right)\Delta v_{s}, v_{s}^{*}-v_{s}^{(i+1)}\right\rangle < 0.
    \end{equation*}
    Combining the inequalities and utilizing that $E(z)$ is locally Lipschitz continuous by Assumption~\ref{assum:quad_conv_assumptions}(b), then
    \begin{equation*}
        K\norm{z^{(i+1)}-z^{(i)}} \leq E\left(z^{(i+1)}\right)-E\left(z^{*}\right) \leq \mathcal{O}\left(\norm{z^{(i)}-z^{*}}^{2}\right),
    \end{equation*}
    which implies that there exists an $\epsilon>0$ such that
    \begin{equation*}
        \forall z^{(i)} \in B_{\epsilon}\left(z^{*}\right): \quad \norm{z^{(i+1)}-z^{(i)}} \leq c\norm{z^{(i)}-z^{*}}^{2},
    \end{equation*}
    for $c>0$.
\end{proof}
The assumption in the above proof that the line search parameter satisfies $\alpha^{*}=1$ locally at the minimum point is based on the fact that, locally, the approximation of the discretized squared travel time functional becomes close to the true function.

    \section{Experiments} \label{ap:experiments}
    The constraints used in the numerical experiments presented in the article are described below. The code is available at \url{https://github.com/FrederikMR/zermelo_tacking}.
    \subsection{Hardware} \label{ap:hardware}
    All plots and figures in Section~\ref{sec:examples} have been computed on a \textit{HP} computer with Intel Core i9-11950H 2.6 GHz 8C, 15.6'' FHD, 720P CAM, 32 GB (2$\times$16GB) DDR4 3200 So-Dimm, Nvidia Quadro TI2004GB Discrete Graphics, 1TB PCle NVMe SSD, backlit Keyboard, 150W PSU, 8cell, W11Home 64 Advanced, 4YR Onsite NBD.

The tacking curves and pregeodesics have been computed on a GPU for at most 24 hours with a maximum memory of $10$ GB. The GPU consists of $4$ nodes on a \textit{Tesla A100 PCIE}.
    \subsection{Hyper-Parameters} \label{ap:hyper_parameters}
For \textit{GEORCE-H} we set the backtracking parameter $\rho=0.5$ similar to \cite{georce} and tolerance $0.0001$. Each pregeodesic is computed with $T=1,000$ grid points and a maximum of $10,000$ iterations, except for the position-only dependent bi-metric, where we use $T=100$ to reduce runtime.

For computing the tack curves, we use Algorithm~\ref{al:tacking_optimization} with $10,000$ iterations with a tolerance of $0.0001$ for $10$ sub-iterations using \textit{GEORCE-H} in Algorithm~\ref{al:georceh}. We use the \textit{ADAM}-optimizer \citep{kingma2017adammethodstochasticoptimization} with learning rate $0.01$ to optimize the tack points.
\end{appendix}

\end{document}